\documentclass[reqno]{amsart}
\usepackage[utf8x]{inputenc}
\usepackage{amsfonts,amsmath,amssymb,mathtools,enumitem,bm}
\usepackage{datetime}
\usepackage{hyperref}
\hypersetup{colorlinks=true, pdfstartview=FitV, linkcolor=BrickRed,citecolor=BrickRed, urlcolor=BrickRed, linktoc=page}

\usepackage{graphicx}
\usepackage[dvipsnames]{xcolor}
\usepackage{caption}
\usepackage{subcaption}

\definecolor{labelkey}{rgb}{0.6,0,0}
\usepackage{cancel}

\mathtoolsset{showonlyrefs=true}

\usepackage[margin=2.5cm]{geometry}
\numberwithin{equation}{section}

\def\eps{\varepsilon}

\newcommand{\Z}{\mathbb{Z}}

\newcommand{\N}{\mathbb{N}}
\newcommand{\R}{\mathbb{R}}
\newcommand{\T}{\mathbb{T}}

\newcommand{\Sz}{\abs{\mathcal{S}}}
\newcommand{\W}{\widetilde{\mathcal{W}}}
\newcommand{\id}{\textnormal{Id}}
\newcommand{\I}{\mathcal{I}}

\newcommand{\B}{\mathcal{B}}

\newcommand{\U}{\mathcal{U}}
\newcommand{\Q}{\mathcal{Q}}

\newcommand{\m}{\mathfrak{m}}

\newcommand{\Ups}{\Upsilon}
\newcommand{\h}{\textnormal{h}}

\renewcommand{\L}{\mathcal{L}}

\newcommand{\blue}[1]{#1}

\renewcommand{\div}{{\mathrm{div }}}
\newcommand{\curl}{{\mathrm{curl }}}

\providecommand{\ip}[1]{\langle#1\rangle}
\providecommand{\abs}[1]{\left\lvert#1\right\rvert}
\providecommand{\norm}[1]{\left\|#1\right\|}

\newtheorem{theorem}{Theorem}[section]
\newtheorem{lemma}[theorem]{Lemma}
\newtheorem{corollary}[theorem]{Corollary}
\newtheorem{proposition}[theorem]{Proposition}

\newtheorem{remark}[theorem]{Remark}

\begin{document}

\title{\vspace*{-1cm}Global Axisymmetric Euler Flows with Rotation}

\author{Yan Guo}
\address{Brown University, Providence, RI, USA}
\email{yan\_guo@brown.edu}

\author{Benoit Pausader}
\address{Brown University, Providence, RI, USA}
\email{benoit\_pausader@brown.edu}

\author{Klaus Widmayer}
\address{University of Zurich, Zurich, Switzerland \& University of Vienna, Vienna, Austria}
\email{klaus.widmayer@math.uzh.ch}



\begin{abstract}
We construct a class of \emph{global, dynamical} solutions to the $3d$ Euler equations near the stationary state given by uniform ``rigid body'' rotation. These solutions are axisymmetric, of Sobolev regularity, have non-vanishing swirl and scatter linearly, thanks to the dispersive effect induced by the rotation.

To establish this, we introduce a framework that builds on the symmetries of the problem and precisely captures the anisotropic, dispersive mechanism due to rotation. This enables a fine analysis of the geometry of nonlinear interactions and allows us to propagate sharp decay bounds, which is crucial for the construction of global Euler flows.
\end{abstract}

\setcounter{tocdepth}{1}
\maketitle
\vspace*{-.75cm}
\tableofcontents
\vspace*{-.75cm}

\section{Introduction}
While global regularity of solutions to the incompressible $3d$ Euler equations for $\bm{U}:\R\times\R^3\to\R^3$
\begin{equation}\label{eq:E}
 \partial_t\bm{U}+\bm{U}\cdot\nabla \bm{U}+\nabla P=0,\qquad\div(\bm{U})=0,
\end{equation}
remains an outstanding open problem, there are several examples of stationary states (see e.g.\ \cite{Cho2020,CLV2019,FB1973,Gav2019} for some nontrivial ones). A particularly simple yet relevant one is given by \emph{uniform rotation} around a fixed axis. 
In Cartesian coordinates with $\vec{e}_3$ along the axis of rotation, these ``rigid motions'' are given by $\bm{U}_{rot}=(-x_2,x_1,0)$ (with pressure $P_{rot}=(x_1^2+x_2^2)/2$).
Working with solutions that are \emph{axisymmetric} (i.e.\ invariant with respect to rotation about $\vec{e}_3$) and writing $\bm{U}=\bm{U}_{rot}+\bm{u}$, one sees that $\bm{U}$ solves \eqref{eq:E} iff the velocity field $\bm{u}:\R\times\R^3\to\R^3$ satisfies the \emph{Euler-Coriolis} equations
\begin{equation}\label{eq:EC}
\begin{cases}
 &\partial_t\bm{u}+\bm{u}\cdot\nabla \bm{u}+\vec{e}_3\times \bm{u}+\nabla p=0,\\
 &\div(\bm{u})=0.
\end{cases}
\end{equation}
As an alternative viewpoint, \eqref{eq:EC} are the incompressible, $3d$ Euler equations written in a uniformly rotating frame of reference, where the Coriolis force is given as $\vec{e}_3\times\bm{u}$. The scalar pressure $p:\R\times\R^3\to\R$ serves to maintain the incompressibility condition $\div(\bm{u})=0$ and can be recovered from $\bm{u}$ by solving the elliptic equation $\Delta p=-(\partial_1\bm{u}_2-\partial_2\bm{u}_1)-\div(\bm{u}\cdot\nabla\bm{u})$.

Our main result shows that sufficiently small and smooth initial data $\bm{u}_0$ that are axisymmetric
lead to global, unique solutions to \eqref{eq:EC}:
\begin{theorem}\label{MainThm}
There exist $N_0>0$ and a norm $Z$, finite for Schwartz data, and $\eps_0>0$ such that if $\bm{u}_0\in H^3(\mathbb{R}^3)$ is \emph{axisymmetric} and satisfies
\begin{equation}\label{eq:id-main}
 \norm{\bm{u}_0}_{Z}+\norm{\bm{u}_0}_{H^{2N_0}}\le\varepsilon<\varepsilon_0,
\end{equation}
then there exists a unique global solution $\bm{u}\in C(\R:H^{2N_0})$ of \eqref{eq:EC} with initial data $\bm{u}_0$, and thus also a global solution $\bm{U}$ for \eqref{eq:E} with initial data $\bm{U}_0=\bm{U}_{rot}+\bm{u}_0$.

Moreover, $\bm{u}(t)$ decays over time \blue{at the optimal rate}
\begin{equation}
 \norm{\bm{u}(t)}_{L^\infty}\lesssim \eps \ip{t}^{-1}
\end{equation}
and scatters linearly in $L^2$: There exists $\bm{u}_0^\infty$ such that the solution $\bm{u}_{lin}(t)$ of the linearization of \eqref{eq:EC} with initial data $\bm{u}_0^\infty$,
\begin{equation}\label{eq:lin-EC}
 \partial_t\bm{u}_{lin}+\vec{e}_3\times \bm{u}_{lin}+\nabla p=0,\quad \div(\bm{u}_{lin})=0,\quad \bm{u}_{lin}(0)=\bm{u}_0^\infty,
\end{equation}
satisfies 
\begin{equation}
 \norm{\bm{u}(t)-\bm{u}_{lin}(t)}_{L^2}\to 0,\qquad t\to\infty.
\end{equation}
\end{theorem}

We comment on a few points of immediate relevance:
\begin{enumerate}[wide,itemsep=3pt]
 \item A more precise version of Theorem \ref{MainThm} is given below in Theorem \ref{thm:EC} of Section \ref{ssec:mainresult}. In particular, the $Z$ norm in the above statement is given explicitly as a sum of $B$ and $X$ norms -- defined in \eqref{eq:defBnorm} resp.\ \eqref{eq:defXnorm} after the introduction of appropriate technical tools  -- plus regularity in terms of a scaling vector field. With this, the scattering statement can be refined and holds in a stronger topology than $L^2$ -- see Corollary \ref{cor:scatter}.

 \item We may view Theorem \ref{MainThm} as a \emph{global stability} result (in the class of axisymmetric perturbations satisfying \eqref{eq:id-main}) for uniformly rotating solutions $\bm{U}_{rot}=r\vec{e}_\theta$ in cylindrical coordinates $(r,\theta,z)$ to the incompressible $3d$ Euler equations \eqref{eq:E}. From this perspective, our result connects with the study of stability of infinite energy solutions to the $2d$ Euler equations, such as shear flows \cite{BM15,IJ2019,IJ20,MaZhaShears2020} or stratified configurations \cite{BBZD2021}, even though the stability mechanism (``phase mixing'') in these settings is different. However, to the best of our knowledge there are no such results for the Euler equations in $3d$.
 
 We point out that the particular rotating solution $\bm{U}_{rot}$ is but one example of a family of stationary states of the $3d$ Euler equations, given by $\bm{U}_f=f(r)\vec{e}_\theta$, with $f:\R^+\to\R$. The $3d$ Euler dynamics near $\bm{U}_f$ can be described as $\bm{U}=\bm{U}_f+\bm{u}$ where $\bm{u}$ satisfies
 \begin{equation*}
  \partial_t\bm{u}+\bm{u}\cdot\nabla\bm{u}+\frac{f(r)}{r}\vec{e}_z\times\bm{u}+\frac{f(r)}{r}\partial_\theta\bm{u}+r\partial_r\left(\frac{f(r)}{r}\right)(\bm{u}\cdot\vec{e}_r)\vec{e}_\theta+\nabla p=0,\quad \div(\bm{u})=0.
 \end{equation*}
 (For $f(r)=r$ and under axisymmetry this reduces to \eqref{eq:EC}.) Our result thus initiates the study of the stability of these equilibriums.

 \item Apart from smallness, localization and axisymmetry assumptions, no restrictions are put on the initial data in Theorem \ref{MainThm}. Classical theory thus only predicts the existence of \emph{local} solutions for a time span of order $\eps^{-1}$. In contrast, the \emph{global} solutions we construct can (and in general do -- see Remark \ref{rem:a-vs-swirl}) have \emph{non-vanishing swirl}. (We recall that without swirl, solutions exist globally under relatively mild assumptions, see e.g.\ \cite[Section 4.3]{MB2002}.) In this context, the crucial role of \emph{axisymmetry} is to suppress a $2d$ Euler-type dynamic in \eqref{eq:EC}. \blue{Without axial symmetry, it is unclear whether a similar stability result can hold: the $3d$ Euler equations are notoriously unstable and there are reasons to believe that even for small initial data the aforementioned $2d$ Euler dynamic (with its potential for extremely fast norm growth) would play an important role -- for more on this we refer the reader to} the discussion in Section \ref{sec:role-axisymm}.

 \item It is remarkable that a uniform rotation keeps solutions from Theorem \ref{MainThm} {\it globally regular} in the absence of dissipation. Without rotation, even axisymmetric initial data may lead to finite time blow-up, as conjectured in \cite{Hou2021,Hou2014,Luo2014} and recently established in \cite{elgindiBU,EGM2019} for $C^{1,\alpha}$ solutions. For related equations, one can produce finite time blow-up even in the presence of rotation, e.g.\ in the inviscid primitive equations \cite{ILT2020}.
 
 At the heart of this result is a dispersive effect due to rotation. This is a linear mechanism that on $\R^3$ leads to amplitude decay of solutions of the linearization \eqref{eq:lin-EC} of the Euler-Coriolis system. The anisotropy of the problem is reflected in the dispersion relation, which is degenerate and yields a critical decay rate of at most $t^{-1}$ (see Corollary \ref{cor:decay}). In particular, our nonlinear solutions decay at the same rate as linear solutions.

 \item The influence and importance of rotational effects in fluids has been documented in various contexts, in particular in the geophysical fluids literature (see e.g.\ \cite{GS2007,mcwilliamsGFD,GFD} or for the $\beta$-plane model \cite{b-plane,Gal2008,globalbeta}). In the setting of fast rotation, the (inverse) speed of rotation introduces a parameter of smallness that can be used to prolong the time of existence of solutions. For Euler-Coriolis \eqref{eq:EC}, this has been done in \cite{AF2018,CDGG2002a,Dut2005,JW2020,KohE,MR3488136,WC2018} via Strichartz estimates associated to the linear semigroup, based on work in the viscous setting \cite{CDGG2006,GR2009}. Such results do not require axisymmetry and apply for sufficiently smooth initial data without size restrictions. By rescaling\footnote{Note that if $\bm{u}$ solves \eqref{eq:EC} on a time interval $[0,T]$, then for $\omega>0$ we have that $\bm{u}_\omega(t,x):= \omega\bm{u}(\omega t,x)$ solves \eqref{eq:EC} with $\vec{e}_3$ replaced by $\omega\vec{e}_3$ on the time interval $[0,\omega^{-1}T]$, so that speed of rotation and size of initial data can be related.}, these results amount to a logarithmic improvement of the time scale of existence in Sobolev spaces, with a slightly stronger improvement available in Besov spaces \cite{AF2018,WC2018}.

 \item This article expands on the line of work initiated in \cite{rotE}: we \emph{globally} control the evolution of small, axisymmetric initial data and find their asymptotic behavior. We develop a framework that tracks various important anisotropic parameters and -- crucially -- introduce an angular Littlewood-Paley decomposition to propagate fractional type regularity in certain angular derivatives on the Fourier side. This is coupled with a novel, refined analysis of the linear effect due to rotation, which allows us to obtain sharp decay rates with a weak control of the unknowns, and a precise understanding of the geometry of nonlinear interactions. We refer the reader to Section \ref{sec:method} for a more detailed description of our ``method of partial symmetries''.

 \item While the techniques and ideas of this article are developed with a precise adaptation to the geometry of the Euler-Coriolis system, we believe they can be of much wider use, for instance for stratified systems (such as the Boussinesq equations of \cite{3dBouss} or \cite{Cha2020}), plasmas with magnetic background fields (e.g.\ in the Euler-Poisson or Euler-Maxwell equations \cite{GIP2016,GP2011}), or in a broader context dynamo theory in the MHD equations (see e.g.\ \cite[Section 7.9]{fitzpatrick2014plasma}). Moreover, they may open directions towards new results or improved thresholds also in the viscous setting \cite{CDGG2006,KohNS}.

\end{enumerate}

\medskip 
We give next an overview of the methodology this article proposes and how these ideas are used to overcome the challenges posed by the anisotropy, quasilinear nature and critical decay rate of \eqref{eq:EC}. 

\subsection{The method of partial symmetries}\label{sec:method}
Underlying our approach are classical techniques for small data/global regularity problems in nonlinear dispersive equations, such as vector fields \cite{Klainer85} and normal forms \cite{ShatahKG85} as unified in a spacetime resonance approach \cite{GermSTR,GMSGWW3d,GNT1} and further developed in \cite{BG2014,Den2018,DIPau,DIPP,CapWW,GIP2016,IP2013,IoPau1,IP2019,IoPu2,G2dWW,KP2011} (see also \cite{Ifrim2017,Ifrim2019a}). To initiate such an analysis, we observe that the linearization of \eqref{eq:EC} is a dispersive equation, with dispersion relation given by
\begin{equation}
 \Lambda(\xi)=\xi_3/\abs{\xi},\qquad \xi\in\R^3.
\end{equation}
This is anisotropic and degenerate, and leads to $L^\infty$ decay at the critical rate $t^{-1}$, \blue{which is also sharp} -- see also Proposition \ref{prop:decay} resp.\ Corollary \ref{cor:decay} \blue{and the discussion thereafter}. 

This anisotropy is also manifest in the full, nonlinear problem \eqref{eq:EC}, which exhibits fewer symmetries and conservation laws than the $3d$ Euler equations without rotation \eqref{eq:E}. In our setting, we only have two unbounded commuting vector fields: the rotation $\Omega$ about the axis $\vec{e}_3$, and the scaling $S$ (see Section \ref{sec:structure}). To obtain regularity in all directions, we complement them with a third vector field $\Ups$, corresponding to a derivative along the polar angle in spherical coordinates on the Fourier side. This choice ensures that $\Ups$ commutes with both $\Omega$ and $S$, but it {\it does not commute with the equation}.

Our overall strategy leans on a general approach to quasilinear dispersive problems and establishes a bootstrapping scheme as follows:
\begin{enumerate}[leftmargin=*]
 \item \emph{Choice of unknowns and formulation as dispersive problem (Section \ref{sec:structure}).} 
We parameterize the fluid velocity by two real scalar unknowns $U_{\pm}$ which diagonalize the linear system and commute with the geometric structure (Hodge decomposition and vector fields). Normalizing them properly then reveals a ``null type'' structure  in the case of axisymmetric solutions (Lemma \ref{lem:ECmult}).
  
 \item \emph{Linear decay analysis (Section \ref{sec:lin_decay}).} The key point here is to identify a \emph{weak criterion} for \emph{sharp decay} which will allow to retain \emph{optimal} pointwise decay even though the highest order energies increase slowly over time. This criterion largely determines the norm we will propagate in the bootstrap; it incorporates localized control of vector fields and angular derivatives in direction $\Ups$ via a $B$ and $X$ norm, respectively.

 \item \emph{Nonlinear Analysis 1: energy and refined estimates for vector fields (Section \ref{sec:Bnorm}).} Thanks to the commutation of $S$ with the equation, energy estimates for (arbitrary) powers of $S$ on the unknowns $U_\pm$ follow directly from the decay at rate $t^{-1}$. We then upgrade these $L^2$ bounds of many vector fields to refined, uniform bounds for fewer vector fields on the profiles $\U_\pm$ of $U_\pm$ in a norm $B$. This norm is designed as a relaxation of the requirement that the Fourier transform of the profiles $\U_\pm$ be in $L^\infty$.

 \item \emph{Nonlinear Analysis 2: propagation of regularity in $\Ups$ (Section \ref{sec:Xnorm}).} This is the most delicate part of the arguments, and the design of the $X$ norm to capture the angular regularity in $\Ups$ plays a key role: roughly speaking, while stronger norms give easier access to decay, they are also harder to bound along the nonlinear evolution. In the balance struck here the $X$ norm corresponds to a fractional, angular regularity on the Fourier transforms of the profiles $\U_\pm$.
\end{enumerate}

We highlight some key aspects of our novel approach:
\begin{itemize}
 \item \emph{Anisotropic localizations:} To precisely capture the degeneracy of dispersion and to be able to quantify the size of nonlinear interactions, it is important to track both horizontal and vertical components of interacting frequencies. New analytical challenges include the control of singularities due to anisotropic degeneracy (see e.g.\ Proposition \ref{prop:decay} or Lemma \ref{lem:vfsizes-mini}). We thus work with Littlewood-Paley decompositions (with associated parameters $p,q\in\Z^-$) relative to the horizontal $\abs{\xi_\h}/\abs{\xi}$ and vertical components $\abs{\xi_3}/\abs{\xi}$  of a vector $\xi=(\xi_1,\xi_2,\xi_3)\in\R^3$, where $\xi_\h=(\xi_1,\xi_2)$.
 
 \item \emph{Angular Littlewood-Paley decomposition:}  A crucial new ingredient is the introduction of an ``angular'' Littlewood-Paley decomposition quantifying angular regularity (see Section \ref{ssec:angLP}). Since our solutions are axisymmetric, this amounts to define and control fractional powers $\Upsilon^{1+\beta}$, for $0<\beta\ll 1$. This is fundamental for our analysis in that it enables us to pinpoint a \emph{weak criterion} for \emph{sharp decay} that moreover can be controlled \emph{globally}.\footnote{While sharp decay would also follow from control of a higher power of $\Ups$  such as $\Ups^2$, the resulting terms seem to resist uniform in time bounds and are thus very hard to manage.}
 
 \item \emph{Emphasis on natural derivatives:} We view the vector fields $S,\Omega$ generated by the symmetries as the \emph{natural derivatives} of this problem, and our approach is tailored to rely on them to the largest extent possible. In particular, we develop a framework of integration by parts along these vector fields (Section \ref{sec:vfibp}). The precise quantification of this technique is achieved by combining information from the anisotropic localizations and the new  angular Littlewood-Paley decomposition. Furthermore, a remarkable interplay with the ``phases'' of the nonlinear interactions reveals a natural dichotomy on which we can base our nonlinear analysis. Compared to traditional spacetime resonance analysis, one may view this as a {\it qualified} version of the absence of spacetime resonances, relying only on the natural derivatives coming from the symmetries.
\end{itemize}

\medskip
In what follows, we describe some of our arguments in more detail.
\subsubsection*{Linear Decay}
We collect the control necessary for decay in a norm $D$ in \eqref{eq:defDnorm}, that combines the aforementioned $B$ and $X$ norms (associated with localized control of vector fields and angular derivatives in direction $\Ups$, respectively). In particular, it guarantees $L^\infty$-control of the Fourier transform. This enables a stationary phase argument \emph{adapted to the vector fields}, and yields (in Proposition \ref{prop:decay}) a novel, anisotropic dispersive decay result: we split the action of the linear semigroup of \eqref{eq:EC} on a function into two well-localized pieces (related to the angular regularity we have), which decay in $L^\infty$ resp.\ $L^2$. In addition, away from the sets of degeneracy of $\Lambda$, these pieces display decay at a faster rate. To quantify this accurately, our anisotropic setup makes use of the horizontal and vertical projections $P_p$, $P_{p,q}$ and associated parameters $p,q\in\Z^-$. In combination with the localization information and a null structure of nonlinear interactions, this provides a key advantage over some traditional dispersive estimates.

\subsubsection*{Choice of Norms}\label{it:normchoice}
Our norms are modeled on $L^2$ to exploit the Hilbertian structure, and play a complementary role. The $B$-norm \eqref{eq:defBnorm} weights the projections $P_{p,q}$ \emph{negatively} in $p,q$. \blue{For functions localized at unit frequencies,} this provides normal $L^2$ control of $\hat{f}$ for frequencies where dispersion yields full $t^{-3/2}$ decay (i.e.\ when $p,q\ge -1$), but strengthens to scale as $L^\infty$ control on $\hat{f}$ where the decay degenerates to the nonintegrable rate $t^{-1}$. It is primarily used to control the contribution of the region where $q\blue{<}-10$. The $X$-norm \eqref{eq:defXnorm} gives a strong control of $1+\beta$ angular derivatives in $\Ups$, quantified via the angular Littlewood-Paley decomposition $R_\ell$ of Section \ref{ssec:angLP}, $\ell\in\Z^+$. Weighting \emph{positively} in $p$ we obtain a control that degenerates to scale as the $L^\infty$ norm of $\hat{f}$ for vertical frequencies. This is used chiefly to control the region where $p\blue{<}-10$. \blue{In addition to the weighting in terms of anisotropic localization, our norms also include factors that help overcome the derivative loss due to the quasilinear nature of the equations.}

\subsubsection*{Nonlinear Analysis}
With a suitable choice of two scalar unknowns $U_+$ and $U_-$ (Section \ref{sec:structure}), the quasilinear structure of \eqref{eq:EC} reveals a ``null type'' structure (Lemma \ref{lem:ECmult}) that will be important for the estimates to come. Conjugating by the linear evolution we can reformulate \eqref{eq:EC} in terms of bilinear Duhamel formulas for two scalar profiles $\U_+,\U_-$ -- see Section \ref{ssec:profiles}. The nonlinear analysis can then be reduced to suitable bilinear estimates for the profiles in the $B$ and $X$ norms relevant for the decay. For the resulting oscillatory integrals of the form \eqref{eq:def_Qm}, we have versions of the classical tools of normal forms or integration by parts at our disposal. 

Here our anisotropic framework invokes the horizontal and vertical parameters $p,p_j$ and $q,q_j$, $j=1,2$  -- corresponding to the interacting and output frequencies -- that are adapted to capture (inter alia) the size of the nonlinear ``phase'' functions $\Phi$ and its vector field derivatives $V\Phi$ (Lemma \ref{lem:vfsizes-mini}). 
It is valuable to observe that a gap in the values of either the horizontal or vertical parameters immediately yields a robust lower bound for $S\Phi$ or $\Omega\Phi$, expressed again in terms of those parameters $p,p_j,q,q_j$, with additional singularity in $p_j$ due to the anisotropy, see \eqref{eq:def_sigma} and \eqref{eq:vflobound_0}. Moreover, we have the striking fact that if $\Phi$ is (relatively) small, then $V\Phi$ will be (relatively) large for some vector field $V\in\{S,\Omega\}$ (see Proposition \ref{prop:phasevssigma}). To take full advantage of this dichotomy, it is important to establish \emph{sharp criteria} for when integration by parts along vector fields is beneficial (Section \ref{sec:vfibp}). Here the Littlewood-Paley decomposition $R_\ell$ 
in the angular direction $\Ups$ plays a vital role, and quantifies the effect on ``cross terms'' via associated parameters $\ell_j$, $j=1,2$ (see also Lemma \ref{lem:VFcross}).

In bilinear estimates, the resulting framework for \emph{iterated} integration by parts \emph{along vector fields} then allows us to force parameters $\ell,\ell_j$ at the cost of $p,p_j$ and $q,q_j$, $j=1,2$, roughly speaking. As it is not viable to localize in all parameters at once (see also Remark \ref{rem:comm-loc}), we first decompose our profiles with respect to $R_{\ell_j}P_{p_j}$, and only later include the full $P_{p_j,q_j}$, $j=1,2$. In practice, we will then be able to first enforce that $p,p_j$ are all comparable (no ``gap in $p$'', as we call it), then that there are no size discrepancies in $q,q_j$ (no ``gap in $q$''), and either work with normal forms or use our new decay estimates for the linear semigroup (Proposition \ref{prop:decay}).

The simplest version of these arguments appears in Section \ref{sec:dtfbds}, and gives an improved decay at almost the optimal rate $t^{-\frac{3}{2}}$ for the $L^2$ norm of time derivatives of the profiles $\U_\pm$. This is a demonstration of the flexibility and power of our approach, which in this instance overcomes the criticality of the sharp $t^{-1}$ decay with relative ease. Here, when there are no gaps in $p$ nor $q$ (and integration by parts is thus not feasible), normal forms are not available due to the time derivative. However, with our novel decay analysis and its well-localized contributions (Proposition \ref{prop:decay}) we can gain additional decay in a straightforward $L^2\times L^\infty$ estimate. 

Including normal form arguments and a refined study of the delicate contributions of terms with localization in $q,q_j$, we can then show the $B$ norm bounds \eqref{eq:btstrap-concl1.1-B} -- see Section \ref{sec:Bnorm}. Finally, the control of the $X$ norm in Section \ref{sec:Xnorm} is the most challenging aspect of this article and requires a more subtle splitting of cases and an adapted version of iterated integration by parts along vector fields (as presented in Section \ref{ssec:D3ibp}).

\subsection{Plan of the Article}
After the necessary background in Section \ref{sec:structure}, in Section \ref{sec:mainresult} we introduce the functional framework (including the angular Littlewood-Paley decomposition) and present our main result in detail with an overview of its proof. This is followed by the linear dispersive analysis that gives the decay estimate (Section \ref{sec:lin_decay}).

The formalism for repeated integration by parts in the vector fields is subsequently developed in Section \ref{sec:vfibp}, and first used in Section \ref{sec:dtfbds} to establish some useful bounds for the time derivative of our unknowns in $L^2$. In Section \ref{sec:Bnorm} we recall the straightforward $L^2$ based energy estimates and prove the claimed $B$ norm bounds, while those for the $X$ norm are given in Sections \ref{sec:Xnorm}.

Appendix \ref{apdx} includes the proof of basic properties of the angular Littlewood-Paley decomposition (Appendix \ref{apdx:angLP}) and collects some useful lemmata that are used throughout the text (Appendices \ref{apdx:set_gain}--\ref{sec:symbols}).

\section{Structure of the equations}\label{sec:structure}
In this section we present our choice of dispersive unknowns and investigate the nonlinear structure of the equations \eqref{eq:EC} in these variables. Parts of this have already been developed in our previous work \cite[Section 2]{rotE}, but we include all necessary details for the convenience of the reader.

\subsection{Symmetries and vector fields}
The equations \eqref{eq:EC} exhibit the two symmetries of \emph{scaling} and \emph{rotation}
\begin{equation}
 \bm{u}_\lambda(t,x)=\lambda \bm{u}(t,\lambda^{-1}x),\quad \lambda>0,\qquad \bm{u}_\Theta(t,x)=\Theta^\intercal \bm{u}(t,\Theta x),\quad \Theta\in O(3).
\end{equation}
These are generated by the vector fields $S$ resp.\ $\Omega$, which act on 
vector fields $\bm{v}$ and functions $f$ as
\begin{equation}\label{ScalingVF}
 S\bm{v}=\sum_{j=1}^3x^j\partial_j\bm{v}-\bm{v},\qquad Sf=x\cdot\nabla f
\end{equation} 
resp.\footnote{In terms of the rotations $\Omega_{ab}$ of Section \ref{ssec:angLP} we have that $\Omega=\Omega_{12}$.}
\begin{equation}\label{RotationVF}
\Omega \bm{v}=(x^1\partial_2-x^2\partial_1)\bm{v}-\bm{v}_\h^\perp,\qquad \Omega f=(x^1\partial_2-x^2\partial_1)f.
\end{equation}

In both cases, we observe that the vector field $V\in\{S,\Omega\}$ commutes with the Hodge decomposition and leads to the linearized equation:
\begin{equation}\label{eq:IER-VF}
 \partial_t V \bm{u}+V \bm{u}\cdot\nabla \bm{u}+\bm{u}\cdot\nabla V \bm{u}+\vec{e}_3\times V \bm{u}+\nabla p_V=0,\quad \div\, V \bm{u}=0.
\end{equation}
In particular, the nonlinear flow of \eqref{eq:EC} preserves \emph{axisymmetry}, the invariance under the action of $\Omega$, i.e.\ under rotations about the $\vec{e}_3$ axis.

We note that both $S$ and $\Omega$ are natural in the sense that they correspond to flat derivatives in spherical coordinates $(\rho,\theta,\phi)$:
\begin{equation*}
\Omega=\partial_\theta,\qquad S=\rho\partial_\rho.
\end{equation*}
In particular, they commute and they both behave well under Fourier transform: we have 
\begin{equation}
 \widehat{Sf}=-3\hat{f}-S\hat{f}=S^*\hat{f},\qquad \widehat{\Omega f}=\Omega\hat{f}.
\end{equation}
In practice, we will thus be able to equivalently work with $\mathcal{F}^{-1}(V\hat{f})$ or $V f$, $V\in\{\Omega,S\}$ (since they differ by at most a multiple of $f$), and will henceforth ignore this distinction.

\subsection{Choice of unknowns and nonlinearity in axisymmetry}\label{ssec:unknownsetc}
To motivate our choice of variables, we first observe that the linear part of \eqref{eq:EC},
\begin{equation}
 \partial_t \bm{u}+\vec{e}_3\times \bm{u}+\nabla p=0, \quad \div\; \bm{u}=0,
\end{equation}
is \emph{dispersive}. Here $\Delta p=\curl_\h \bm{u}:=\partial_{x^1}u^2-\partial_{x^2}u^1$, so using the divergence condition one sees directly that the linear system is equivalent to
\begin{equation}\label{eq:linIER}
\begin{aligned}
\partial_t\curl_\h \bm{u}-\partial_3\bm{u}^3&=0,\qquad
\partial_t\bm{u}^3+\partial_3\Delta^{-1}\curl_\h \bm{u}=0.
\end{aligned}
\end{equation}
The dispersion relations $\pm i\Lambda(\xi)$ satisfies $(i\Lambda)^2=-\xi_3^2/\abs{\xi}^2$, and we choose
\begin{equation}
 \Lambda(\xi)=\frac{\xi_3}{\abs{\xi}}.
\end{equation}
We also use this notation to denote the associated differential operators, e.g.\ the real operator $i\Lambda=\partial_3\vert\nabla\vert^{-1}$.

\subsubsection{Scalar unknowns}
Due to the incompressibility condition, $\bm{u}$ has two scalar degrees of freedom. To exploit this we will work with the (scalar) variables
\begin{equation}\label{eq:a-c-def}
 A:=\abs{\nabla_\h}^{-1}\curl_\h \bm{u},\qquad C:=\abs{\nabla} \abs{\nabla_\h}^{-1}\bm{u}^3,\qquad\nabla_h=(\partial_{x^1},\partial_{x^2},0)
\end{equation}
which are chosen such that the normalization \eqref{MainPropU} holds. Here $\bm{u}$ can be recovered from $(A,C)$ as
\begin{equation}\label{Formu0}
 \bm{u}=\bm{u}_A+\bm{u}_C,
\end{equation}
where\footnote{We use the convention that repeated latin indices are summed $1-2$ and repeated greek indices are summed $1-3$.}
\begin{equation}\label{eq:Formu}
 \begin{split}
 \bm{u}_A&:=-\nabla_\h^\perp\abs{\nabla_\h}^{-1}A,\qquad \bm{u}_A^j=\in^{jk}\vert\nabla_\h\vert^{-1}\partial_k A,\\
 \bm{u}_C&:= i\Lambda\nabla_\h\abs{\nabla_\h}^{-1}C+\sqrt{1-\Lambda^2}C\;\vec{e}_3,\qquad \bm{u}_C^j=i\Lambda\vert\nabla_\h\vert^{-1}\partial_j C,\quad \bm{u}_C^3=\sqrt{1-\Lambda^2}C
\end{split}
\end{equation}
and for any vector field $V\in\{S,\Omega\}$ and any Fourier multiplier $m:\mathbb{R}^3\to\mathbb{R}$,
\begin{equation}\label{MainPropU}
\begin{aligned}
 V\bm{u}^\alpha&=\bm{u}_{VA}^\alpha+\bm{u}_{VC}^\alpha,\qquad V\in\{S,\Omega\},\\
 \norm{m\bm{u}}_{L^2}^2&=\norm{mA}_{L^2}^2+\norm{mC}_{L^2}^2=\norm{m\bm{u}_A}_{L^2}^2+\norm{m\bm{u}_C}_{L^2}^2.\\
\end{aligned}
\end{equation}
Using that
\begin{equation}
 \Delta p=\vert\nabla_\h\vert A-\div\left[\bm{u}\cdot\nabla \bm{u}\right]=\vert\nabla_\h\vert A-\partial_\alpha\partial_\beta\left[\bm{u}^\alpha \bm{u}^\beta\right].
\end{equation}
we obtain that \eqref{eq:EC} is equivalent to
\begin{equation}\label{eq:EC2}
\begin{aligned}
\partial_tA-i\Lambda C&=-\vert\nabla_\h\vert^{-1}\partial_j\partial_n\in^{jk}\left[\bm{u}^n\bm{u}^k\right]-i\Lambda\vert\nabla\vert \in^{jk}\vert\nabla_\h\vert^{-1}\partial_j\left[\bm{u}^3\bm{u}^k\right],\\
\partial_tC-i\Lambda A&=-i\Lambda\vert\nabla\vert\sqrt{1-\Lambda^2}\left[\vert\nabla_\h\vert^{-2}\partial_j\partial_k\left[\bm{u}^j\bm{u}^k\right]+\bm{u}^3\bm{u}^3\right]-\vert\nabla\vert \vert\nabla_\h\vert^{-1}\partial_j(1-2\Lambda^2)\left[\bm{u}^j\bm{u}^3\right].
\end{aligned}
\end{equation}
Here the structure of the nonlinearity is apparent as a quasilinear, quadratic form in $A,C$ without singularities at low frequency. 

\begin{remark}\label{rem:a-vs-swirl}
 In the classical axisymmetric formulation of flows as $\bm{u}=\bm{u}_\theta \vec{e}_\theta+\bm{u}_r \vec{e}_r+\bm{u}_z \vec{e}_3$ where $(\vec{e}_\theta,\vec{e}_r,\vec{e}_3)$ are the basis vectors of a cylindrical coordinate system, one has that $\curl_\h\bm{u}=r^{-1}\partial_r(r\bm{u}_\theta)$. Our unknown $A$ is thus closely linked to the swirl $\bm{u}_\theta$ of $\bm{u}$: it satisfies $\abs{\nabla_\h}A=r^{-1}\partial_r(r\bm{u}_\theta)$. In general $A$ will not vanish for the solutions we construct, and neither will their swirl.
\end{remark}

\subsubsection{On the role of axisymmetry}\label{sec:role-axisymm}
A particular family of solutions to \eqref{eq:EC} is given by a $3d$ system of $2d$ Euler equations, i.e.\ $\bm{u}(t,x_1,x_2,x_3)=(v_\h(t,x_1,x_2),w(t,x_1,x_2))$ satisfies \eqref{eq:EC} provided that $v_\h:\R\times\R^2\to\R^2$ and $w:\R\times\R^2\to\R$ solve
 \begin{equation}
  \partial_t v_\h+v_\h\cdot\nabla v_\h +(-v_2,v_1)^\intercal+\nabla q=0, \quad \partial_t w+v_h\cdot\nabla w=0,\quad \div(v_\h)=0.
 \end{equation}
 Since in $2d$ the rotation term $(-v_2,v_1)^\intercal$ is a gradient, it can be absorbed into the pressure and thus $\bm{u}$ as above is a solution to the Euler-Coriolis system if $v_\h$ satisfies the $2d$ Euler equations, with $w$ passively advected by $v_\h$. While such solutions have infinite energy and are thus excluded from our functional setting on $\R^3$, they have been shown in \cite{BMN1997,Gre1997} to be of leading order on a (generic) torus $\T^3$ with sufficiently fast rotation. 
 
 In the setting of $\R^3$ one also encounters the $2d$ Euler equations through a \emph{resonant subsystem}: substituting $\bm{u}$ in terms of $A,C$ as in \eqref{eq:Formu} one sees that \eqref{eq:EC2} is of the form
\begin{equation}\label{eq:EC-fullstructure}
\begin{aligned}
 \partial_tA-i\Lambda C&=Q_{null}^A(A,C)+Q^A_\Omega(A,C)+Q^A_E(A,C),\\
 \partial_tC-i\Lambda A&= Q_{null}^C(A,C)+Q^C_\Omega(A,C)+Q^C_E(A,C),
\end{aligned}
\end{equation}
where for $W\in\{A,C\}$ the quadratic terms
\begin{itemize}
 \item $Q_{null}^W(A,C)$ contain a favorable null type structure (discussed below in Section \ref{sec:axisymm} in detail),
 \item $Q^W_\Omega(A,C)$ contain a rotational product structure,
 \item $Q^W_E(A,C)$ contain the $2d$ Euler equations in the following sense: near $\Lambda=0$ their contribution to $A$ is
 \begin{equation}
 \begin{aligned}
  \partial_t\widetilde{A}+(\nabla_h^\perp\vert\nabla_h\vert^{-2} \widetilde{A})\cdot \nabla_h\widetilde{A}=l.o.t.,\qquad \widetilde{A}:=\abs{\nabla_\h}A=\curl_\h(\bm{u}),
 \end{aligned}
 \end{equation}
 in which one recognizes the $2d$ Euler equations in vorticity formulation for $\widetilde{A}$, while $C$ is being passively transported by $A$.
\end{itemize}

In terms of the nonlinear structure, the crucial observation for our purposes is that $Q^W_E$ \emph{vanishes on axisymmetric functions}, so that in our setting we do not have to contend with a possible fast norm growth due to $2d$ Euler-type nonlinear interactions in \eqref{eq:EC-fullstructure}. 
Moreover, it turns out that also $Q^W_\Omega$ vanishes under axisymmetry, but this is less important for our analysis.

\begin{remark}
\begin{enumerate}[wide]
 \item The assumption of axisymmetry brings some further simplifications (see e.g.\ Lemma \ref{lem:VFcross}), but those are less vital for our arguments. 
 \item Although all our functions (including the localizations) are axisymmetric \emph{in their arguments} and we have that 
 \begin{equation}
  \Omega f=0\quad\text{if } f\text{ axisymmetric,}
 \end{equation}
 the vector field $\Omega$ still plays an important role, since it does not vanish on expressions of several arguments, such as the phase functions $\Phi$ (see e.g.\ Lemma \ref{lem:vfsizes-mini}).
\end{enumerate}
\end{remark}

\subsubsection{The equations in axisymmetry}\label{sec:axisymm}
In order to properly describe the structure of the nonlinearity in \eqref{eq:EC2} for axisymmetric solutions, we introduce the following collection of zero homogeneous symbols:
\begin{equation}\label{eq:barE}
 \bar{E}:=\left\{\Lambda(\zeta),\sqrt{1-\Lambda^2(\zeta)},\frac{\xi_{\h}\cdot\theta_{\h}}{\abs{\xi_{\h}}\abs{\theta_{\h}}},\frac{\xi_{\h}^\perp\cdot\theta_{\h}}{\abs{\xi_{\h}}\abs{\theta_{\h}}}\,:\,\zeta\in\{\xi,\xi-\eta,\eta\},\theta\in\{\xi-\eta,\eta\}\right\}.
\end{equation}
With the standard notation 
\begin{equation*}
 \mathcal{F}(Q_{\mathfrak{n}}(f,g))(\xi)=\int_\eta \mathfrak{n}(\xi,\eta)\hat{f}(\xi-\eta)\hat{g}(\eta)d\eta
\end{equation*}
for quadratic expressions with multiplier $\mathfrak{n}$, we have the following result (see also \cite[Lemma 2.1]{rotE}):
\begin{lemma}\label{lem:ECmult}
 Let $\bm{u}$ be an axisymmetric solution to \eqref{eq:EC} on a time interval $[0,T]$, so that $A,C$ as defined in \eqref{eq:a-c-def} are axisymmetric functions that solve \eqref{eq:EC2}. Then the \emph{dispersive unknowns} 
 \begin{equation}
  U_+:=A+C,\quad U_-:=A-C,
 \end{equation}
 satisfy the equations
 \begin{equation}\label{eq:EC_disp}
 \begin{aligned}
  (\partial_t-i\Lambda)U_+&=Q_{\mathfrak{n}^{++}_+}(U_+,U_+)+Q_{\mathfrak{n}^{+-}_+}(U_+,U_-)+Q_{\mathfrak{n}^{--}_+}(U_-,U_-),\\
  (\partial_t+i\Lambda)U_-&=Q_{\mathfrak{n}^{++}_-}(U_+,U_+)+Q_{\mathfrak{n}^{+-}_-}(U_+,U_-)+Q_{-\mathfrak{n}^{--}_-}(U_-,U_-),
 \end{aligned}
 \end{equation}
 with multipliers satisfying 
 \begin{equation}
 \begin{aligned}
  \mathfrak{n}^{\mu\nu}_\kappa(\xi,\eta)&=\abs{\xi}\bar{\mathfrak{n}}^{\mu\nu}_\kappa(\xi,\eta),\qquad\kappa,\mu,\nu\in\{+,-\},\\
  \bar{\mathfrak{n}}^{\mu\nu}_\kappa&\in\textnormal{span}_\R\left\{\Lambda(\zeta_1)\sqrt{1-\Lambda^2(\zeta_2)}\cdot \bar{\mathfrak{n}}(\xi,\eta),\zeta_1,\zeta_2\in\{\xi,\xi-\eta,\eta\},\bar{\mathfrak{n}}\in\bar{E}\right\}.
 \end{aligned} 
 \end{equation}
\end{lemma}

In words: in the axisymmetric case, in the dispersive variables $U_\pm$ the symbols of the quadratic, quasilinear nonlinearity of \eqref{eq:EC_disp} contain a derivative $\abs{\nabla}$ and factors of $\Lambda(\zeta_1)\sqrt{1-\Lambda^2(\zeta_2)}$ for some $\zeta_1,\zeta_2\in\{\xi,\xi-\eta,\eta\}$. We shall make frequent use of this null type structure in our nonlinear estimates -- a quantified version of it may be found below in Lemma \ref{lem:ECmult_bds}.

\begin{proof}[Proof of Lemma \ref{lem:ECmult}]
 This has been established in our prior work \cite[Lemma 2.1]{rotE}.
\end{proof}

\subsubsection{Profiles and bilinear expressions}\label{ssec:profiles}
Introducing the \emph{profiles} $\U_\pm$ of the dispersive unknowns $U_\pm$ as
\begin{equation}\label{eq:def_profiles}
 \U_+(t):=e^{-it\Lambda}U_+(t),\qquad \U_-(t):=e^{it\Lambda}U_-(t),
\end{equation}
we can express \eqref{eq:EC_disp} in terms of $\U_\pm$ and see that the bilinear terms are of the form
\begin{equation}\label{eq:def_Qm}
 \Q_\m(\U_\mu,\U_\nu)(s):=\mathcal{F}^{-1}\left(\int_{\R^3}e^{\pm is\Phi_{\mu\nu}(\xi,\eta)}\m(\xi,\eta)\widehat{\U_\mu}(s,\xi-\eta)\widehat{\U_\nu}(s,\eta)d\eta\right),\qquad \mu,\nu\in\{-,+\},
\end{equation}
for a phase function
\begin{equation}\label{eq:def_phi}
 \Phi_{\mu\nu}(\xi,\eta):=\Lambda(\xi)+\mu\Lambda(\xi-\eta)+\nu\Lambda(\eta),\qquad \mu,\nu\in\{-,+\},
\end{equation}
and $\m$ one of the multipliers $\mathfrak{n}_\kappa^{\mu,\nu}$ of Lemma \ref{lem:ECmult}. By Duhamel's formula we thus have from \eqref{eq:EC_disp} that
\begin{equation}\label{eq:EC_disp_Duham}
\begin{aligned}
 \U_+(t)=\U_+(0)+\int_0^t  \left[\Q_{m^{++}_+}(\U_+,\U_+)+\Q_{m^{+-}_+}(\U_+,\U_-)+\Q_{m^{--}_+}(\U_-,\U_-)\right](s) ds,\\
 \U_-(t)=\U_-(0)+\int_0^t  \left[\Q_{m^{++}_-}(\U_+,\U_+)+\Q_{m^{+-}_-}(\U_+,\U_-)+\Q_{m^{--}_-}(\U_-,\U_-)\right](s) ds.
\end{aligned} 
\end{equation}
Defining for a multiplier $\mathfrak{n}$ the bilinear expression
\begin{equation}
 B_\mathfrak{n}(f,g)(t):= \int_0^t \Q_\mathfrak{n}(f,g)(s) ds
\end{equation}
we may thus write \eqref{eq:EC_disp_Duham} compactly as
\begin{equation}\label{eq:EC-Duhamel}
 \U_\pm(t)=\U_\pm(0)+\sum_{\mu,\nu\in\{+,-\}}B_{\mathfrak{n}_\pm^{\mu\nu}}(\U_\mu,\U_\nu)(t).
\end{equation}
We will use this expression as the basis for our bootstrap arguments.

\section{Functional framework and main result}\label{sec:mainresult}
We begin with a discussion of some necessary background in Sections \ref{ssec:loc} and \ref{ssec:angLP}, to make our statement in Theorem \ref{MainThm} more precise -- see Section \ref{ssec:mainresult}.

\subsection{Localizations}\label{ssec:loc}
Let $\psi\in C^\infty(\R,[0,1])$ be a radial, non-increasing bump function supported in $[-\frac{8}{5},\frac{8}{5}]$ with $\psi|_{[-\frac{4}{5},\frac{4}{5}]}\equiv1$, and set $\varphi(x):=\psi(x)-\psi(2x)$.

We use the notations that for $a\in\Z$ and $b,c\in\Z^-$
\begin{equation}\label{eq:loc_def2}
 \varphi_{a,b}(\zeta):=\varphi(2^{-a}\abs{\zeta})\varphi(2^{-b}\sqrt{1-\Lambda^2(\zeta)}), \qquad \varphi_{a,b,c}(\zeta):=\varphi_{a,b}(\zeta)\varphi(2^{-c}\Lambda(\zeta)),
\end{equation}
and will generically denote by $\widetilde{\varphi}$ a function that has similar support properties as $\varphi$, and analogously for $\widetilde{\varphi}_{a,b}$ and $\widetilde{\varphi}_{a,b,c}$. 
We define the associated Littlewood-Paley projections $P_{k_j,p_j}$ and $P_{k_j,p_j,q_j}$ as
\begin{equation}
 \mathcal{F}(P_{k_j,p_j}f)(\xi)=\varphi_{k_j,p_j}(\xi)\hat{f}(\xi),\qquad \mathcal{F}(P_{k_j,p_j,q_j}f)(\xi)=\varphi_{k_j,p_j,q_j}(\xi)\hat{f}(\xi),
\end{equation}
and remark that these projections are bounded on $L^r$, $1\le r\le \infty$. We note that $p,q$ are \emph{not independent} parameters -- on the support of $P_{k,p,q}$ there holds that $2^{2p+q}=\min\{2^{2p},2^q\}$ and $2^{2p}+2^q\simeq 1$. In particular, there is a discrepancy between $p$ and $q$, in that the natural comparison of scales is between $2p$ and $q$ (rather than $p$ and $q$).

To collect the above localizations we will make use of the notation
\begin{equation}\label{eq:loc_def3}
 \chi_\h(\xi,\eta):=\varphi_{k,p}(\xi)\varphi_{k_1,p_1}(\xi-\eta)\varphi_{k_2,p_2}(\eta),\qquad \chi(\xi,\eta):=\varphi_{k,p,q}(\xi)\varphi_{k_1,p_1,q_1}(\xi-\eta)\varphi_{k_2,p_2,q_2}(\eta),
\end{equation}
and write 
\begin{equation}
 w_{\max}:=\max\{w,w_1,w_2\},\quad w_{\min}:=\min\{w,w_1,w_2\},\qquad w\in\{k,p,q\}.
\end{equation}

\subsection{Angular Littlewood-Paley decomposition}\label{ssec:angLP}
We now introduce angular \blue{regularity} localizations via associated Littlewood-Paley type projectors. \blue{Due to axial symmetry, these can be constructed based on the spectral decomposition\footnote{\blue{That this controls regularity in $\Upsilon$ can be seen from \eqref{eq:R-Bernstein} and \eqref{eq:UpsOmegas} below.}} of the Laplacian on $\mathbb{S}^2$, $\Delta_{\mathbb{S}^2}=\Upsilon^2$.}

Let $N=(0,0,1)\in\R^3$ denote the north pole of the standard 2-sphere $\mathbb{S}^2$ and let $\mathcal{Z}_n(P)=\mathfrak{Z}_n(\langle P,N\rangle)$ denote the $n$-th zonal spherical harmonic, given explicitly via the Legendre polynomial $L_n$ by
\begin{equation}
\mathfrak{Z}_n(x)=\frac{2n+1}{4\pi}L_n(x),\qquad L_n(z)=\frac{1}{2^n(n!)}\frac{d^n}{dz^n}[(z^2-1)^n].
\end{equation}
Using this, for $k\in\Z$ we define the ``angular Littlewood-Paley projectors'' $\bar{R}_{\leq \ell}, \bar{R}_\ell$ by
\begin{equation}
\begin{split}
 \left(\bar{R}_{\le \ell}f\right)(x)&=\sum_{n\ge 0}\psi(2^{-\ell}n)\int_{\mathbb{S}^2}f(\vert x\vert\vartheta)\mathfrak{Z}_n(\langle \vartheta,\frac{x}{\vert x\vert}\rangle)d\nu_{\mathbb{S}^2}(\vartheta),\\
 \left(\bar{R}_\ell f\right)(x)&=\sum_{n\ge 0}\varphi(2^{-\ell}n)\int_{\mathbb{S}^2}f(\vert x\vert\vartheta)\mathfrak{Z}_n(\langle \vartheta,\frac{x}{\vert x\vert}\rangle)d\nu_{\mathbb{S}^2}(\vartheta),
\end{split}
\end{equation}
where $d\nu_{\mathbb{S}^2}$ denotes the standard measure on the sphere (so that $\nu_{\mathbb{S}^2}(\mathbb{S}^2)=4\pi$).
These operators are bounded on $L^2$ and self-adjoint; their key properties parallel those of standard Littlewood-Paley projectors:

\begin{proposition}\label{prop:LPOmega}
For any $\ell\in\Z$, the angular Littlewood-Paley projectors $\bar{R}_\ell$ satisfy:
\begin{enumerate}[label=(\roman*),wide]
\item\label{it:LPang-comm} $\bar{R}_\ell$ commutes with regular Littlewood-Paley projectors, both in space and in frequency. Besides, $\bar{R}_\ell$ commutes with vector fields $\Omega_{ab}=x_a\partial_{x_b}-x_b\partial_{x_a}$ ($a,b\in\{1,2,3\}$), $S$, and the Fourier transform:
\begin{equation*}
\begin{split}
 [\Omega_{ab},\bar{R}_\ell]=[S,\bar{R}_\ell]=[\bar{R}_\ell,P_{k}]=[\bar{R}_\ell,\mathcal{F}]=0.
\end{split}
\end{equation*}

\item\label{it:LPang-orth} $\bar{R}_\ell$ constitutes an almost orthogonal partition of unity in the sense that
\begin{equation*}
\begin{split}
f=\sum_{\ell\ge0}\bar{R}_\ell f,&\qquad \bar{R}_\ell\bar{R}_{\ell^\prime}=0\quad\hbox{whenever}\quad \vert \ell-\ell^\prime\vert\ge4,\qquad
\Vert f\Vert_{L^2}^2\simeq\sum_{\ell\ge 0}\Vert \bar{R}_\ell f\Vert_{L^2}^2,\\
\bar{R}_\ell\left[\bar{R}_{\ell_1}f\cdot \bar{R}_{\ell_2}g\right]&=0\quad\hbox{whenever}\quad\max\{\ell,\ell_1,\ell_2\}\ge\textnormal{med}\{\ell,\ell_1,\ell_2\}+4
\end{split}
\end{equation*}

\item\label{it:LPang-bd} $\bar{R}_\ell$ and $\bar{R}_{\le \ell}$ are bounded on $L^r$, $1\le r\le\infty$,
\begin{equation}
\Vert \bar{R}_\ell f\Vert_{L^r}+\Vert \bar{R}_{\le \ell}f\Vert_{L^r}\lesssim \Vert f\Vert_{L^r}.
\end{equation}

\item\label{it:LPang-Bern} We have a Bernstein property: There holds that
\begin{equation}\label{eq:R-Bernstein}
\begin{split}
 \sum_{1\leq a<b\leq 3}\Vert \Omega_{ab}\bar{R}_\ell f\Vert_{L^r}\simeq 2^\ell\Vert \bar{R}_\ell f\Vert_{L^r},\qquad 1\le r\le\infty.
\end{split}
\end{equation}
\end{enumerate}
\end{proposition}

We refer the reader to Appendix \ref{apdx:angLP} for the proof of this proposition. It is important to understand the interplay between the $R_\ell$ and $P_{k,p}$ localizations. By direct computations we have that
\begin{equation}\label{eq:comm-angLP-p}
 \norm{[\Omega_{j3},P_{k,p}]}_{L^r\to L^r}\lesssim2^{-p},
 \qquad j=1,2,\quad 1\leq r\leq \infty.
\end{equation}
In particular, for localization $P_{k,p}\bar{R}_\ell$ (in both horizontal frequency $p$ and ``angular frequency'' $\ell$) this shows that we should not go below the scale $-p>\ell$, since there the projections do not commute (up to lower order terms). In practice we will thus work with projectors that incorporate this ``uncertainty principle'' $\ell+p\geq 0$: for $p\in\Z^{-}$, $\ell\in\Z^{+}$, we introduce the operators
\begin{equation}
\begin{aligned}
 R^{(p)}_\ell:=
\begin{cases}
 0,&p+\ell< 0,\\
 \bar{R}_{\le\ell},&p+\ell= 0,\\
 \bar{R}_{\ell},&p+\ell> 0.
\end{cases}
\end{aligned}
\end{equation}
\subsubsection*{Convention:} For simplicity of notation we shall henceforth \emph{drop the superscript $(p)$} on $R_\ell$, i.e.\
\begin{equation}
 R_\ell\equiv R_\ell^{(p)},
\end{equation}
since it will always be clear from the context of localization in the corresponding $p$.

Clearly, key features of Proposition \ref{prop:LPOmega} transfer to $R_\ell$: For example, we have the decomposition
\begin{equation}
 P_{k}f=\sum_{\ell+p\geq 0}P_{k,p}R_{\ell}^{(p)}f\equiv \sum_{\ell+p\geq 0}P_{k,p}R_{\ell}f.
\end{equation}

\begin{remark}\label{rem:comm-loc}
 One checks that
 \begin{equation}\label{eq:comm-angLP-q}
 \norm{[\Omega_{j3},P_{k,p,q}]}_{L^r\to L^r}\lesssim 2^{-p-q},
 \qquad j=1,2,\quad 1\leq r\leq \infty,
 \end{equation}
 Since $q$ plays a similar role as $2p$ in terms of scales, it does not seem advantageous to at once localize in $p,\ell$ and additionally $q$. Rather, typically we will first only work with localizations in $p$ and $\ell$, and only introduce localizations in $q$ once the other parameters are under control.
\end{remark}

\subsection{Main Result}\label{ssec:mainresult}
With the notations $k^+:=\max\{0,k\}$, $k^-:=\min\{k,0\}$ and for $\beta>0$ to be chosen we introduce \blue{now our key norms, both weighted, $L^2$ based to allow for a Fourier analysis based approach:}
\begin{align}
 \norm{f}_B&:=\sup_{k\in\Z,\, p,q\in\Z^-}2^{3k^+-\frac{1}{2}k^{-}}2^{-p-\frac{q}{2}}\norm{P_{k,p,q}f}_{L^2},\label{eq:defBnorm}\\
 \norm{f}_X&:=\sup_{\substack{k\in\Z,\,\ell\in\Z^+\!\!,\,p\in\Z^-\\\ell+p\geq 0}}2^{3k^+} 2^{(1+\beta)\ell}2^{\beta p}\norm{ P_{k,p}R_\ell f}_{L^2}.\label{eq:defXnorm}
\end{align}
\blue{As discussed in the introduction on page \pageref{it:normchoice}, these norms play complementary roles. Through appropriate weighting of the anisotropic resp.\ angular Littlewood-Paley projectors, the $B$ norm captures anisotropic localization and scales like the Fourier transform in $L^\infty$,\footnote{\blue{Such a scaling may also be motivated by the fact that the stationary phase arguments that yield linear decay can only be optimal if one controls the Fourier transform in $L^\infty$. This control is indeed given by a \emph{combination} of $B$ and $X$ norms including some vector fields $S$, as we show in Lemma \ref{lem:ControlLinfty} -- we refer to its proof for a further demonstration of the different roles in terms of angular regularity of the two norms in \eqref{eq:defBnorm}, \eqref{eq:defXnorm}.}} while the $X$ norm accounts for angular derivatives in $\Upsilon$. The additional weights in terms of the frequency size $k$ are designed to capture both the sharp decay at the linear level, and also allow to overcome the derivative loss inherent to the nonlinearity. In particular the large power of $k^+$ ensures that we have
\begin{equation*}
\Vert \nabla R_\ell f\Vert_{L^\infty}\lesssim 2^{-\ell}\Vert  f\Vert_X,\qquad\Vert\nabla P_{k,p,q}f\Vert_{L^\infty}\lesssim 2^{p+q/2}\Vert f\Vert_B.
\end{equation*}}

In detail, our main result from Theorem \ref{MainThm} can then be stated as the following global existence result for the Euler-Coriolis system \eqref{eq:EC}: 
\begin{theorem}\label{thm:EC}
 Let $N\geq 5$. There exist $M,N_0\in\N$, $\beta>0$ with $N_0\gg M\gg \beta^{-1}+N$, and $\eps_\ast>0$ such that if $U_{\pm,0}$ satisfy
 \begin{equation}
 \begin{aligned}\label{eq:id}
  \norm{U_{\pm,0}}_{H^{2N_0}\cap \dot{H}^{-1}}+\norm{S^aU_{\pm,0}}_{L^2\cap \dot{H}^{-1}}&\le\eps_0,\qquad 0\le a\le M,\\
  \norm{S^bU_{\pm,0}}_{B}+\norm{S^bU_{\pm,0}}_{X}&\le\eps_0,\qquad 0\le b\le N,
 \end{aligned}
 \end{equation}
for some $0<\varepsilon_0<\varepsilon_\ast$, then there exists a unique global solution $U_\pm\in C(\R,\R^3)$ to \eqref{eq:EC_disp}. Moreover, $U_\pm(t)$ decay and have (at most) slowly growing energy
\begin{equation}
 \norm{U_\pm(t)}_{L^\infty}\lesssim \eps_0 \ip{t}^{-1},\qquad \norm{U_\pm(t)}_{H^{2N_0}}\lesssim \eps_0 \ip{t}^{C\eps_0}
\end{equation}
for some $C>0$, and in fact $U_\pm(t)$ scatters linearly.
\end{theorem} 

\begin{remark}
 In order to keep the essence of the arguments as clear as possible, we have not striven to optimize the number of vector fields and derivatives in the above result. As our arguments show, a choice of $\beta=10^{-2}$, $N_0=O(10^{7})$ and such that $N_0\gg M\gg \beta^{-1}$ works.
\end{remark}

Theorem \ref{thm:EC} is established through a bootstrap argument, \blue{which we discuss next}. We will show the following result, which implies all of Theorem \ref{thm:EC} except for the scattering statement:

\begin{proposition}\label{prop:btstrap}
 Let $U_\pm\in C([0,T],\R^3)$ be solutions to \eqref{eq:EC_disp} for some $T>0$ with profiles $\U_{\pm}:=e^{\mp it\Lambda}U_{\pm}$, and initial data satisfying \eqref{eq:id}.
 
 If for $t\in [0,T]$ there holds that
 \begin{equation}
 \begin{aligned}\label{eq:btstrap-assump}
  \norm{S^b\U_\pm(t)}_{B}+\norm{S^b\U_\pm(t)}_{X}&\leq \eps_1,\qquad 0\le b\le N, 
 \end{aligned}
 \end{equation}
for some $0<\varepsilon_1\le\sqrt{\varepsilon_0}$, then in fact we have the improved bounds
 \begin{align}
  \norm{S^b\U_\pm(t)}_{B}+\norm{S^b\U_\pm(t)}_{X}&\lesssim\eps_0+\eps_1^2,\qquad 0\le b\le N, \label{eq:btstrap-concl1}
 \end{align}
 and for some $C>0$ there holds that
 \begin{equation}\label{eq:btstrap-concl2}
  \norm{\U_\pm(t)}_{H^{2N_0}\cap H^{-1}}+\norm{S^a\U_\pm(t)}_{L^2\cap H^{-1}}\lesssim\eps_0\ip{t}^{C\eps_1},\qquad 0\le a\le M.
 \end{equation}

\end{proposition} 

Finally, the linear scattering in $L^2$ is a direct consequence of the fast decay of $\partial_t\U_\pm(t)$, and more is true:
\begin{corollary}\label{cor:scatter}
 Let $U_\pm(t)$ be global solutions to \eqref{eq:EC_disp}, constructed via the bootstrap in Proposition \ref{prop:btstrap} that in particular satisfy the bounds \eqref{eq:btstrap-concl1} and \eqref{eq:btstrap-concl2}. There exist $\U_{\pm}^\infty\in  B\cap X$ such that
 \begin{equation}\label{eq:scatter}
  \norm{e^{\mp it\Lambda}U_{\pm}(t)-\U_{\pm}^{\infty}}_{B\cap X}\to 0,\qquad t\to +\infty.
 \end{equation}
\end{corollary}
\begin{proof}
 By Lemma \ref{lem:dtfinL2} we have that $\norm{\partial_t \U_\pm(t)}_{L^2}\lesssim \ip{t}^{-\frac{5}{4}}$, and hence $\U_\pm(t)$ are $L^2$ Cauchy sequences (in time) and converge to $\mathcal{U}_{\pm}^\infty$ as $t\to\infty$. Similarly, by Proposition \ref{prop:Bnorm} resp.\ Propositions \ref{prop:Xnorm1} and \ref{prop:Xnorm2} we have that $\U_\pm(t)$ are Cauchy sequences in the $B$ resp.\ $X$ norm.
\end{proof}

\blue{We outline next the strategy of proof of Proposition \ref{prop:btstrap}. In particular, we show how control of the nonlinearity can be obtained through a reduction to several bilinear estimates, which are at the heart of the rest of this article.}
\begin{proof}[Proof of Proposition \ref{prop:btstrap}]
We note that under the assumptions \eqref{eq:btstrap-assump} it follows from the linear decay estimates in Proposition \ref{prop:decay} that
\begin{equation}\label{eq:decay_cons}
 \norm{S^aU_\pm(t)}_{L^\infty}\lesssim \eps_1 \ip{t}^{-1},\quad 0\leq a\leq N-3,
\end{equation}
and thus the slow growth of the energy and vector fields \eqref{eq:btstrap-concl2} follows from a standard energy estimate for the system \eqref{eq:EC} -- see Corollary \ref{cor:energy}. We note further that by interpolation we also have bounds for up to $N$ vector fields in $H^{N_0}$: With Lemma \ref{lem:interpol} and \eqref{eq:btstrap-concl2} there holds that for $b\leq N$ we have
\begin{equation}\label{eq:interpol_energy}
 \norm{S^b U_\pm(t)}_{H^{N_0}}\lesssim \eps_0 \ip{t}^{C\eps_1}.
\end{equation}

The key point is thus to establish \eqref{eq:btstrap-concl1}. We proceed as follows:

\subsubsection*{Reduction}
From the Duhamel formula \eqref{eq:EC-Duhamel} for the profiles $\U_\pm$ we have that
\begin{equation}
 \norm{S^b\U_\pm(t)}_{B\cap X}\leq \norm{S^b\U_\pm(0)}_{B\cap X}+\!\!\!\sum_{\mu,\nu\in\{+,-\}}\norm{S^bB_{\mathfrak{n}_\pm^{\mu\nu}}(\U_\mu,\U_\nu)}_{B\cap X},
\end{equation}
hence to prove \eqref{eq:btstrap-concl1} it suffices to show that under the bootstrap assumptions \eqref{eq:btstrap-assump}, for any multiplier $\m=\mathfrak{n}_\pm^{\mu,\nu}$ as in Lemma \ref{lem:ECmult} there holds that
\begin{equation}
 \norm{S^bB_{\m}(\U_\mu,\U_\nu)}_B+\norm{S^bB_{\m}(\U_\mu,\U_\nu)}_X\lesssim \eps_1^2,\qquad 0\leq b\leq N.
\end{equation}

Since $S$ generates a symmetry of the equation, its application to a bilinear term $B_\m$ yields a favorable structure: With $S_\xi\Lambda(\xi-\eta)=-S_\eta\Lambda(\xi-\eta)$ and $S\Lambda=0$ there holds that $(S_\xi+S_\eta)\Phi_{\mu\nu}=0$, and one computes directly\footnote{Note that $S_\xi+S_\eta$ vanishes on the elements of $\bar{E}$ from \eqref{eq:barE}, and $S_\xi\abs{\xi}=\abs{\xi}$ (alternatively, see \cite[Lemma A.6]{rotE}).} that $(S_\xi+S_\eta)\m=\m$, so that from integration by parts we deduce that
\begin{equation}
\begin{aligned}
 S_\xi\mathcal{F}\left(\Q_\m(\U_\mu,\U_\nu)\right)(\xi)&=S_\xi\int_{\R^3}e^{\pm is\Phi_{\mu\nu}(\xi,\eta)}\m(\xi,\eta)\widehat{\U_\mu}(s,\xi-\eta)\widehat{\U_\nu}(s,\eta)d\eta\\
 &=\int_{\R^3}((S_\xi+S_\eta)e^{\pm is\Phi_{\mu\nu}(\xi,\eta)})\m(\xi,\eta)\widehat{\U_\mu}(s,\xi-\eta)\widehat{\U_\nu}(s,\eta)d\eta\\
 &\qquad+ \int_{\R^3}e^{\pm is\Phi_{\mu\nu}(\xi,\eta)}(S_\xi+S_\eta)\left(\m(\xi,\eta)\widehat{\U_\mu}(s,\xi-\eta)\widehat{\U_\nu}(s,\eta)\right)d\eta\\
 &\qquad +3\int_{\R^3}e^{\pm is\Phi_{\mu\nu}(\xi,\eta)}\m(\xi,\eta)\widehat{\U_\mu}(s,\xi-\eta)\widehat{\U_\nu}(s,\eta)d\eta\\
 &=\mathcal{F}\left(-2\,\Q_\m(\U_\mu,\U_\nu)+\Q_\m(S\U_\mu,\U_\nu)+\Q_\m(\U_\mu,S\U_\nu)\right)(\xi).
\end{aligned}
\end{equation}
It thus suffices to show that for $\mu,\nu\in\{+,-\}$, and $b_1,b_2\geq 0$ there holds that
\begin{equation}\label{eq:btstrap-concl1.1}
 \norm{B_{\m}(S^{b_1}\U_\mu,S^{b_2}\U_\nu)}_B+\norm{B_{\m}(S^{b_1}\U_\mu,S^{b_2}\U_\nu)}_X\lesssim \eps_1^2,\qquad b_1+b_2\leq N.
\end{equation}

\subsubsection*{Bilinear estimates}
To prove \eqref{eq:btstrap-concl1.1} it is convenient to localize the time variable. For $t\in[0,T]$ we choose a decomposition of the indicator function $\mathfrak{1}_{[0,t]}$ by functions $\tau_0,\ldots,\tau_{L+1}:\R \to [0,1]$, $\abs{L-\log_2(2+t)}\leq 2$, satisfying 
\begin{equation}
\begin{aligned}
& \mathrm{supp} \,\tau_0 \subseteq [0,2], \quad \mathrm{supp} \,\tau_{L+1}\subseteq [t-2,t],
  \quad \mathrm{supp}\,\tau_m\subseteq [2^{m-1},2^{m+1}]
  \quad \text{for} \quad  m \in \{1,\dots,L\},
\\
& \sum_{m=0}^{L+1}\tau_m(s) = \mathbf{1}_{[0,t]}(s), \quad \tau_m\in C^1(\R) \quad \text{and} \quad \int_0^t|\tau'_m(s)|\,ds\lesssim 1
  \quad \text{for} \quad m\in \{1,\ldots,L\}.
\end{aligned}
\end{equation}
We can then decompose
\begin{equation}
 B_\m(F,G)=\sum_m B_\m^{(m)}(F,G),\qquad B_\m^{(m)}(F,G):=\int_0^t\tau_m(s)\,\Q_\m(F,G)(s)ds.
\end{equation}
For simplicity of the expressions we will not carry the superscript $(m)$ and instead generically write $\B_\m$ for any of the time localized bilinear expressions $B_\m^{(m)}$ above. After establishing the relevant background and methodology in Sections \ref{sec:lin_decay}--\ref{sec:dtfbds}, we prove \eqref{eq:btstrap-concl1.1} by establishing for some $\delta>0$ the stronger bounds
\begin{equation}\label{eq:btstrap-concl1.1-B}
 \norm{\B_{\m}(S^{b_1}\U_\mu,S^{b_2}\U_\nu)}_B\lesssim 2^{-\delta^3 m}\eps_1^2,\qquad b_1+b_2\leq N
\end{equation}
in Proposition \ref{prop:Bnorm}, and 
\begin{equation}\label{eq:btstrap-concl1.1-X}
 \norm{\B_{\m}(S^{b_1}\U_\mu,S^{b_2}\U_\nu)}_X\lesssim 2^{-\delta^3 m}\eps_1^2,\qquad b_1+b_2\leq N
\end{equation}
in Propositions \ref{prop:Xnorm1} and \ref{prop:Xnorm2}. Explicitly we will choose
\begin{equation}\label{eq:def_delta}
 \delta:=2M^{-\frac{1}{2}},
\end{equation}
and another relevant parameter of smallness will be
\begin{equation}\label{eq:def_nu}
 \delta_0:=2N_0^{-1}.
\end{equation}
For technical reasons, in the proofs it will be useful to have the following hierarchy between $\delta_0$ and $\delta$, related to the sizes of $N_0, M$ in Theorem \ref{thm:EC}:
\begin{equation}\label{eq:nuvsdelta}
 10\delta_0<\delta^2, \quad\textnormal{i.e.}\quad N_0> 20 M.
\end{equation}

\end{proof}

\section{Linear decay}\label{sec:lin_decay}
Introducing the ``decay norm''
\begin{equation}\label{eq:defDnorm}
 \norm{f}_D:=\sup_{0\leq a\leq 3} \left( \norm{S^a f}_B+\norm{S^a f}_X \right),
\end{equation}
we have the following decay result:
\begin{proposition}\label{prop:decay}
Let $f$ be axisymmetric and $t>0$. We can split
\begin{equation}
 P_{k,p,q}e^{it\Lambda}f=I_{k,p,q}(f)+I\!I_{k,p,q}(f)
\end{equation}
where for any $0<\beta'<\beta$
\begin{equation}\label{eq:I-IIdecay}
\begin{aligned}
 \norm{I_{k,p,q}(f)}_{L^\infty}&\lesssim 2^{\frac{3}{2}k-3k^+}\cdot \min\{2^{2p+q},2^{-p-\frac{q}{2}}t^{-\frac{3}{2}}\}\norm{f}_D,\\
 \norm{I\!I_{k,p,q}(f)}_{L^2}&\lesssim 2^{-3k^+}t^{-1-\beta'}2^{(-1-2\beta')p}\cdot\mathfrak{1}_{2^{2p+q}\gtrsim t^{-1}}\norm{f}_D.  
\end{aligned} 
\end{equation}
\end{proposition}

The proof gives a slightly finer decomposition and makes crucial use of the fact that the $D$ norm of a function bounds its Fourier transform in $L^\infty$ -- see Lemma \ref{lem:ControlLinfty}. We remark that the ideas and techniques underlying Proposition \ref{prop:decay} also apply in a general (i.e.\ non-axisymmetric) setting, where upon inclusion of sufficient powers of the rotation vector field $\Omega$ in the $D$ norm an analogous result can be established.
\begin{remark}\label{rem:decay-sum}
We note that the corresponding $L^\infty$ bound for $I\!I_{k,p,q}(f)$ reads
 \begin{equation}\label{eq:IIdecayLinfty}
 \begin{aligned}
  \norm{I\!I_{k,p,q}(f)}_{L^\infty}&\lesssim 2^{\frac{3}{2}k-3k^+}t^{-1-\beta'}2^{-2\beta'p}\mathfrak{1}_{2^{2p+q}\gtrsim t^{-1}} \norm{f}_D,
 \end{aligned}
 \end{equation}
 and we have summability in $p,q$:
 \begin{equation}
 \sum_{2^{2p+q}\gtrsim t^{-1}}\norm{I\!I_{k,p,q}(f)}_{L^2}\lesssim t^{-\frac{1}{2}}2^{-3k^+}\norm{f}_D, \qquad  \sum_{2^{2p+q}\gtrsim t^{-1}}\norm{I\!I_{k,p,q}(f)}_{L^\infty}\lesssim 2^{\frac{3}{2}k-3k^+}t^{-1}\norm{f}_D.
\end{equation}
\end{remark}
Together with the above bound for $I_{k,p,q}(f)$, we thus conclude that:
\begin{corollary}\label{cor:decay}
 \begin{equation}\label{eq:cor_decay}
 \norm{P_k e^{it\Lambda}f}_{L^\infty}\lesssim \ip{t}^{-1}2^{\frac{3}{2}k-3k^+}\norm{f}_D.
 \end{equation}
 In particular, under the bootstrap assumptions \eqref{eq:btstrap-assump} there holds for a solution $\bm{u}$ to \eqref{eq:EC} that
 \begin{equation}
  \norm{\bm{u}(t)}_{L^\infty}+\norm{\nabla_x\bm{u}(t)}_{L^\infty}\lesssim \eps_1\ip{t}^{-1}.
 \end{equation}
\end{corollary}
We highlight that the decay rate $t^{-1}$ in \eqref{eq:cor_decay} is optimal: For radial $f\in L^2\cap C^0(\R^3)$ there holds that $e^{it\Lambda}f(0)=\frac{\sin t}{t}f(0)$. 
 
After a brief review of some geometric background in the following Section \ref{sec:TS}, we give the proof of these results in Section \ref{sec:decay_proof}.

\subsection{Spanning the tangent space}\label{sec:TS}
The vector fields $S,\Omega$ are related to spherical coordinates as follows: For $(\rho,\theta,\phi)\in[0,\infty)\times [0,2\pi)\times[0,\pi]$ we let
\begin{equation}
 \xi=(\rho\cos\theta\sin\phi,\rho\sin\theta\sin\phi,\rho\cos\phi)=(\rho\cos\theta\sqrt{1-\Lambda^2},\rho\sin\theta\sqrt{1-\Lambda^2},\rho\Lambda),
\end{equation}
with
\begin{equation}
 \Lambda=\cos\phi,\quad \sqrt{1-\Lambda^2}=\sin\phi,
\end{equation}
and have that
\begin{equation}
\begin{split}
\partial_\rho\xi&=\rho^{-1}\xi,\qquad \partial_\theta\xi=\xi_\h^\perp,\\
\partial_\phi\xi&=(\rho\cos\theta\cos\phi,\rho\sin\theta\cos\phi,-\rho\sin\phi)=\frac{\Lambda}{\sqrt{1-\Lambda^2}}\xi_\h-\sqrt{1-\Lambda^2}\rho\,\vec{e}_3,\\
\partial_\Lambda\xi&=-\frac{\Lambda}{\sqrt{1-\Lambda^2}}(\rho\cos\theta,\rho\sin\theta,0)+(0,0,\rho),
\end{split}
\end{equation}
so that $S$ is the radial scaling vector field and $\Omega$ the azimuthal angular derivative, i.e.\
\begin{equation}\label{eq:vfs_coords}
 S_\xi=\xi\cdot\nabla_\xi=\rho\partial_\rho,\qquad \Omega_\xi=\xi_\h^\perp\cdot\nabla_\xi=\partial_\theta.
\end{equation}

To complement these to a full set of vector fields\footnote{\blue{While other choices of complementing vector field are possible, $\Ups$ seems to play a particularly favorable role with respect to the linear and nonlinear structure. In the context of cylindrical symmetry, (a $0$-homogeneous version of) the vertical derivative $\partial_{\xi_3}$ would be another natural choice, but this leads to a degenerate coordinate system near the vertical axis (where $S=\xi_3\partial_{\xi_3}$) and complicates the nonlinear analysis.}} that spans the tangent space at a point in $\R^3$, we define the polar angular derivative by
\begin{equation}\label{eq:ups}
\begin{aligned}
 \Upsilon_\xi:=\partial_\phi=-\sqrt{1-\Lambda^2}\partial_\Lambda&=\frac{\Lambda}{\sqrt{1-\Lambda^2}}(\xi_1\partial_{\xi_1}+\xi_2\partial_{\xi_2})-\vert\xi\vert\sqrt{1-\Lambda^2}\partial_{\xi_3}=\frac{1}{\sqrt{1-\Lambda^2}}\left[\Lambda S_\xi-\abs{\xi}\partial_{\xi_3}\right].
\end{aligned} 
\end{equation}
In terms of the rotation vector fields $\Omega^x_{ab}=x_a\partial_{x_b}-x_b\partial_{x_a}$ introduced in the context of the angular Littlewood-Paley decomposition (Section \ref{ssec:angLP}), this can also be expressed as
\begin{equation}\label{eq:UpsOmegas}
 \Ups_\xi=-\frac{\xi_1}{\abs{\xi_\h}}\Omega^\xi_{13}-\frac{\xi_2}{\abs{\xi_\h}}\Omega^\xi_{23}.
\end{equation}
 
\begin{figure}[h]
 \centering
 \includegraphics[width=0.4\textwidth]{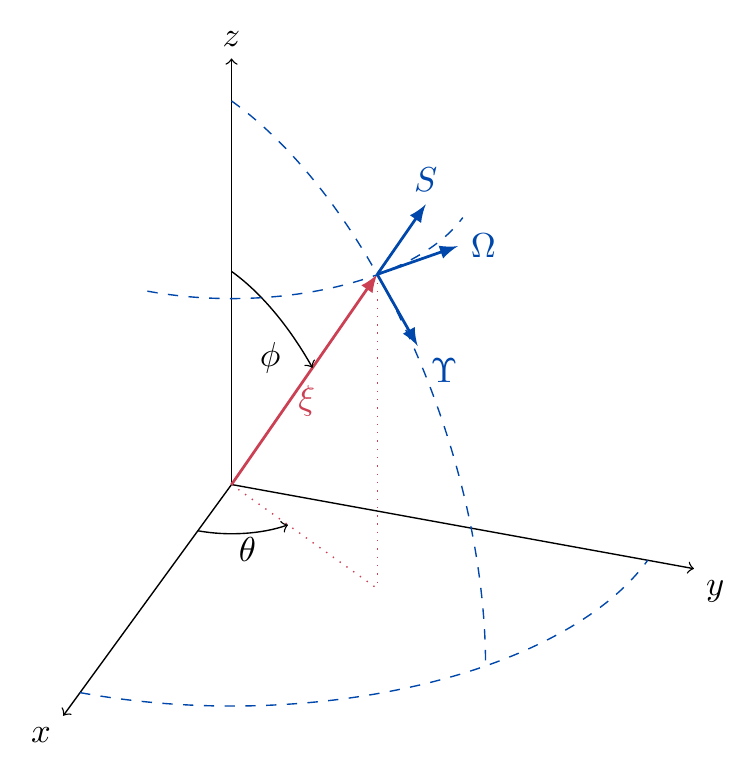}
 \caption{The vector fields $S,\Omega,\Ups$.}
 \label{fig:3dcoords}
\end{figure}
  
\subsection{Proof of Proposition \ref{prop:decay}}\label{sec:decay_proof} 

By scaling and rotation symmetry, we may assume that $k=0$ and ${\bf x}=(x,0,z)$ for some $x\ge0$. If $t2^{2p+q}\le C$, we simply use a crude integration to get
\begin{equation*}
\begin{split}
\vert I_{0,p,q}\vert\lesssim 2^{2p+q}\norm{f}_B.
\end{split}
\end{equation*}
Henceforth we will assume that $2^{-2p-q}\leq C^{-1}t$. We have that 
\begin{equation}
 P_{k,p,q}e^{it\Lambda}f({\bf x})=\int_{\mathbb{R}^3}e^{i\left[t\Lambda(\xi)+\langle {\bf x},\xi\rangle\right]}\widehat{P_{k,p,q}f}(\xi_1,\xi_2,\xi_3)d\xi.
\end{equation}
In spherical coordinates $\xi\mapsto(\rho,\Lambda,\theta)$, with
\begin{equation*}
 d\xi=\rho^2\sin\phi d\theta d\phi d\rho=\rho^2d\theta d\Lambda d\rho,
\end{equation*}
upon integration in $\theta$ we thus need to consider the integral
\begin{equation*}
\begin{split}
I(x,z,t)&:=\int_{\mathbb{R}_+\times[-1,1]}e^{i\left[t\Lambda+\rho\Lambda z\right]}\varphi(2^{-p}\rho\sqrt{1-\Lambda^2})\varphi(2^{-q}\rho\Lambda)\cdot J_0(\rho\sqrt{1-\Lambda^2}x)\cdot \widehat{f}\rho^2\varphi(\rho)d\rho d\Lambda,
\end{split}
\end{equation*}
where $J_0$ denotes the Bessel function of order $0$. 
By standard results on Bessel functions (see e.g.\ \cite[page 338]{Ste1993}), this reduces to studying
\begin{equation*}
\begin{split}
I^\pm(x,z,t)&:=\int_{\mathbb{R}_+\times[-1,1]}e^{i\Psi}\varphi(2^{-p}\rho\sqrt{1-\Lambda^2})\varphi(2^{-q}\rho\Lambda)\cdot H_\pm(\rho\sqrt{1-\Lambda^2}x)\cdot \widehat{f}\rho^2\varphi(\rho)d\rho d\Lambda,\\
\Psi&:=t\Lambda+\rho\left[\Lambda z\pm\sqrt{1-\Lambda^2}x\right],
\end{split}
\end{equation*}
where
\begin{equation}\label{eq:dxBessel}
\left\vert\left(\frac{d}{dx}\right)^aH_\pm(x)\right\vert\lesssim \langle x\rangle^{-\frac{1}{2}-a}.
\end{equation}
We focus on the case with sign $+$; the other estimate is similar. We can compute the gradient
\begin{equation}\label{SPgradients}
\begin{split}
\partial_\Lambda\Psi&=t+\rho\left[z-\frac{\Lambda}{\sqrt{1-\Lambda^2}} x\right],\qquad\partial_\rho\Psi=\Lambda z+\sqrt{1-\Lambda^2}x,\\
\partial_\Lambda^2\Psi&=-\frac{\rho x}{\left[1-\Lambda^2\right]^\frac{3}{2}},\qquad\partial_\Lambda\partial_\rho\Psi=z-\frac{\Lambda}{\sqrt{1-\Lambda^2}} x,\qquad\partial_\rho^2\Psi=0.
\end{split}
\end{equation}

\medskip
For fixed $p,q$ and $0<\kappa<(\beta-\beta^\prime)/20$, we let $\ell_0$ be the greatest integer such that $2^{\ell_0}\le 2^p t\cdot(2^{2p+q}t)^{-\kappa}$, and we decompose
\begin{equation*}
\begin{split}
 \blue{f=R_{\le \ell_0}f+(Id-R_{\leq\ell_0})f,}
\end{split}
\end{equation*}
with \blue{$I^+=I^+_{\leq\ell_0}+I^+_{>\ell_0}$} accordingly. On the one hand, we see that
\begin{equation*}
\begin{split}
 \blue{\Vert I^+_{>\ell_0}\Vert_{L^2}}&\lesssim  \sum_{\ell\ge\ell_0}2^{-(1+\beta)\ell}2^{-\beta p} \norm{P_{0,p}R_\ell f}_{X}\lesssim 2^{-p}t^{-1}\cdot (2^{2p+q}t)^{-\beta+\kappa+\beta\kappa}\cdot 2^{q\beta}\norm{f}_X,
 \end{split}
\end{equation*}
which yields the $L^2$ contribution to \eqref{eq:I-IIdecay}. From now on, together with \eqref{eq:R-Bernstein} we can thus assume that \blue{$f=R_{\le \ell_0}f$} satisfies for all $a\ge0$ and $0\le b\le 2$ that
\begin{equation}\label{FSignalProperty}
\begin{split}
\Vert S^b\partial_\Lambda^a \widehat{f}\Vert_{L^\infty}&\lesssim_a t^a\cdot(2^{2p+q}t)^{-a\kappa}\Vert S^b\widehat{f}\Vert_{L^\infty}\lesssim_a t^a\cdot(2^{2p+q}t)^{-a\kappa}\norm{f}_D.
\end{split}
\end{equation}
We will bound the remaining terms in $L^\infty$, and distinguish cases as follows:

\medskip

{\bf Case 1}:
\begin{equation*}
\begin{split}
0\le x\le C^{-1}t2^{p+q},\qquad\hbox{ and }\qquad \vert z\vert\le C^{-1} t2^{2p}.
\end{split}
\end{equation*}
In these conditions, there holds that
\begin{equation}\label{eq:dlambdaPsi}
\begin{split}
 \vert\partial_\Lambda\Psi\vert\ge t/2,\qquad \abs{\partial_\Lambda^2\Psi}\leq C^{-1}t2^{q-2p},\qquad \abs{\partial_\Lambda^3\Psi}\leq C^{-1}t2^{2q-4p}.
\end{split}
\end{equation}

Using that (with $h=\varphi(2^{-p}\rho\sqrt{1-\Lambda^2})\varphi(2^{-q}\rho\Lambda)H_+(\rho\sqrt{1-\Lambda^2}x)\rho^2\varphi(\rho)$)
\begin{equation*}
-i\iint_{\mathbb{R}_+\times[-1,1]}e^{i\Psi}h\,\partial_\Lambda^n\widehat{f}  d\rho d\Lambda=\iint_{\mathbb{R}_+\times[-1,1]}e^{i\Psi}\frac{h}{\partial_\Lambda\Psi}\partial_\Lambda^{n+1}\widehat{f}  d\rho d\Lambda+\iint_{\mathbb{R}_+\times[-1,1]}e^{i\Psi}\partial_\Lambda\left(\frac{h}{\partial_\Lambda\Psi}\right)\partial_\Lambda^n\widehat{f} d\rho d\Lambda,
\end{equation*}
we integrate by parts at most $N$ times in $\Lambda$, with $N\kappa\ge2$, stopping before if a second derivative does not hit $\widehat{f}$. Note that the boundary terms vanish since we assume $2^{2p}t\gtrsim 1$. Once this is done, we have several types of terms:
\begin{enumerate}[label=(\roman*),wide]
\item if all derivatives hit $\widehat{f}$, a crude estimate using \eqref{FSignalProperty} gives
\begin{equation*}
\begin{split}
\left\vert \iint_{\mathbb{R}_+\times[-1,1]}e^{i\Psi}\frac{h}{(\partial_\Lambda\Psi)^N}\partial_\Lambda^N\widehat{f} d\rho d\Lambda\right\vert &\lesssim 2^{2p+q}t^{-N}\cdot t^N\cdot(2^{2p+q}t)^{-N\kappa}\norm{f}_D\lesssim t^{-1}\cdot (2^{2p+q}t)^{-1}\norm{f}_D.
\end{split}
\end{equation*}

\item if all but one derivative hit $\widehat{f}$, we have a similar estimate since by \eqref{eq:dlambdaPsi} there holds that
\begin{equation*}
\vert \partial_\Lambda h/\partial_\Lambda \Psi\vert+\vert \partial_\Lambda^2\Psi/(\partial_\Lambda\Psi)^2\vert\lesssim 1.
\end{equation*}

\item if two derivatives do not hit $\widehat{f}$, using \eqref{eq:dlambdaPsi} we compute that
\begin{equation*}
\begin{split}
\frac{\vert\partial_\Lambda^3\Psi\vert}{\vert\partial_\Lambda\Psi\vert^3}+\frac{\vert\partial^2_\Lambda\Psi\vert^2}{\vert\partial_\Lambda\Psi\vert^4}+\frac{\vert\partial_\Lambda^2\Psi\vert}{\vert\partial_\Lambda\Psi\vert^2}\frac{\vert\partial_\Lambda h\vert}{\vert\partial_\Lambda\Psi\vert}+\frac{\vert\partial_\Lambda^2h\vert}{\vert\partial_\Lambda\Psi\vert^2}\lesssim (2^{2p+q}t)^{-2}
\end{split}
\end{equation*}
and therefore a crude estimate gives a similar bound.
\end{enumerate}

\medskip

{\bf Case 2}:\footnote{Cases 2 and 3 have already been treated similarly in \cite[Proof of Proposition 4.1]{rotE}.}
\begin{equation*}
\begin{split}
x\ge C^{-1}t2^{p+q}\,\hbox{ and }\,\vert z\vert\le C^{-2}t2^{2p},\qquad\hbox{ or }\qquad  \vert x\vert\le C^{-2}t2^{p+q}\,\hbox{ and }\,\vert z\vert\ge C^{-1}t2^{2p}.
\end{split}
\end{equation*}
Here we have that
\begin{equation*}
\vert\partial_\rho\Psi\vert\gtrsim t2^{2p+q},
\end{equation*}
and we can integrate by parts twice with respect to $\rho$ to obtain after crude integration that
\begin{equation}
 \abs{I^+(x,z,t)}\lesssim (t2^{2p+q})^{-2}\cdot 2^{2p+q}\norm{(1,\partial_\rho,\partial_\rho^2)\widehat{f}}_{L^\infty}\lesssim t^{-2}2^{-2p-q}\norm{(1,S,S^2)f}_D
\end{equation}
upon using Lemma \ref{lem:ControlLinfty}.

\medskip

{\bf Case 3}:
\begin{equation*}
\begin{split}
x\ge C^{-2}t2^{p+q},\,\hbox{ and }\,\vert z\vert\ge C^{-2}t2^{2p}.
\end{split}
\end{equation*}
In these conditions, there holds that
\begin{equation*}
\begin{split}
2^q\vert \partial_\rho\Psi\vert+\vert\partial_\rho\partial_\Lambda\Psi\vert\gtrsim t
\end{split}
\end{equation*}
which follows from
\begin{equation*}
\begin{split}
\Lambda\partial_\Lambda\partial_\rho\Psi-\partial_\rho\Psi&=-\frac{1}{\sqrt{1-\Lambda^2}}x,\qquad\partial_\Lambda\partial_\rho\Psi+\frac{\Lambda}{1-\Lambda^2}\partial_\rho\Psi=\frac{z}{1-\Lambda^2}
\end{split}
\end{equation*}
using the first estimate if $\blue{p\leq -10}$ and the second otherwise.  We now decompose
\begin{equation*}
\begin{split}
I=\sum_{n\ge0}I_n,\qquad
I_n(x,z,t)&:=\int_{\mathbb{R}_+\times[-1,1]}e^{i\Psi}\overline\varphi_{p,q}\cdot H_+(\rho\sqrt{1-\Lambda^2}\vert x\vert)\cdot\ \varphi(2^{-n}\partial_\rho\Psi)\widehat{f}\rho^2\varphi(\rho)d\rho d\Lambda,\\
\overline\varphi_{p,q}&:=\varphi(2^{-p}\rho\sqrt{1-\Lambda^2})\varphi(2^{-q}\rho\Lambda).
\end{split}
\end{equation*}
On the support of $I_0$ we have that $\vert\partial_\Lambda\partial_\rho\Psi\vert\gtrsim t$, and thus with $g(\rho,\Lambda)=\partial_\rho\Psi$
\begin{equation*}
\begin{split}
\vert I_0\vert&\lesssim \Vert \widehat{f}\Vert_{L^\infty}\iint \vert H_+(\rho\sqrt{1-\Lambda^2}x)\vert\cdot \varphi(g)\cdot\rho^2\varphi(\rho)d\rho d\Lambda\lesssim t^{-1}\cdot (2^{2p+q}t)^{-\frac{1}{2}}\Vert \widehat{f}\Vert_{L^\infty}.
\end{split}
\end{equation*}
For $n\ge 1$, we integrate by parts twice in $\rho$ and we find that
\begin{equation*}
\begin{split}
I_n(x,z,t)&:=\int_{\mathbb{R}_+\times[-1,1]}e^{i\Psi}\frac{1}{(\partial_\rho\Psi)^2}\varphi(2^{-n}\partial_\rho\Psi)\partial_\rho^2\left(\overline\varphi_{p,q}\cdot H_+(\rho\sqrt{1-\Lambda^2}x)\cdot\widehat{f}\rho^2\varphi(\rho)\right)d\rho d\Lambda
\end{split}
\end{equation*}
and we hence deduce
\begin{equation*}
\begin{split}
\vert I_n\vert&\lesssim 2^{-2n}\int_{\mathbb{R}_+\times[-1,1]}\vert \widetilde{H}(\rho\sqrt{1-\Lambda^2} x)\vert \cdot\ \varphi(2^{-n}\partial_\rho\Psi)\widetilde{\overline\varphi}_{p,q}F\widetilde{\varphi}(\rho)d\rho d\Lambda,\\
\widetilde{\overline\varphi}_{p,q}&:=\widetilde{\varphi}(2^{-p}\rho\sqrt{1-\Lambda^2})\widetilde{\varphi}(2^{-q}\rho\Lambda),\qquad \widetilde{H}(x):=\vert H_+(x)\vert+\langle x\rangle\vert \frac{dH_+}{dx}(x)\vert+\langle x\rangle^2\vert\frac{d^2H_+}{dx^2}(x)\vert,\\
F&:=\vert \widehat{f}\vert+\vert \partial_\rho\widehat{f}\vert+\vert \partial_\rho^2\widehat{f}\vert,
\end{split}
\end{equation*}
so that
\begin{equation*}
\begin{split}
\sum_{n\geq 1}\vert I_n\vert&\lesssim (t2^{2p+q})^{-\frac{1}{2}}\sum_{q+n\le \ln(t)}2^{-2n}\Vert F\Vert_{L^\infty}\int_{\mathbb{R}_+\times[-1,1]}  \varphi(2^{-n}g)\varphi(t^{-1}\partial_\Lambda g)\widetilde{\overline\varphi}\widetilde{\varphi}(\rho)d\rho d\Lambda\\
&\quad+(t2^{2p+q})^{-\frac{1}{2}}\sum_{q+n\ge \ln(t)}2^{-2n}\Vert F\Vert_{L^\infty}\int_{\mathbb{R}_+\times[-1,1]}  \widetilde{\overline\varphi}\widetilde{\varphi}(\rho)d\rho d\Lambda\\
&\lesssim (t2^{2p+q})^{-\frac{1}{2}}\Vert F\Vert_{L^\infty}\left(\sum_{q+n\le \ln(t)}2^{-n} t^{-1}+\sum_{q+n\ge \ln(t)}2^{-2n}2^{2p+q}\right).
\end{split}
\end{equation*}
Summing and using Lemma \ref{lem:ControlLinfty} finishes the proof.

\section{Integration by parts along vector fields}\label{sec:vfibp}
In this section we develop the formalism for repeated integration by parts along vector fields. To systematically do this, we first address (Section \ref{ssec:vfphase}) some important analytic aspects of the vector fields and how they relate to the bilinear structure of the equations \eqref{eq:EC_disp}. Then we introduce some multiplier classes related to the nonlinearity of \eqref{eq:EC_disp} and study their behavior under the vector fields (Section \ref{ssec:multipliers}). Subsequently we prove bounds for repeated integration by parts along vector fields in Section \ref{ssec:vfibp}.

\subsection{Vector fields and the phase}\label{ssec:vfphase}
We discuss here some aspects related to the interaction of the vector fields $V\in\{S,\Omega\}$ and the phase functions $\Phi_{\mu\nu}$ as in \eqref{eq:def_phi}. We use subscripts to denote the Fourier variable in which a vector field acts, so that
 \begin{equation}
  \Omega_\eta=\eta_\h^\perp\cdot\nabla_{\eta_\h},\quad S_\eta=\eta\cdot\nabla_{\eta}.
 \end{equation}
We begin by recalling that by construction there holds that 
\begin{equation}
 S_\zeta\Lambda(\zeta)=\Omega_\zeta\Lambda(\zeta)=0,
\end{equation}
and thus
\begin{equation}
 V_\eta\Phi_{\mu\nu}(\xi,\eta)=\mu\,V_\eta\Lambda(\xi-\eta),\qquad V\in\{S,\Omega\},\quad\mu,\nu\in\{-,+\}.
\end{equation}
To simplify the notation we will henceforth assume that $\mu=+$ and simply write $\Phi$ for any of the phase functions $\Phi_{\mu\nu}$, when the precise sign combination in \eqref{eq:def_phi} is inconsequential.

The quantity
\begin{equation}\label{eq:def_sigma}
  \bar\sigma\equiv\bar\sigma(\xi,\eta):=\xi_3\eta_\h-\eta_3\xi_\h=-(\xi\times\eta)_\h^\perp,
\end{equation}
will play an important role in our analysis. We note that
\begin{equation}\label{eq:def_sigma2}
 \bar\sigma(\xi,\eta)=\bar\sigma(\xi-\eta,\eta)=-\bar\sigma(\xi,\xi-\eta),
\end{equation}
and $\bar\sigma$ combines horizontal and vertical components of our frequencies, over which we will have precise control (see e.g.\ \eqref{eq:loc_def3}). Moreover, it turns out that $\bar\sigma$ controls the size of vector fields acting on the phase. A direct computation yields:
\begin{lemma}\label{lem:vfsizes-mini}
There holds that
 \begin{equation}\label{eq:vf_sigma_0}
  S_\eta\Phi=\bar\sigma(\xi,\eta)\cdot\frac{\xi_\h-\eta_\h}{\abs{\xi-\eta}^3},\quad \Omega_\eta\Phi=-\bar\sigma(\xi,\eta)\cdot\frac{(\xi_\h-\eta_\h)^\perp}{\abs{\xi-\eta}^3},
 \end{equation}
 and hence
 \begin{equation}\label{eq:vflobound_0}
  \abs{S_\eta\Phi}+\abs{\Omega_\eta\Phi}\sim \frac{\abs{\xi_\h-\eta_\h}}{\abs{\xi-\eta}}\abs{\xi-\eta}^{-2}\abs{\bar\sigma(\xi,\eta)}. 
 \end{equation}
\end{lemma}

\begin{proof}
 See \cite[Lemma 6.1]{rotE}.
\end{proof}
We will make frequent use of this lemma when integrating by parts along vector fields (see Section \ref{ssec:vfibp}).

Another crucial observation is contained in the following proposition: it shows that either we have a lower bound for $\bar\sigma$ (and by \eqref{eq:vflobound_0} thus also for $V_\eta\Phi$), or the phase is relatively large. More precisely, we have shown in \cite[Proposition 6.2]{rotE} that:
\begin{proposition}\label{prop:phasevssigma}
 Assume that $\abs{\Phi}\leq 2^{q_{\max}-10}$. Then in fact $2^{p_{\max}}\sim 1$, and $\abs{\bar\sigma}\gtrsim 2^{q_{\max}}2^{k_{\max}+k_{\min}}$.
\end{proposition}
In practice, this implies that \emph{either we can integrate by parts along a vector field $V\in\{S,\Omega\}$ or perform a normal form.} This may also be viewed as a qualified (and quantified) statement of absence of spacetime resonances. Remarkably, it only makes use of the easily accessible derivatives given by the symmetries, rather than the full gradient.

\subsection{Some multiplier mechanics}\label{ssec:multipliers}
Let us consider the following set of ``elementary'' multipliers
\begin{equation}
\begin{aligned}
 E&:=\left\{\Lambda(\zeta),\sqrt{1-\Lambda^2(\zeta)},\frac{\zeta_{\h}\cdot\zeta_{1,\h}}{\abs{\zeta_{\h}}\abs{\zeta_{1,\h}}},\frac{\zeta_{\h}^\perp\cdot\zeta_{1,\h}}{\abs{\zeta_{\h}}\abs{\zeta_{1,\h}}},\frac{\zeta_2\cdot\zeta_3}{\abs{\zeta_2}\abs{\zeta_3}}\,:\,\zeta,\zeta_1\in\{\xi,\xi-\eta,\eta\},\zeta_2,\zeta_3\in\{\xi-\eta,\eta\}\right\}.
\end{aligned} 
\end{equation}
We note that for $e\in E$ there holds that $\abs{e}\leq 1$, and $E$ is an enlarged version of $\bar{E}$ in \eqref{eq:barE}, that includes the \emph{horizontal} ``angles'' between all frequencies. As we will see, up to products and homogeneity this yields a class of multipliers that is closed under the action of the vector fields $V\in\{S,\Omega\}$, and allows us to express not only multipliers but also dot products with $\bar\sigma$ (as is needed for $V_\eta\Phi$) in terms of building blocks from $E$.

To track the orders of multipliers we encounter, we define the following collections of products of elementary multipliers
\begin{equation}
\begin{aligned}
 E_0\equiv E^{0}_{0}&:=\textnormal{span}_{\R}\left\{\prod_{i=1}^N e_i\,:\, e_i\in E,N\in\N\right\},\\
 E^{a}_{b}&:=\textnormal{span}_{\R}\left\{\abs{\xi-\eta}^{-a}\abs{\eta}^{a}\abs{\xi_\h-\eta_\h}^{-b}\abs{\eta_\h}^{b}\cdot e\,:\,e\in E_0\right\}, \quad a,b\in\Z.
\end{aligned} 
\end{equation}
Furthermore, for $n\in \N$ we let
\begin{equation}
  E(n):=\cup_{\substack{a+b\leq n\\a,b\geq 0}} E^{a}_{b},\quad E(-n):=\cup_{\substack{a+b\geq -n\\a,b\leq 0}} E^{a}_{b},
\end{equation}
which includes all multipliers up to a certain order of homogeneity. 

We remark that $\Phi\in E_0$. From Lemma \ref{lem:ECmult} it follows that the multipliers of the nonlinearity of Euler-Coriolis in dispersive formulation \eqref{eq:EC_disp} are elements of $E_0$ that satisfy certain bounds:
\begin{lemma}\label{lem:ECmult_bds}
 Let $\m$ be a multiplier of the nonlinearity of \eqref{eq:EC_disp}. Then there exists $e\in E_0$ such that
 \begin{equation}
 \m=\abs{\xi}\cdot e. 
\end{equation}
Moreover, we have the bounds
\begin{equation}\label{BoundsOnMAddedBenoit}
 \abs{\m}\cdot\chi_\h\lesssim 2^{k+p_{\max}},\qquad \abs{\m}\cdot\chi\lesssim 2^{k+p_{\max}+q_{\max}}.
\end{equation}
\end{lemma}

As a consequence, it will be important to understand the effect of vector fields on the above classes of multipliers, allowing us to keep track of their orders (e.g.\ when integrating by parts). This is the goal of the following lemma. 
\begin{lemma}\label{lem:vfE}
If $e\in E_0$, then $V_\eta e\in E(1)$ and $V_{\xi-\eta} e\in E(-1)$, and thus
\begin{equation}\label{eq:Ve1}
\begin{aligned}
 \abs{V_\eta e}\cdot\chi_\h(\xi,\eta)&\lesssim 1+2^{k_2-k_1}(1+2^{p_2-p_1}),\\
 \abs{V_{\xi-\eta} e}\cdot\chi_\h(\xi,\eta)&\lesssim 1+2^{k_1-k_2}(1+2^{p_1-p_2}).
\end{aligned} 
\end{equation}
More generally, if $e\in E^{a}_{b}$ 
then 
\begin{equation}\label{eq:Ve2}
\begin{aligned}
 V_{\eta}e\in E^{a}_{b}\cup E^{a+1}_{b} \cup E^{a}_{b+1}:\quad \abs{V_\eta e}\cdot\chi_\h(\xi,\eta)&\lesssim [1+2^{k_2-k_1}(1+2^{p_2-p_1})]\norm{e\chi_\h}_{L^\infty},\\
 V_{\xi-\eta}e\in E^{a}_{b}\cup E^{a-1}_{b} \cup E^{a}_{b-1}:\quad \abs{V_{\xi-\eta} e}\cdot\chi_\h(\xi,\eta)&\lesssim [1+2^{k_1-k_2}(1+2^{p_1-p_2})] \norm{e\chi_\h}_{L^\infty}.
\end{aligned} 
\end{equation}
\end{lemma}
\begin{proof}
 By symmetry it suffices to show the above claims for $V_\eta$, and \eqref{eq:Ve2} follows with analogous computations as for \eqref{eq:Ve1}. 
 
 To establish \eqref{eq:Ve1} we recall that $V_\eta\Lambda(\eta)=0$, and we note that $V_\eta\Lambda(\xi-\eta)=V_\eta\Phi$, so that by \eqref{eq:vf_sigma_0} of Lemma \ref{lem:vfsizes-mini} there holds
 \begin{equation}
  V_\eta(\sqrt{1-\Lambda^2(\xi-\eta)})=-\frac{\Lambda(\xi-\eta)}{\sqrt{1-\Lambda^2(\xi-\eta)}}V_\eta\Lambda(\xi-\eta)=-\frac{\Lambda(\xi-\eta)}{\abs{\xi-\eta}^2}\,\bar\sigma(\xi,\eta)\cdot
  \begin{cases}
   \frac{\xi_\h-\eta_\h}{\abs{\xi_\h-\eta_\h}},&V=S,\\
   -\frac{(\xi_\h-\eta_\h)^\perp}{\abs{\xi_\h-\eta_\h}},&V=\Omega,
  \end{cases}
 \end{equation}
 and we note that $\bar\sigma(\xi,\eta)=\bar\sigma(\xi-\eta,\eta)$. Together with some straightforward, but slightly tedious computations for the ``angles'' (see \cite[Appendix A.2]{rotE}) this is implies the claim.
\end{proof}

As a consequence, we can establish bounds for vector field quotients in cases where we can integrate by parts, i.e.\ when we have a suitable lower bound for $\bar\sigma$:
\begin{lemma}\label{lem:vfQ}
 Assume that $\abs{\bar\sigma}\cdot \chi\gtrsim 2^{k_{\max}+k_{\min}}2^{p_{\max}+q_{\max}}$. 
 Then for $V,V'\in\{S,\Omega\}$ and $n,m\geq 0$ there holds if $\abs{V_\eta\Phi}\gtrsim\abs{\Omega_\eta\Phi}+\abs{S_\eta\Phi}$
 \begin{equation}
  \abs{\frac{(V'_{\xi-\eta})^{n} V_\eta^{m+1}\Phi}{V_\eta\Phi}}\cdot\chi\lesssim [1+2^{k_1-k_2}(1+2^{p_1-p_2})]^n\cdot[1+2^{k_2-k_1}(1+2^{p_2-p_1})]^m.
 \end{equation}
\end{lemma}
\begin{proof}
 A direct computation gives that
 \begin{equation}\label{eq:2ndvfPhi}
 \begin{aligned}
  &S_\eta^2\Phi=S_\eta\Phi\left[3\frac{\eta\cdot(\xi-\eta)}{\abs{\xi-\eta}^2}+2\right]-\frac{\bar\sigma\cdot\xi_\h}{\abs{\xi-\eta}^3},\\
  &\Omega_\eta^2\Phi=3\Omega_\eta\Phi\frac{\in^{ab}\eta_a\xi_b}{\abs{\xi-\eta}^2}-\Lambda(\xi-\eta)\frac{\eta_\h\cdot\xi_\h}{\abs{\xi-\eta}^2},\\
  &S_\eta\Omega_\eta\Phi=\Omega_\eta S_\eta\Phi=S_\eta\Phi\frac{\eta_\h^\perp\cdot\xi_\h}{\abs{\xi-\eta}^2}+\Omega_\eta\Phi[1+2\frac{\eta\cdot(\xi-\eta)}{\abs{\xi-\eta}^2}].
 \end{aligned}
 \end{equation}
 With Lemmas \ref{lem:vfsizes-mini} and \ref{lem:vfE}, by induction we thus see that for $m\in\N$ there exist $e_1,e_2\in E(m)$ and $(e_3^\vartheta)_{\vartheta\in E'}\subset E(m-1)$, where $E'=\left\{\Lambda(\zeta')\frac{\zeta_\h}{\abs{\zeta}},\Lambda(\zeta')\frac{\zeta_\h^\perp}{\abs{\zeta}}:\,\zeta,\zeta' \in\{\xi-\eta,\eta\}\right\}$, such that
 \begin{equation}\label{eq:hivfphase}
  V_\eta^{m+1}\Phi=\Omega_\eta\Phi\cdot e_1+S_\eta\Phi\cdot e_2+\frac{\abs{\eta}}{\abs{\xi-\eta}}\sum_{\vartheta\in E'}\left(\frac{\xi_\h}{\abs{\xi-\eta}}\cdot\vartheta\right)e_3^\vartheta.
 \end{equation}
 Together with the bounds in Lemmas \ref{lem:vfsizes-mini} and \ref{lem:vfE} this proves the claim when $n=0$, since
 \begin{equation}
  \frac{\abs{\Omega_\eta\Phi}+\abs{S_\eta\Phi}}{\abs{V_\eta\Phi}}\lesssim 1,
 \end{equation}
 and since for $\vartheta\in E'$ there holds that
 \begin{equation}
 \begin{aligned}
  \frac{\abs{\eta}}{\abs{\xi-\eta}}\frac{\abs{\xi_\h}}{\abs{\xi-\eta}}\abs{\vartheta}\cdot\abs{V_\eta\Phi}^{-1}&\lesssim 2^{k_2-2k_1}2^{k+p}(2^{q_1}+2^{q_2})(2^{p_1}+2^{p_2})\cdot 2^{2k_1-p_1}2^{-k_{\max}-k_{\min}}2^{-p_{\max}-q_{\max}}\\
  &\lesssim (1+2^{p_2-p_1})(1+2^{k_2-k_1}).
 \end{aligned} 
 \end{equation}

 When $n>0$ the claim follows analogously: we compute that 
 \begin{align}
  \Omega_{\xi-\eta}\Omega_\eta\Phi&=\frac{\xi_3}{\abs{\xi-\eta}}\frac{\abs{\xi_\h-\eta_\h}^2}{\abs{\xi-\eta}^2}-S_\eta\Phi,\qquad
  \Omega_{\xi-\eta}S_\eta\Phi=-S_{\xi-\eta}\Omega_\eta\Phi=\Omega_\eta\Phi, \qquad S_{\xi-\eta}S_\eta\Phi=-S_\eta\Phi,
 \end{align}
 so that with \eqref{eq:hivfphase} there holds that
 \begin{equation}\label{eq:mixhivfphase}
  (V_{\xi-\eta})^n V_\eta^{m+1}\Phi=\Omega_\eta\Phi\cdot \bar{e}_1+S_\eta\Phi\cdot \bar{e}_2+\frac{\abs{\eta}}{\abs{\xi-\eta}}\sum_{\vartheta\in E'}\left(\frac{\xi_\h}{\abs{\xi-\eta}}\cdot\vartheta\right)\bar{e}_3^\vartheta+\frac{\xi_3}{\abs{\xi-\eta}}\frac{\abs{\xi_\h-\eta_\h}^2}{\abs{\xi-\eta}^2}\bar{e}_4,
 \end{equation}
where $\bar{e}_1,\bar{e}_2,\bar{e}_4\in E(m)\cup E(-n)$ and $(\bar{e}_3^\vartheta)_{\vartheta\in E'}\subset E(m-1)\cup E(-n)$. To conclude it suffices to note that
\begin{equation}
 \frac{\abs{\xi_3}}{\abs{\xi-\eta}}\frac{\abs{\xi_\h-\eta_\h}^2}{\abs{\xi-\eta}^2}\cdot\abs{V_\eta\Phi}^{-1}\lesssim 2^{k+q}2^{2p_1-k_1}\cdot 2^{2k_1-p_1}2^{-k_{\max}-k_{\min}}2^{-p_{\max}-q_{\max}}\lesssim 1+2^{k_1-k_2}.
\end{equation}

\end{proof}

\subsection{Integration by parts in bilinear expressions}\label{ssec:vfibp}
Consider now a typical bilinear term $\Q_m(f_1,f_2)$ as in \eqref{eq:def_Qm}:
\begin{equation*}
 \Q_m(f_1,f_2)(s):=\mathcal{F}^{-1}\left(\int_{\R^3}e^{is\Phi(\xi,\eta)}m(\xi,\eta)\widehat{f_1}(s,\xi-\eta)\widehat{f_2}(s,\eta)d\eta\right),
\end{equation*}
 with multiplier in our standard multiplier classes, i.e.\ $ m=\abs{\xi}\cdot e$ for some $e\in E^{a}_{b}$. Our strategy for integration by parts will be to get bounds for $\abs{\bar\sigma}$ via localizations -- firstly in $p,p_j$ and $\ell,\ell_j$, or with more refinement also in $q,q_j$, $j=1,2$ -- from which control of the size of the vector fields applied to the phase follows by \eqref{eq:vflobound_0} in Lemma \ref{lem:vfsizes-mini}. Together with the corresponding quantified control on the inputs this informs us when integration by parts can be carried out advantageously.

We thus decompose
\begin{equation}
 \Q_m(f_1,f_2)=\sum_{\substack{k_j,p_j,\ell_j,\\j=1,2}}\Q_m(P_{k_1,p_1}R_{\ell_1}f_1,P_{k_2,p_2}R_{\ell_2}f_2),
\end{equation}
and using our notation \eqref{eq:loc_def3} for the localizations we have that
\begin{equation}
 \Q_m(P_{k_1,p_1}R_{\ell_1}f_1,P_{k_2,p_2}R_{\ell_2}f_2)=\Q_{m\cdot \chi_\h}(R_{\ell_1}f_1,R_{\ell_2}f_2).
\end{equation}

\subsubsection{Formalism}
We begin by recalling from \cite[Lemma 6.4]{rotE} that we can resolve the action of a vector field in a variable $\zeta\in\{\xi-\eta,\eta\}$ on a function of $\xi-\zeta$ (as we will frequently encounter them when integrating by parts along vector fields in the bilinear expressions \eqref{eq:EC_disp_Duham}) as follows:
\begin{lemma}\label{lem:VFcross}
 Let $\Gamma_{V,\eta}^S,\Gamma_{V,\eta}^\Ups\in E^{1}_{0}$ be defined by
 \begin{equation}
 \begin{aligned}
  \Gamma_{S,\eta}^S&=-\frac{\abs{\eta}}{\abs{\xi-\eta}}\left[\omega_c\sqrt{1-\Lambda^2}(\eta)\sqrt{1-\Lambda^2}(\xi-\eta)+\Lambda(\xi-\eta)\Lambda(\eta)\right],\quad 
  \Gamma_{\Omega,\eta}^S=\frac{\abs{\eta}}{\abs{\xi-\eta}}\omega_s\cdot\sqrt{1-\Lambda^2(\eta)}\sqrt{1-\Lambda^2(\xi-\eta)},\\
  \Gamma_{S,\eta}^\Ups&=-\frac{\abs{\eta}}{\abs{\xi-\eta}}\left[\omega_c\sqrt{1-\Lambda^2}(\eta)\Lambda(\xi-\eta)-\Lambda(\eta)\sqrt{1-\Lambda^2(\xi-\eta)}\right],\quad
  \Gamma_{\Omega,\eta}^\Ups=\frac{\abs{\eta}}{\abs{\xi-\eta}}\omega_s\cdot\sqrt{1-\Lambda^2(\eta)}\Lambda(\xi-\eta),
 \end{aligned}
 \end{equation}
 where 
 \begin{equation}
  \omega_c:=\frac{\eta_\h\cdot(\xi_\h-\eta_\h)}{\abs{\eta_\h}\abs{\xi_\h-\eta_\h}},\quad \omega_s:=\frac{\eta_\h\cdot(\xi_\h-\eta_\h)^\perp}{\abs{\eta_\h}\abs{\xi_\h-\eta_\h}}.
 \end{equation}
 Then on \emph{axisymmetric functions} there holds that 
 \begin{equation}
 \begin{aligned}
  S_\eta &=\Gamma_{S,\eta}^S\cdot S_{\xi-\eta}+\Gamma_{S,\eta}^\Ups\cdot \Ups_{\xi-\eta},\qquad  \Omega_\eta =\Gamma_{\Omega,\eta}^S\cdot S_{\xi-\eta}+\Gamma_{\Omega,\eta}^\Ups\cdot \Ups_{\xi-\eta},
 \end{aligned} 
 \end{equation}
 The symmetric statement holds with the roles of $\eta$ and $\xi-\eta$ exchanged and $\Gamma_{V,\xi-\eta}^W\in E^{-1}_{0}$ for $W\in\{S,\Ups\}$.
\end{lemma}
\begin{proof}
 See \cite[Lemma 6.4]{rotE}.
\end{proof}

To systematically treat several integrations by parts, we introduce the following notations. 
For $\zeta\in\{\eta,\xi-\eta\}$, we consider the following three types of operators, as they naturally arise in integration by parts (according to where the vector fields ``land''):
\begin{equation}\label{eq:adjoints}
\begin{aligned}
 \L_{V,\zeta}^{\id,S}&:=\frac{1}{V_\zeta\Phi},\\
 \L_{V,\zeta}^{W,\id}&:=\frac{1}{V_\zeta\Phi}\Gamma^W_{V,\zeta},\qquad W\in\{S,\Ups\},\\
 \L_{V,\zeta}^{\id,\id}&:=V_\zeta\left(\frac{1}{V_\zeta\Phi}\;\cdot\right)+\frac{1}{V_\zeta\Phi}c_V,\qquad c_S=3,\,\, c_\Omega=0.
\end{aligned}
\end{equation}
The first one corresponds to $V_\zeta$ hitting the input of variable $\zeta$, the second to a ``cross term'' with $W\in\{S,\Ups\}$ and the last to $V_\zeta$ acting on the multiplier itself.

Letting further
\begin{equation}
 \I:=\{(\id,S),(S,\id),(\Ups,\id),(\id,\id)\},
\end{equation}
we can write an integration by parts in e.g.\ $V_\eta$ compactly as
\begin{equation}
 \Q_{m}(F,G)=is^{-1}\sum_{(W,Z)\in\I}\Q_{\L_{V,\eta}^{W,Z}(m)}(WF,ZG),
\end{equation}
and in $V'_{\xi-\eta}$ as
\begin{equation}
 \Q_{m}(F,G)=is^{-1}\sum_{(W,Z)\in\I}\Q_{\L_{V',\xi-\eta}^{W,Z}(m)}(ZF,WG),
\end{equation}
and analogously for several consecutive integrations by parts.\footnote{For example, integrating once along $V_\eta$, then $V'_{\xi-\eta}$, then $V_{\eta}$, gives
\begin{equation}
 \Q_m(F,G)=s^{-3}\sum_{(W_i,Z_i)\in\I,\;1\leq i\leq 3} \Q_{\L_{V,\eta}^{W_3,Z_3}\L_{V',\xi-\eta}^{W_2,Z_2}\L_{V,\eta}^{W_1,Z_1}(m)}(W_3Z_2W_1F,Z_3W_2Z_1G).
\end{equation}
We note that in such an expression, only the $W_i$ may equal $\Ups$, and we have $Z_i\in\{S,\id\}$.}

\subsubsection{Bounds}\label{ssec:vfibp_bds}
The following lemma gives bounds for iterated integration by parts along vector fields (when this is possible).
\begin{lemma}\label{lem:new_ibp}
We have the following bounds for repeated integration by parts:
\begin{enumerate}[wide]
 \item\label{it:ibp_p} Assume the localization parameters are such that $\abs{\bar\sigma}\cdot\chi_\h\gtrsim L_1\gtrsim 2^{k_{\max}+k_{\min}}2^{p_{\max}}$. Then we have for any $N\in\N$ that
 \begin{equation}
 \begin{aligned}
\norm{\mathcal{F}\left(\Q_{m\chi_\h}(R_{\ell_1}f_1,R_{\ell_2}f_2)\right)}_{L^\infty}&\lesssim \norm{m\chi_\h}_{L^\infty}\cdot \left(s^{-1}\cdot 2^{-p_1+2k_1}L_1^{-1}\cdot [1+2^{k_2-k_1+\ell_1}]\right)^N \\
  &\qquad \cdot\norm{P_{k_1,p_1}R_{\ell_1}(1,S)^N f_1}_{L^2}\norm{P_{k_2,p_2}R_{\ell_2}(1,S)^Nf_2}_{L^2}.
 \end{aligned} 
 \end{equation}
 
 \item\label{it:ibp_pq} Assume the localization parameters are such that $\abs{\bar\sigma}\cdot\chi\gtrsim L_2\gtrsim 2^{k_{\max}+k_{\min}}2^{p_{\max}+q_{\max}}$. Then we have for any $N\in\N$ that
 \begin{equation}
 \begin{aligned}
  \norm{\mathcal{F}\left(\Q_{m\chi}(R_{\ell_1}f_1,R_{\ell_2}f_2)\right)}_{L^\infty}&\lesssim \norm{m\chi}_{L^\infty}\cdot \left(s^{-1}\cdot 2^{-p_1+2k_1}L_2^{-1}\cdot [1+2^{k_2-k_1}(2^{q_2-q_1}+2^{\ell_1})]\right)^N \\
  &\qquad \cdot\norm{P_{k_1,p_1,q_1}R_{\ell_1}(1,S)^N f_1}_{L^2}\norm{P_{k_2,p_2,q_2}R_{\ell_2}(1,S)^Nf_2}_{L^2}.
 \end{aligned} 
 \end{equation}
\end{enumerate}
\medskip
These claims hold symmetrically if the variables $\eta$, $\xi-\eta$ are exchanged.
\end{lemma}

We note that the precise estimates are slightly stronger, and in fact show that with a loss of $2^{\ell_1}$ also comes a gain of $c_{p}:=2^{p_1}+2^{p_2}$ resp.\ $c_{pq}:=2^{p_1+q_2}+2^{p_2+q_1}$.

\begin{proof}
 \eqref{it:ibp_p} Let us denote for simplicity of notation $F=R_{\ell_1}f_1$, $G=R_{\ell_2}f_2$. By \eqref{eq:vflobound_0} in Lemma \ref{lem:vfsizes-mini} we may partition
 \begin{equation}\label{eq:vfloc_not}
  1=(1-\chi_{V_\eta})+\chi_{V_\eta},\qquad \chi_{V_\eta}:=(1-\psi)(2^{-p_1+2k_1}L_1^{-1}\cdot V_\eta\Phi),
 \end{equation}
 where $V,V'\in\{S,\Omega\}$ are such that
 \begin{equation}
  \abs{V_\eta\Phi}\cdot\chi_\h\chi_{V_\eta}+\abs{V'_\eta\Phi}\cdot\chi_\h(1-\chi_{V_\eta})\gtrsim 2^{p_1-2k_1}L_1.
 \end{equation}
 We then have that
 \begin{equation}\label{eq:bilin-ibp-split}
  \Q_{m \chi_\h}(F,G)=\Q_{m \chi_\h\chi_{V_\eta}}(F,G)+\Q_{m \chi_\h(1-\chi_{V_\eta})}(F,G),
 \end{equation}
 and can integrate by parts in $V_\eta$ resp.\ $V'_\eta$ in the first resp.\ second term. 
 
 We discuss in detail the first term on the right hand side of \eqref{eq:bilin-ibp-split}, the second being almost identical. We begin with the demonstration of \eqref{it:ibp_p} for $N=1$: Upon integration by parts in $V_\eta$ we have
 \begin{equation}
 \Q_{m \chi_\h\chi_{V_\eta}}(F,G)=is^{-1}\sum_{(W,Z)\in\I}\Q_{\L_{V,\eta}^{W,Z}(m\chi_\h\chi_{V_\eta})}(WF,ZG).
\end{equation}
It suffices to estimate the three types of terms separately:
\begin{itemize}
 \item $(W,Z)=(\id,S)$: Then we have that
 \begin{equation}
  \norm{\mathcal{F}\left(\Q_{\frac{m\chi_\h\chi_{V_\eta}}{V_\eta\Phi}}(F,SG)\right)}_{L^\infty}\lesssim \norm{m\chi_\h}_{L^\infty}2^{-p_1+2k_1}L_1^{-1}\cdot\norm{P_{k_1,p_1}F}_{L^2}\norm{P_{k_2,p_2}SG}_{L^2}.
 \end{equation}
 \item $(W,Z)=(S,\id)$: Here we have that
 \begin{equation}
  \abs{\L_{V,\eta}^{S,\id}(m\chi_\h\chi_{V_\eta})}=\abs{\frac{1}{V_\eta\Phi}\Gamma^S_{V,\eta}\cdot m\chi_\h\chi_{V_\eta}}\lesssim \norm{m\chi_\h}_{L^\infty}\cdot 2^{k_2-k_1}\cdot 2^{-p_1+2k_1}L_1^{-1},
 \end{equation}
 so that
 \begin{equation}
  \norm{\mathcal{F}\left(\Q_{\L_{V,\eta}^{S,\id}(m\chi_\h\chi_{V_\eta})}(SF,G)\right)}_{L^\infty}\lesssim \norm{m\chi_\h}_{L^\infty} \cdot 2^{k_2-k_1}\cdot 2^{-p_1+2k_1}L_1^{-1}\cdot
 \norm{P_{k_1,p_1}SF}_{L^2}\norm{P_{k_2,p_2}G}_{L^2}
 \end{equation}
 \item $(W,Z)=(\Ups,\id)$: Similarly we have
 \begin{equation}
  \abs{\L_{V,\eta}^{\Upsilon,\id}(m\chi_\h\chi_{V_\eta})}=\abs{\frac{1}{V_\eta\Phi}\Gamma^\Ups_{V,\eta}\cdot m\chi_\h\chi_{V_\eta}}\lesssim \norm{m\chi_\h}_{L^\infty}\cdot (2^{p_1}+2^{p_2})\cdot 2^{k_2-k_1}\cdot 2^{-p_1+2k_1}L_1^{-1},
 \end{equation}
 so that
  \begin{equation}
  \norm{\mathcal{F}\left(\Q_{\L_{V,\eta}^{S,\id}(m\chi_\h\chi_{V_\eta})}(\Ups F,G)\right)}_{L^\infty}\lesssim \norm{m\chi_\h}_{L^\infty} \cdot (1+2^{p_2-p_1})2^{k_2+k_1}\cdot L_1^{-1}\cdot 2^{\ell_1}   \cdot \norm{P_{k_1,p_1}F}_{L^2}\norm{P_{k_2,p_2}G}_{L^2},
 \end{equation}
 having used that by \eqref{eq:UpsOmegas} and \eqref{eq:R-Bernstein} there holds
 \begin{equation}
  \norm{\Ups R_{\ell_1}f_1}_{L^2}\lesssim 2^{\ell_1}\norm{R_{\ell_1}f_1}_{L^2}.
 \end{equation}

 \item $(W,Z)=(\id,\id)$: Here we have by Lemma \ref{lem:vfQ} (and by direct computation on the localizations $\chi_\h$ and $\chi_{V_\eta}$) that
 \begin{equation}
  \abs{\L_{V,\eta}^{\id,\id}(m\chi_\h\chi_{V_\eta})}\lesssim \norm{m\chi_\h}_{L^\infty}[1+2^{k_2-k_1}(1+2^{p_2-p_1})]\cdot 2^{-p_1+2k_1}L_1^{-1},
 \end{equation}
 so that
 \begin{equation}
  \norm{\mathcal{F}\left(\Q_{\L_{V,\eta}^{\id,\id}(m\chi_\h\chi_{V_\eta})}(F,G)\right)}_{L^\infty}\lesssim \norm{m\chi_\h}_{L^\infty}\cdot 2^{-p_1+2k_1}L_1^{-1}\cdot [1+2^{k_2-k_1}(1+2^{p_2-p_1})]\cdot \norm{P_{k_1,p_1}F}_{L^2}\norm{P_{k_2,p_2}G}_{L^2}.
 \end{equation}

\end{itemize}
Since $\ell_j+p_j\geq 0$, by iteration and Lemma \ref{lem:vfQ} we obtain the claim \eqref{it:ibp_p} for general $N\in\N$.

The proof of \eqref{it:ibp_pq} is similar. The only difference arises from the case where the vector fields land on the localization functions $\chi$. Here we observe that
\begin{equation}
 \abs{V_\eta(\varphi_{k_1,p_1,q_1}(\xi-\eta))}\lesssim \left(1+2^{k_2-k_1}(1+2^{p_2-p_1}+2^{q_2-q_1})\right)\cdot \widetilde\varphi_{k_1,p_1,q_1}(\xi-\eta).
\end{equation}
Hence \eqref{it:ibp_pq} is proved.
\end{proof}

\subsubsection{A ``vertical'' variant}\label{ssec:D3ibp}
When no localizations in $\Lambda$ are involved, a zero homogeneous version of the vertical derivative can also be useful for iterated integrations by parts: We let
 \begin{equation}
  D_3^\eta:=\abs{\eta}\partial_{\eta_3}=\Lambda(\eta)S_\eta-\sqrt{1-\Lambda^2(\eta)}\Ups_\eta,
\end{equation}
and note that
\begin{equation}\label{eq:D3basics}
 D_3^\eta\left(\Lambda(\eta)\right)=1-\Lambda^2(\eta),\qquad D_3^\eta\left(\sqrt{1-\Lambda^2(\eta)}\right)=-\Lambda(\eta)\sqrt{1-\Lambda^2(\eta)},
\end{equation}
as well as
\begin{equation}
 D_3^\eta\left(\Lambda(\xi-\eta)\right)=-\frac{\abs{\eta}}{\abs{\xi-\eta}}(1-\Lambda^2(\xi-\eta)),\qquad D_3^\eta\left(\sqrt{1-\Lambda^2(\xi-\eta)}\right)=\frac{\abs{\eta}}{\abs{\xi-\eta}}\Lambda(\xi-\eta)\sqrt{1-\Lambda^2(\xi-\eta)}.
\end{equation}
Thus 
\begin{equation}
\begin{aligned}
 D_3^\eta \varphi(2^{-p_2}\sqrt{1-\Lambda^2}(\eta))&=-2^{-p_2}\Lambda(\eta)\sqrt{1-\Lambda^2(\eta)}\varphi'(2^{-p_2}\sqrt{1-\Lambda^2}(\eta))\\
 &= -\Lambda(\eta)\cdot \widetilde{\varphi}(2^{-p_2}\sqrt{1-\Lambda^2}(\eta)),\\
 D_3^\eta \varphi(2^{-p_1}\sqrt{1-\Lambda^2}(\xi-\eta))&=2^{-p_1}\frac{\abs{\eta}}{\abs{\xi-\eta}}\Lambda(\xi-\eta)\sqrt{1-\Lambda^2(\xi-\eta)}\varphi'(2^{-p_1}\sqrt{1-\Lambda^2}(\xi-\eta))\\
 &= \Lambda(\xi-\eta)\cdot\frac{\abs{\eta}}{\abs{\xi-\eta}} \widetilde{\varphi}(2^{-p_1}\sqrt{1-\Lambda^2}(\xi-\eta)).
\end{aligned} 
\end{equation}
Together with $D_3^\eta\abs{\eta}=\abs{\eta}\Lambda(\eta)$, $D_3^\eta\abs{\xi-\eta}=-\abs{\eta}\Lambda(\xi-\eta)$ and the fact that $D_3^\eta=-\frac{\abs{\eta}}{\abs{\xi-\eta}}D_3^{\xi-\eta}$ we thus have that
\begin{equation}
\begin{aligned}
 D_3^\eta \,\mathcal{F}\left(P_{k_2,p_2}R_{\ell_2} G\right)(\eta)&\sim \Lambda(\eta) \mathcal{F}\left(P_{k_2,p_2}R_{\ell_2} (1,S)G\right)(\eta)+ 2^{\ell_2+p_2}\mathcal{F}\left(P_{k_2,p_2}R_{\ell_2} G\right)(\eta),\\
 D_3^\eta\, \mathcal{F}\left(P_{k_1,p_1}R_{\ell_1} F\right)(\xi-\eta)&\sim 2^{k_2-k_1}\left[ \Lambda(\xi-\eta) \mathcal{F}\left(P_{k_1,p_1}R_{\ell_1} (1,S)F\right)(\xi-\eta)+ 2^{\ell_1+p_1}\mathcal{F}\left(P_{k_1,p_1}R_{\ell_1} F\right)(\xi-\eta)\right].
\end{aligned} 
\end{equation}

To make use of this in an iterated integration by parts we also need to control $D_3^\eta\Phi$.
From \eqref{eq:D3basics} we see iteratively that for any $M\in\N$ there holds
\begin{equation}\label{eq:D3basics2}
 \abs{(D_3^\eta)^M\Lambda(\eta)}\lesssim 1-\Lambda^2(\eta),\qquad \abs{(D_3^\eta)^M\Lambda(\xi-\eta)}\lesssim \frac{\abs{\eta}^M}{\abs{\xi-\eta}^M}(1-\Lambda^2(\xi-\eta)).
\end{equation}
\subsubsection*{Example}
In the particular case where $2^{k_2}\sim 2^{k_1}$ and $0\gg p_2\gg p_1$, we have that 
$\abs{D_3^\eta\Phi}\sim 2^{2p_2}$, and with \eqref{eq:D3basics2} we see analogously as in Section \ref{ssec:vfibp_bds} that repeated integration by parts along $D_3^\eta$ in a term $\Q_{\m}(P_{k_1,p_1}R_{\ell_1} F,P_{k_2,p_2}R_{\ell_2} G)$ is beneficial if
\begin{equation}\label{eq:D3ibpcond}
 2^{-2p_2}\left(1+2^{\ell_1+p_1}+2^{\ell_2+p_2}\right)<s^{1-\delta}.
\end{equation}

\subsubsection{A preliminary lemma to organize cases}
Since we have a multitude of parameters that govern the losses and gains when integrating by parts as in Lemma \ref{lem:new_ibp}, it is useful to get some overview of natural restrictions. To guide the organization of cases later on we will make use of the following result:
\begin{lemma}\label{lem:gapp-cases}
 Assume that $p\leq \min\{p_1,p_2\}-10$. Then on the support of $\chi_\h$ there holds that $p+k<p_1+k_1-4$, and thus $p_2+k_2-2\leq p_1+k_1\leq p_2+k_2+2$. Moreover, either one of the following options holds:
 \begin{enumerate}
  \item\label{it:gapp-op1} $\abs{k_1-k_2}\leq 4$, and thus also $\abs{p_1-p_2}\leq 6$,
  \item\label{it:gapp-op2} $k_2<k_1-4$ then $\abs{k-k_1}\leq 2$ and $p_1\leq p_2-2$, so that $p\leq p_1-10\leq p_2-12$.
  \item $k_1<k_2-4$ then $\abs{k-k_2}\leq 2$ and $p_2\leq p_1-2$, so that $p\leq p_2-10\leq p_1-12$.
 \end{enumerate}
\end{lemma}

\begin{remark}\label{lem:gapp-casesRem}
 We comment on a few points:
 \begin{enumerate}
  \item The analogous result applies with the roles of $p,p_i$ permuted.
  \item The analogous results hold in the variables $q,q_i$ on the support of $\chi$ in case of a gap in $q_{\min}\leq q_{\max}-10$.
 \end{enumerate}
\end{remark}
\blue{\textbf{Notation.} Since the constants involved here and in many future, similar case by case analyses are independent of the other important parameters in our proofs, we will use the slightly less formal $\ll$, $\sim$, etc. Since the decisive scales are usually given in terms of parameters in dyadic decompositions, to unburden the notation we will use the same symbols to denote both multiplicative bounds (resp.\ equivalences) at the level of the dyadic scales $2^{n}$, as well as additive bounds at the level of the parameter $n\in\Z$, where the distinction is clear from the context. For example, we will refer to the assumption of Lemma \ref{lem:gapp-cases} as $p\ll \min\{p_1,p_2\}$, and will take $p\sim 0$ as equivalent to $2^p\sim 1$, namely that there exist $C\in\N$ such that $-C<p<C$.
}
\medskip

\begin{figure*}[h]
    \centering
    \begin{subfigure}[t]{0.5\textwidth}
        \centering
        \includegraphics[height=4cm]{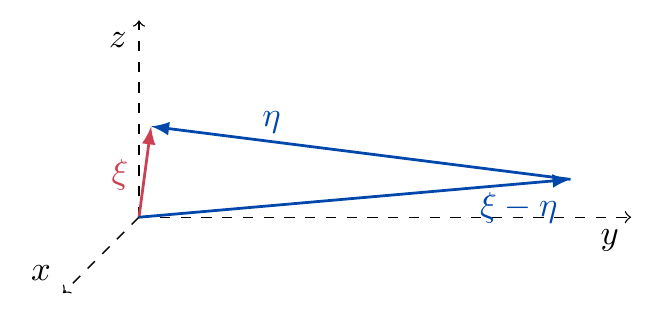}
        \caption*{Case \eqref{it:gapp-op1}}
    \end{subfigure}%
    ~ 
    \begin{subfigure}[t]{0.5\textwidth}
        \centering
        \includegraphics[height=7cm]{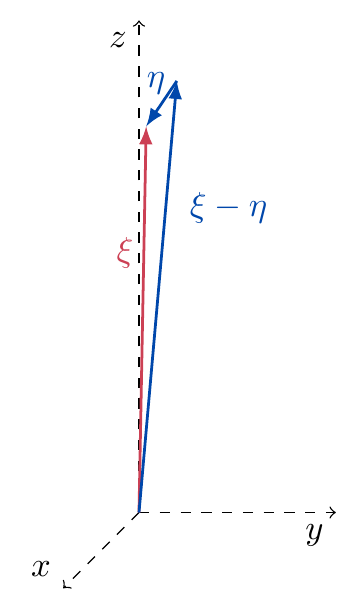}
        \caption*{Case \eqref{it:gapp-op2}}
    \end{subfigure}
    \caption{Exemplary illustration of the scenarios of Lemma \ref{lem:gapp-cases} in Cartesian coordinates.}\label{fig:gapp}
\end{figure*}

For the proof it is convenient to visualize the triangle of frequencies $\xi,\xi-\eta,\eta$ -- see also Figure \ref{fig:gapp} for illustration.
\begin{proof}[Proof of Lemma \ref{lem:gapp-cases}]
 Consider $(\xi,\eta)\in\textnormal{supp}(\chi_\h)$. Let $p\leq \min\{p_1,p_2\}-10$, and assume for the sake of contradiction that $ p+k\geq p_1+k_1-4$. Then from $\eta_\h=\xi_h-(\xi_\h-\eta_h)$ we have that $p_2+k_2\leq p+k+6$, and it also follows that $k_1\leq p-p_1+k+4\leq k-6$, and hence $k-2\leq k_2\leq k+2$ since $\eta=\xi-(\xi-\eta)$. But then we arrive at the contradiction that $p\geq p_2+k_2-k-6\geq p_2-8$. Hence we conclude that $p+k<p_1+k_1-4$, and thus $p_1+k_1\in[p_2+k_2-2,p_2+k_2+2]$.
 
 Moreover, if $\abs{k_1-k_2}\leq 4$, then it follows that $\abs{p_1-p_2}\leq 6$. Finally, if $k_2< k_1-4$, then $\abs{k-k_1}\leq 2$ and  thus $p_1-p_2\leq k_2-k_1+2\leq -2$, so that $p\leq p_1-10\leq p_2-12$. The third statement is the symmetric version upon exchanging the roles of $\xi-\eta$ and $\eta$.
\end{proof}

\subsection{Remark on normal forms}\label{ssec:nfs}
In the bilinear expressions we encounter we will also perform normal forms. For a parameter $\lambda>0$ to be chosen we decompose the multiplier into ``resonant'' and ``non-resonant'' parts
\begin{equation}\label{eq:mdecompNF}
 \m(\xi,\eta)=\psi(\lambda^{-1}\Phi)\mathfrak{m}(\xi,\eta)+(1-\psi(\lambda^{-1}\Phi))\m(\xi,\eta)=:\mathfrak{m}^{res}(\xi,\eta)+\m^{nr}(\xi,\eta),
\end{equation}
and correspondingly have that
\begin{equation}
 \B_\m(F_1,F_2)=\B_{\m^{res}}(F_1,F_2)+\B_{\m^{nr}}(F_1,F_2).
\end{equation}
A direct integration by parts in time yields that
\begin{equation}\label{mdecompNFNRNorms}
 \norm{P_{k,p}\B_{\m^{nr}}(F_1,F_2)}_{L^2}\lesssim \norm{P_{k,p}\Q_{\Phi^{-1}\m^{nr}}(F_1,F_2)}_{L^2}+\norm{P_{k,p}\B_{\Phi^{-1}\m^{nr}}(\partial_tF_1,F_2)}_{L^2}+\norm{P_{k,p}\B_{\Phi^{-1}\m^{nr}}(F_1,\partial_tF_2)}_{L^2}.
\end{equation}
\begin{lemma}\label{lem:nfs}
 Let $\lambda>0$ and $G_j=P_{k_j,p_j}G_j$.  We have the following bounds:
 \begin{enumerate}
  \item\label{it:NF-bd1}   The non-resonant part satisfies
  \begin{equation}\label{eq:NF-bd1}
   \norm{P_{k,p}\Q_{\Phi^{-1}\m^{nr}}(G_1,G_2)}_{L^2}\lesssim 2^{k+p_{\max}}\lambda^{-1}\cdot \Sz \norm{G_1}_{L^2}\norm{G_2}_{L^2}.
  \end{equation}
  
  \item\label{it:NF-bd2} If we can choose $\lambda>0$ such that $\abs{\Phi\chi_\h}\geq\lambda\gtrsim 1$, then we have that $\m^{res}=0$ and thus $\m=\m^{nr}$, and in addition to \eqref{eq:NF-bd1} there holds the alternative bound
  \begin{equation}\label{eq:NF-bd2}
   \norm{P_{k,p}\Q_{\Phi^{-1}\m^{nr}}(G_1,G_2)}_{L^2}\lesssim 2^{k+p_{\max}}\cdot \min\{\norm{e^{it\Lambda} G_1}_{L^\infty}\norm{G_2}_{L^2},\norm{G_1}_{L^2}\norm{e^{it\Lambda}G_2}_{L^\infty}\}.
  \end{equation}
  
  \item\label{it:NF-bd34} If there holds that $\abs{\partial_{\eta_3}\Phi\,\chi_\h}\gtrsim L>0$, then we also have the following set size gains:
  \begin{equation}\label{eq:NF-bd3}
   \norm{P_{k,p}\Q_{\m^{res}}(G_1,G_2)}_{L^2}\lesssim 2^{k+p_{\max}}\lambda^{\frac{1}{2}} L^{-\frac{1}{2}}\cdot \min\{2^{k_1+p_1},2^{k_2+p_2}\}\norm{G_1}_{L^2}\norm{G_2}_{L^2},
  \end{equation}
  and
  \begin{equation}\label{eq:NF-bd4}
   \norm{P_{k,p}\Q_{\Phi^{-1}\m^{nr}}(G_1,G_2)}_{L^2}\lesssim \abs{\log\lambda}2^{k+p_{\max}}\lambda^{-\frac{1}{2}} L^{-\frac{1}{2}}\cdot \min\{2^{k_1+p_1},2^{k_2+p_2}\}\norm{G_1}_{L^2}\norm{G_2}_{L^2}.
  \end{equation}

  \item The analogous bounds hold when additional localizations in $q, q_j$, $j=1,2$ are considered.
 \end{enumerate}
\end{lemma}

\begin{proof}
 The first claim \eqref{it:NF-bd1} follows from Lemma \ref{lem:set_gain} and the fact that $\abs{\Phi^{-1}\m^{nr}\chi_\h}\lesssim 2^{k+p_{\max}}\lambda^{-1}$. For \eqref{it:NF-bd2} it suffices to notice that under these assumptions by Lemma \ref{lem:phasesymb_bd} there holds that $\norm{\Phi^{-1}\m^{nr}}_{\W_\h}\lesssim 2^{k+p_{\max}}$.
 
 Finally, \eqref{eq:NF-bd3} follows with the improved set size gain of Lemma \ref{lem:set_gain2}. To obtain \eqref{eq:NF-bd4}, we further decompose
 \begin{equation}
  \m^{nr}(\xi,\eta)=\sum_{r\geq 1}\m_r(\xi,\eta),\qquad \m_r(\xi,\eta)=\varphi(2^{-r}\lambda^{-1}\Phi(\xi,\eta))\m^{nr}(\xi,\eta),
 \end{equation}
 and correspondingly $\Q_{\Phi^{-1}\m^{nr}}(G_1,G_2)=\sum_{r\geq 1}\Q_{\Phi^{-1}\m_r}(G_1,G_2)$. For these we invoke again Lemma \ref{lem:set_gain2} and note that since $\abs{\Phi}\lesssim 1$ there are at most $\abs{\log\lambda}$ terms.
\end{proof}

\section{Bounds for $\partial_t S^Nf$ in $L^2$}\label{sec:dtfbds}
We have the following estimates for the time derivative of the dispersive unknowns. We note that in view of the fact that these are bilinear expressions in $3d$, without further removal of resonant parts this is the fastest decay (up to minor losses) one can hope for.
\begin{lemma}\label{lem:dtfinL2}
Let $f$ be a dispersive unknown in Euler-Coriolis, and assume the bootstrap assumptions \eqref{eq:btstrap-assump}.

Then there exists $0<\gamma\ll \beta$ such that for $m\in\N$ and $t\in [2^m,2^{m+1})\cap [0,T]$ there holds that 
 \begin{equation}
  \norm{\partial_t P_k S^bf(t)}_{L^2}\lesssim 2^{\frac{k}{2}-k^+}\cdot 2^{-\frac{3}{2}m+\gamma m}\cdot \eps_1^2, \qquad 0\leq b\leq N.
 \end{equation}

\end{lemma}

\begin{proof}
We know that $\partial_t P_k S^bf$, $0\leq b\leq N$ is a sum of terms of the form $\sum_{b_1+b_2\leq N} P_k \Q_\m(S^{b_1}F_1,S^{b_2}F_2)$ with $\m$ a multiplier as in Lemma \ref{lem:ECmult} and $F_j\in\{\U_+,\U_-\}$ dispersive unknowns, $j=1,2$, so it suffices to bound such expressions in $L^2$. Localizing the inputs in frequency we have that
\begin{equation}
 \norm{P_k \Q_\m(S^{b_1}F_1,S^{b_2}F_2)}_{L^2}\lesssim \sum_{k_1,k_2\in\Z}\norm{P_k\Q_\m(P_{k_1}S^{b_1}F_1,P_{k_2}S^{b_2}F_2)}_{L^2}.
\end{equation}
By the energy estimates \eqref{eq:interpol_energy} and direct bounds we have that
\begin{equation}
 \norm{P_k\Q_\m(P_{k_1}S^{b_1}F_1,P_{k_2}S^{b_2}F_2)}_{L^2}\lesssim 2^{k+\frac{3}{2}k_{\min}}\cdot 2^{-N_0(k_1^++k_2^+)}\cdot\norm{S^{b_1}F_1}_{H^{N_0}}\norm{S^{b_2}F_2}_{H^{N_0}},
\end{equation}
so the claim follows provided that $k_{\max}\geq 2N_0^{-1} m$, or if $k_{\min}\leq -2m$. With $\delta_0= 2N_0^{-1}$ we will thus assume that $-2m< k,k_1,k_2<\delta_0 m$.

Localizing further in $p,p_i$ and $\ell_i$, $i=1,2$, and writing $f_i=P_{k_i,p_i}R_{\ell_i}S^{b_i}F_i$ for simplicity of notation, we can further assume that $p,p_i\geq -2m$ and $\ell_i\leq 2m$, since
\begin{equation}
\begin{aligned}
 \norm{P_{k,p}\Q_\m(f_1,f_2)}_{L^2}\lesssim 2^{k+\frac{3}{2}k_{\max}+p_{\min}} &\min\{2^{p_1}\norm{f_1}_{B},2^{-\ell_1}\norm{f_1}_X\} \cdot \min\{2^{p_2}\norm{f_2}_B,2^{-\ell_2}\norm{f_2}_X\}.
\end{aligned} 
\end{equation}
Assuming without loss of generality that $k_2\leq k_1$ (so that $2^{k_{\max}+k_{\min}}\sim 2^{k+k_2}$), it thus suffices to show that
\begin{equation}\label{eq:dtfclaim}
 \norm{P_{k,p}\Q_\m(f_1,f_2)}_{L^2}\lesssim 2^{-\frac{3}{2}m+\frac{\gamma}{2} m}\cdot\eps_1^2.
\end{equation}
The basic strategy for this will be to either repeatedly integrate by parts as in Lemma \ref{lem:new_ibp}, or to use the $X$ and $B$ norm bounds. For this it is useful to distinguish cases based on the localizations and whether there are size gaps, first in $p$, $p_i$ (Case 1), then in $q$, $q_i$ (Case 2), since this gives lower bounds for $\abs{\bar\sigma}$ and thus for $V_\zeta\Phi$, $\zeta\in\{\eta,\xi-\eta\}$, as per Lemma \ref{lem:vfsizes-mini}. At the end (Case 3) this leaves us with the setting where these localizations are comparable.

\subsubsection*{Case 1: Gap in $p$}
Here we assume that $p_{\min}\ll p_{\max}$. By Lemma \ref{lem:vfsizes-mini} we have $\abs{\bar\sigma}\gtrsim 2^{p_{\max}}2^{k+k_2}$, and we may choose $V\in\{S,\Omega\}$ such that
\begin{equation}
 \abs{V_{\xi-\eta}\Phi}\sim 2^{p_2-2k_2}2^{p_{\max}+k+k_2}=2^{p_2+p_{\max}}2^{k-k_2},
\end{equation}
where we used that by convention $k_2\leq k_1$, so that $k_{\min}\in\{k,k_2\}$.

\subsubsection*{Case 1.1: $k_2=k_{\min}$}
Here we have that $2^k\sim 2^{k_1}$.

Then from Lemma \ref{lem:new_ibp}\eqref{it:ibp_p} we see that repeated integration by parts along $V_{\xi-\eta}$ gives for $K\in\N$ that
\begin{equation}
\begin{aligned}
 \norm{P_{k,p}\Q_\m(f_1,f_2)}_{L^2}&\lesssim 2^{k+\frac{3}{2}k_2}\left(2^{-m}2^{-p_2-p_{\max}}\cdot 2^{k_2-k}\cdot2^{k_1-k_2}[1+2^{\ell_2}]\right)^K\cdot \norm{(1,S)^Kf_1}_{L^2}\norm{(1,S)^Kf_2}_{L^2}\\
 &\lesssim 2^{k+\frac{3}{2}k_2}\left(2^{-m}2^{-p_2-p_{\max}}2^{\ell_2}\right)^K\cdot\eps_1^2.
\end{aligned} 
\end{equation}
Choosing $K =O(M)\gg 1$ yields the claim, provided that for $\delta=2K^{-1}=O(M^{-1})$ we have (since $\ell_2+p_2\geq 0$)
\begin{equation}
 -p_{\max}+2\ell_2\leq (1-\delta)m.
\end{equation}
If on the other hand $\ell_2>(1-\delta)\frac{m}{2}+\frac{p_{\max}}{2}$, using an $L^\infty\times L^2$ estimate and \eqref{BoundsOnMAddedBenoit} and Corollary \ref{cor:extrapol_decay} with $\kappa\ll\beta$ if $f_1$ has more than $N-3$ vector fields, we have that
\begin{equation}\label{eq:pgap_bd}
\begin{aligned}
 \norm{P_{k,p}\Q_\m(f_1,f_2)}_{L^2}&\lesssim 2^{k+p_{\max}} \norm{e^{it\Lambda}f_1}_{L^\infty}\cdot 2^{-\ell_2}\norm{f_2}_{X}\lesssim  2^{k+p_{\max}}2^{-m+\kappa m}\cdot 2^{-\frac{m}{2}+\delta\frac{m}{2}-\frac{p_{\max}}{2}}\eps_1^2\\
 &\lesssim  2^k2^{-\frac{3}{2}m+(\frac{\delta}{2}+\kappa) m}\cdot \eps_1^2.
\end{aligned} 
\end{equation}

\subsubsection*{Case 1.2: $k=k_{\min}$} 
We now aim to show \eqref{eq:dtfclaim} in case $k=k_{\min}$, where in particular $2^{k_1}\sim 2^{k_2}$. This can be done as Case 1.1 above, with the difference that now there may be a loss in $k$. This however, is recovered directly by the multiplier $\m$, so we can proceed in close parallel to Case 1.1.

By Lemma \ref{lem:new_ibp} integration by parts is feasible provided that
\begin{equation}
 2^{-m}\cdot 2^{-p_2-p_{\max}-k+k_2}\cdot (1+2^{p_1-p_2}+2^{\ell_2})<2^{-\delta m},
\end{equation}
which can be guaranteed by requiring that $-p_{\max}+2\ell_2-k+k_2\leq (1-\delta)m$. If on the other hand $-p_{\max}+2\ell_2-k+k_2> (1-\delta)m$, then as in \eqref{eq:pgap_bd} we have
\begin{equation}\label{eq:pgap_bd-2}
\begin{aligned}
 \norm{P_{k,p}\Q_\m(f_1,f_2)}_{L^2}&\lesssim 2^{k+p_{\max}} \norm{e^{it\Lambda}f_1}_{L^\infty}\cdot 2^{-\ell_2}\norm{f_2}_{X}\lesssim 2^{k+p_{\max}}2^{-m+\kappa m}2^{-\frac{1-\delta}{2}m-\frac{k}{2}+\frac{k_2}{2}}\cdot \eps_1^2\\
 &\lesssim 2^{\frac{k+p_{\max}}{2}+\frac{k_2}{2}}2^{-\frac{3}{2}m+(\kappa+\frac{\delta}{2}) m}\eps^2 \lesssim 2^{\frac{k}{2}}2^{-\frac{3}{2}m+(\frac{\delta}{2}+\kappa+\delta_0) m}\eps_1^2.
\end{aligned} 
\end{equation}

\medskip
We may henceforth assume that $2^{p_{\max}}\sim 2^{p_{\min}}$.

\subsubsection*{Case 2: Gap in $q$}
Now we localize further in $q,q_i$, and  write $g_i=P_{k,p_i,q_i}R_{\ell_i}f_i$, $i=1,2$. Then by $B$ norm estimates and the set size bound $\Sz\lesssim 2^{\frac{3}{2}k_{\max}}2^{p+\frac{q}{2}}$ we can assume that $q,q_i\geq -4m$, since if $q_{\min}<-4m$ there holds
\begin{equation}
\begin{aligned}
 \norm{P_{k,p,q}\Q_\m(g_1,g_2)}_{L^2}&\lesssim 2^{\frac{5}{2}k+p+\frac{q}{2}}\cdot 2^{p_1+\frac{q_1}{2}}\norm{g_1}_B 2^{p_2+\frac{q_2}{2}}\norm{g_2}_B\lesssim 2^{\frac{5}{2}k}2^{-2m}\eps_1^2.
\end{aligned} 
\end{equation}
Assuming now that $q_{\min}\ll q_{\max}$, we have that $2^{p_{\max}}\sim 1$, and thus by the previous case that $2^{p_{\max}}\sim 2^{p_{\min}}\sim 1$. Analogously to before we choose $V\in\{S,\Omega\}$ such that 
\begin{equation}
 \abs{V_{\xi-\eta}\Phi}\sim 2^{-2k_2}2^{q_{\max}+k+k_2}=2^{q_{\max}}2^{k-k_2}.
\end{equation}

\subsubsection*{Case 2.1: $k_2=k_{\min}$} 
By Lemma \ref{lem:new_ibp}\eqref{it:ibp_pq}, repeated integration by parts along $V_{\xi-\eta}$ gives 
\begin{equation}
\begin{aligned}
 \norm{P_{k,p,q}\Q_\m(g_1,g_2)}_{L^2}&\lesssim 2^{k+\frac{3}{2}k_2}\left(2^{-m}\cdot 2^{-q_{\max}}2^{k_2-k}\cdot 2^{k-k_2}[1+2^{q_1-q_2}+2^{\ell_2}]\right)^K\cdot \norm{(1,S)^Kg_1}_{L^2}\norm{(1,S)^Kg_2}_{L^2}\\
 &\lesssim 2^{k+\frac{3}{2}k_2}\left(2^{-m}2^{-q_{\max}}\max\{2^{q_1-q_2},2^{\ell_2}\}\right)^K\cdot\eps_1^2,
\end{aligned} 
\end{equation}
and the claim follows provided that
\begin{equation}
 2^{-m}2^{-q_{\max}}\max\{2^{q_1-q_2},2^{\ell_2}\}<2^{-\delta m}.
\end{equation}
Case 2.1a: In case $\ell_2\leq q_{1}-q_2$ this is satisfied if $q_2\geq (-1+\delta)m$. If on the other hand $q_2<(-1+\delta)m$, then by a $L^\infty\times L^2$ norm estimate and \eqref{BoundsOnMAddedBenoit} we have
\begin{equation}
\begin{aligned}
 \norm{P_{k,p,q}\Q_\m(g_1,g_2)}_{L^2}&\lesssim 2^{k+q_{\max}}\cdot \norm{e^{it\Lambda}g_1}_{L^\infty}\cdot 2^{\frac{q_2}{2}}\norm{g_2}_{B}\lesssim 2^{k}\cdot 2^{-m+\kappa m}\cdot 2^{-\frac{1-\delta}{2}m}\eps_1^2\lesssim 2^k 2^{-\frac{3}{2}m+(\kappa+\frac{\delta}{2})m}\cdot \eps_1^2.
\end{aligned} 
\end{equation}
Case 2.1b: In case $\ell_2> q_{1}-q_2$ we can repeatedly integrate by parts if $\ell_2-q_{\max}\leq (1-\delta)m$. Else we use an $L^\infty\times L^2$ estimate to get that
\begin{equation}
\begin{aligned}
 \norm{P_{k,p,q}\Q_\m(g_1,g_2)}_{L^2}&\lesssim 2^{k+q_{\max}} \norm{e^{it\Lambda}g_1}_{L^\infty}\cdot 2^{-\ell_2}\norm{g_2}_{X}\lesssim 2^k2^{-2m+(\kappa+\delta) m}\cdot \eps_1^2.
\end{aligned} 
\end{equation}

\subsubsection*{Case 2.2: $k=k_{\min}$ and $\vert k_1-k_2\vert\le 10$} 
By Lemma \ref{lem:new_ibp}\eqref{it:ibp_pq}, repeated integration by parts along $V_{\xi-\eta}$ is feasible if
\begin{equation}
 2^{-m}2^{-q_{\max}}2^{k_2-k}\max\{2^{q_1-q_2},2^{\ell_2}\}<2^{-\delta m}.
\end{equation}
If this condition is violated we distinguish cases as above in Cases 2.1a resp. 2.1b: either $-q_2+k_2-k>(1-\delta) m$ and then
\begin{equation}
\begin{aligned}
 \norm{P_{k,p,q}\Q_\m(g_1,g_2)}_{L^2}&\lesssim 2^{k+q_{\max}}\cdot \norm{e^{it\Lambda}g_1}_{L^\infty}\cdot 2^{\frac{q_2}{2}}\norm{g_2}_{B}\lesssim 2^{\frac{k+k_2}{2} m}\cdot 2^{-m+\kappa m}\cdot 2^{-\frac{1-\delta}{2}m}\eps^2\\
 &\lesssim 2^{\frac{k}{2}}2^{-\frac{3}{2}m+(\kappa+\frac{\delta+\delta_0+2}{2}) m}\cdot \eps_1^2,
\end{aligned} 
\end{equation}
or $\ell_2-q_{\max}+k_2-k>(1-\delta) m$ and then
\begin{equation}
\begin{aligned}
 \norm{P_{k,p,q}\Q_\m(g_1,g_2)}_{L^2}&\lesssim 2^{k+q_{\max}} \norm{e^{it\Lambda}g_1}_{L^\infty}\cdot 2^{-\frac{\ell_2}{2}}\norm{g_2}_{X}\lesssim 2^{\frac{k+k_2}{2}}\cdot 2^{-m+\kappa m}\cdot 2^{-\frac{(1-\delta)}{2}m}\cdot \eps^2\\
 &\lesssim 2^{\frac{k}{2}} 2^{-\frac{3}{2}m+(\kappa+\frac{\delta+\delta_0}{2}) m}\cdot \eps_1^2.
\end{aligned} 
\end{equation}

\medskip

Thus the only scenario we are left with is the following:
\subsubsection*{Case 3: No gaps}
 In this case we have that $2^p\sim 2^{p_i}$ and $2^q\sim 2^{q_i}$, $i=1,2$. As before we can also assume that $p,q\ge-4m$. Then we are done by a direct $L^2\times L^\infty$ estimate: Assuming without loss of generality that $g_1$ has fewer vector fields than $g_2$ (i.e.\ $b_1\leq b_2$ in our original notation), we have by Proposition \ref{prop:decay} that $e^{it\Lambda} P_{k_1,p_1,q_1}R_{\ell_1}g_1=I_{1,1}+I_{1,2}$ with 
 \begin{equation}
  \norm{I_{1,1}}_{L^\infty}\lesssim \eps_1\cdot 2^{-\frac{3}{2}\vert k_1\vert} 2^{-p-\frac{q}{2}}t^{-\frac{3}{2}},\quad \norm{I_{1,2}}_{L^2}\lesssim \eps_1 \cdot 2^{-\frac{3}{2}\vert k_1\vert}(t2^p)^{-1}\mathfrak{1}_{2^{2p+q}\gtrsim t^{-1}}, 
 \end{equation}
so that using \eqref{BoundsOnMAddedBenoit},
\begin{equation}
\begin{aligned}
 \norm{P_{k,p,q}\Q_\m(g_1,g_2)}_{L^2}&\lesssim 2^{p+q+k}[\norm{I_{1,1}}_{L^\infty}2^{p+\frac{q}{2}}\norm{f_2}_{B}+\norm{I_{1,2}}_{L^2}\norm{e^{it\Lambda}g_2}_{L^\infty}]\lesssim 2^{k}\cdot 2^{-\frac{3}{2}m}\cdot \eps_1^2,
\end{aligned} 
\end{equation}
where we have used the dispersive decay (at rate at least $t^{-1/2}$) of $e^{it\Lambda}g_2$, which in case $g_2$ has more than $N-3$ vector fields follows by interpolation (see Lemma \ref{lem:interpol} resp.\ Corollary \ref{cor:extrapol_decay}). 
\end{proof}

\section{Energy estimates and $B$ norm bounds}\label{sec:Bnorm}
It is classical to obtain energy estimates for \eqref{eq:EC}. As we showed in \cite[Proposition 5.1]{rotE}, both derivatives and vector fields can be controlled in $L^2$ as follows:
\begin{proposition}[Proposition 5.1 in \cite{rotE}]\label{prop:EnergyIncrement}
Assuming that $\bm{u}$ solves \eqref{eq:EC} on $0\le t\le T$, for $n\in\N$ there holds that
\begin{equation}\label{EnergyIncrementEq}
\begin{split}
 \Vert \bm{u}(t)\Vert_{H^{n}}^2-\Vert \bm{u}(0)\Vert_{H^{n}}^2&\lesssim \int_{s=0}^t\alpha(s)\cdot \Vert \bm{u}(s)\Vert_{H^n}^2\cdot \frac{ds}{1+s},\\
 \Vert S^n\bm{u}(t)\Vert_{L^2}^2-\Vert S^n\bm{u}(0)\Vert_{L^2}^2&\lesssim \int_{s=0}^t\alpha(s)\cdot \left(\Vert \bm{u}(s)\Vert_{H^n}^2+\sum_{b=0}^n\Vert S^b\bm{u}(s)\Vert_{L^2}^2\right)\cdot \frac{ds}{1+s},\\
 \Vert \vert\nabla\vert^{-1}S^n\bm{u}(t)\Vert_{L^2}^2-\Vert \vert\nabla\vert^{-1}S^n\bm{u}(0)\Vert_{L^2}^2&\lesssim \int_{s=0}^t \alpha(s)\cdot \sum_{b=0}^n\Vert S^b\bm{u}(s)\Vert_{L^2}^2\cdot \frac{ds}{1+s},
\end{split}
\end{equation}
where
\begin{equation}
\begin{split}
 \alpha(s)&=(1+s)\left[\Vert \bm{u}(s)\Vert_{L^\infty}+\Vert \nabla_x \bm{u}(s)\Vert_{L^\infty}\right].
\end{split}
\end{equation}
\end{proposition}

As a consequence, with the decay bounds of Corollary \ref{cor:decay} we obtain:
\begin{corollary}\label{cor:energy}
 Under the bootstrap assumptions \eqref{eq:btstrap-assump} there holds that
 \begin{equation}\label{eq:btstrap-concl2'}
  \norm{U_\pm(t)}_{H^{2N_0}\cap H^{-1}}+\norm{S^aU_\pm(t)}_{L^2\cap H^{-1}}\lesssim\eps_0\ip{t}^{C\eps_1},\qquad 0\le a\le M.
 \end{equation}
\end{corollary}
\begin{proof}
 It suffices to note that under the assumptions \eqref{eq:btstrap-assump}, we can use Proposition \ref{prop:decay} to get a constant $C>0$ such that
 \begin{equation}
  \alpha(s)\leq C\eps_1.
 \end{equation}
\end{proof}

The main goal of this section is then to upgrade this $L^2$ information on many vector fields to stronger $B$ norm bounds of fewer vector fields on the solution profiles. After the reduction to bilinear bounds as in the proof of Proposition \ref{prop:btstrap}, this is done by establishing the following claim (see also \eqref{eq:btstrap-concl1.1-B}):
\begin{proposition}\label{prop:Bnorm}
 Assume the bootstrap assumptions \eqref{eq:btstrap-assump} of Proposition \ref{prop:btstrap}. Then for $\delta=2M^{-1/2}>0$ and with $F_j=S^{b_j}\U_{\mu_j}$, $0\leq b_1+b_2\leq N$, $\mu_j\in\{+,-\}$, $j=1,2$, there holds that
 \begin{equation}\label{eq:B-claim0}
  \norm{\B_\m(F_1,F_2)}_B\lesssim 2^{-\delta^3 m}\eps_1^2.
 \end{equation}
\end{proposition}

We recall again that here $\m$ is one of the multipliers of the Euler-Coriolis system in the dispersive formulation \eqref{eq:EC_disp} (see Lemma \ref{lem:ECmult}), for which we have the bounds of Lemma \ref{lem:ECmult_bds}. The remainder of this section now gives the proof of Proposition \ref{prop:Bnorm}.

\begin{proof}[Proof of Proposition \ref{prop:Bnorm}]
In most cases, we will be able to prove the stronger bound
\begin{equation}\label{StrongerThanBNorm}
2^k2^{4k^+} \norm{\mathcal{F}\left[P_k\B_\m(F_1,F_2)\right]}_{L^\infty}\lesssim 2^{-\delta^3 m}\eps_1^2.
\end{equation}

\subsubsection{Some simple cases}\label{ssec:Bnorm-reduction}
From the energy bounds \eqref{eq:interpol_energy} and with $\abs{\m}\lesssim 2^k$ and $\Sz\leq 2^{p+\frac{q}{2}+\frac{3}{2}k}$ we deduce that since $F_i=\sum_{k_i}P_{k_i}F_i$ there holds that
\begin{equation}
\begin{aligned}
 2^k2^{4k^+}\norm{\mathcal{F}\left[P_{k}\B_\m(F_1,F_2)\right]}_{L^\infty}\lesssim 2^m\cdot 2^{k+4k^+}&\sum_{k_1,k_2}\min\{2^{-N_0k_1}\norm{F_1}_{H^{N_0}},2^{k_1}\norm{F_1}_{\dot{H}^{-1}}\} \\
 &\qquad \cdot \min\{2^{-N_0k_2}\norm{F_2}_{H^{N_0}},2^{k_2}\norm{F_2}_{\dot{H}^{-1}}\},
\end{aligned} 
\end{equation}
and \eqref{StrongerThanBNorm} follows if $\min\{k,k_1,k_2\}\leq -2m$ or $\max\{k,k_1,k_2\}\geq \delta_0 m$, where $\delta_0=2N_0^{-1}$. 

Localizing in $p_j, \ell_j$ with $\ell_j\geq -p_j$, $j=1,2$, we have that with
\begin{equation}
 f_j=P_{k_j,p_j}R_{\ell_j}F_j, \qquad j=1,2,
\end{equation}
there holds that
\begin{equation}
 2^{k+4k^+}\norm{\mathcal{F}\left[P_{k}\B_\m(f_1,f_2)\right]}_{L^\infty}\lesssim 2^{(1+2\delta_0)m}\cdot 2^{-\ell_1}\norm{f_1}_X\cdot 2^{-\ell_2}\norm{f_2}_X\lesssim 2^{(1+2\delta_0)m-\ell_1-\ell_2}\varepsilon_1^2,
\end{equation}
and this gives \eqref{StrongerThanBNorm} if $\min\{\ell_1,\ell_2\}\ge 2m$, so that to prove \eqref{eq:B-claim0} it suffices to show that for
\begin{equation}\label{AssumptionBnormAddedBenoit}
 -2m\leq k,k_j\leq \delta_0 m,\qquad -2m\leq p_j\leq 0,\qquad -p_j\leq \ell_i\leq 2m,\quad j=1,2, 
\end{equation}
we have
\begin{equation}\label{eq:B-claim}
 \sup_{k,p,q}2^{-\frac{1}{2}k^{-}}2^{-p-\frac{q}{2}}\norm{P_{k,p,q}\B_\m(f_1,f_2)}_{L^2}\lesssim 2^{-\delta m}\eps_1^2.
\end{equation}

The rest of this section establishes \eqref{eq:B-claim}, by first treating the case of a gap in $p$ (i.e.\ $p_{\min}\ll p_{\max}$) with $2^{p_{\max}}\sim 1$ (Section \ref{ssec:B-p-gap}), secondly that of $p_{\max}\ll 0$ (Section \ref{ssec:B-pmax}), thirdly that of a gap in $q$ (Section \ref{ssec:B-q-gap}, and finally the case of no gaps (Section \ref{ssec:B-nogaps}).

\subsection{Gap in $p$, with $p_{\max}\sim 0$.}\label{ssec:B-p-gap}
We show \eqref{eq:B-claim} when \eqref{AssumptionBnormAddedBenoit} holds and in addition $p_{\min}\ll p_{\max}\sim 0$.

We further subdivide according to whether the output $p$ or one of the inputs $p_i$ is small, and use Lemma \ref{lem:gapp-cases} to organize these cases. Without loss of generality we will assume that $p_1\leq p_2$, so that we have two main cases to consider. Noting that $\abs{\bar\sigma}\sim 2^{p_{\max}}2^{k_{\min}+k_{\max}}$ and using that $\ell_i+p_i\geq 0$ (and thus $p_2-p_1\leq -\ell_1$, $p_1-p_2\leq -\ell_2$), repeated integration by parts is feasible if (see Lemma \ref{lem:new_ibp})
\begin{equation}
\begin{aligned}
 V_\eta:\qquad &2^{-p_1}2^{2k_1-k_{\min}-k_{\max}}(1+2^{k_2-k_1}2^{\ell_1})\leq 2^{(1-\delta)m},\\
 V_{\xi-\eta}:\qquad &2^{-p_2}2^{2k_2-k_{\min}-k_{\max}}(1+2^{k_1-k_2}2^{\ell_2})\leq 2^{(1-\delta)m}.
\end{aligned}
\end{equation}

\subsubsection{Case 1: $p\ll p_1,p_2$} By Lemma \ref{lem:gapp-cases} we have three scenarios to consider:
\subsubsection*{Subcase 1.1: $2^{k_1}\sim 2^{k_2}$.} Here we have $2^{p_1}\sim 2^{p_2}\sim 1$. Using Lemma \ref{lem:new_ibp}, iterated integration by parts in $V_\eta$ or $V_{\xi-\eta}$ gives the result if $\min\{\ell_1,\ell_2\}<(1-\delta)m+k-k_1$. Else we have the bound
\begin{equation}
\begin{aligned}
2^{k+4k^+}\norm{\mathcal{F}\left[P_{k,p}\Q_\m(f_1,f_2)\right]}_{L^\infty}&\lesssim 2^{2k-2k_2^+}\cdot 2^{-\ell_1-\ell_2}\norm{f_1}_X\norm{f_2}_{X}\\
 &\lesssim 2^{-2(1-\delta)m}\norm{f_1}_X\norm{f_2}_{X}\ll 2^{-(1+\frac{\beta}{2})m}\eps_1^2.\\
\end{aligned} 
\end{equation}

\subsubsection*{Subcase 1.2: $2^{k_2}\ll 2^{k_1}\sim 2^k$} Then we have that $2^{p_1-p_2}\sim 2^{k_2-k_1}\ll 1$, so that $p\ll p_1\ll p_2\sim 0$. Using Lemma \ref{lem:new_ibp}, iterated integration by parts in $V_{\xi-\eta}$ gives the claim if $\ell_2<(1-\delta)m$. Else, when $\ell_2>(1-\delta)m$ there holds that 
\begin{equation}
\begin{aligned}
2^{k+4k^+}\norm{\mathcal{F}\left[P_{k,p}\Q_\m(f_1,f_2)\right]}_{L^\infty}&\lesssim 2^{2k_1+4k_1^+}\cdot \norm{f_1}_{L^2} 2^{-(1+\beta)\ell_2}\norm{f_2}_{X}
 \ll 2^{-(1+\frac{\beta}{2})m}\eps_1^2.
\end{aligned} 
\end{equation}

\subsubsection*{Subcase 1.3: $2^{k_1}\ll 2^{k_2}\sim 2^k$} This leads to $p_2\ll p_1$ which is excluded.

\subsubsection{Case 2: $p_1\ll p,p_2$}
By Lemma \ref{lem:gapp-cases} we have three scenarios to consider:
\subsubsection*{Subcase 2.1: $2^{k}\sim 2^{k_2}$} Then $2^{p}\sim 2^{p_2}\sim 1$. Iterated integration by parts in $V_\eta$ (with $\abs{\bar\sigma}\gtrsim 2^{k_1+k}$) gives the claim if $\ell_1-p_1\leq (1-\delta)m$, whereas iterated integration by parts in $V_{\xi-\eta}$ suffices if
\begin{equation}
 2^{k_2-k_1}+2^{\ell_2}\leq 2^{(1-\delta)m}.
\end{equation}
We may thus assume that $\ell_1-p_1> (1-\delta)m$ and that $\max\{k_2-k_1,\ell_2\}>(1-\delta)m$, but this suffices in view of the crude bound
\begin{equation}
\begin{aligned}
 2^{3k^+-\frac{1}{2}k^-}2^{-\frac{q}{2}}\norm{P_{k,p,q}\Q_\m(f_1,f_2)}_{L^2}
&\lesssim 2^{-\frac{q}{2}}\Sz\cdot 2^{k-\frac{1}{2}k^-+3k_2^+}2^{-\ell_1-\ell_2}\norm{f_1}_{X} \norm{f_2}_{X}\\
 &\lesssim 2^{\frac{k+k^+}{2}+k_1+p_1}2^{-(1-\delta)m-p_1-\ell_2}\norm{f_1}_{X}\norm{f_2}_{X}\\
 &\lesssim 2^{\frac{3}{2}k+\frac{1}{2}k^+}\cdot 2^{k_1-k_2-\ell_2}\cdot 2^{-(1-\delta)m}\norm{f_1}_{X}\norm{f_2}_{X}.
\end{aligned} 
\end{equation}

\subsubsection*{Subcase 2.2: $2^{k}\ll 2^{k_2}\sim 2^{k_1}$} Then $2^{p_2-p}\sim 2^{k-k_2}\ll 1$, and thus $p_1\ll p_2\ll p\sim 0$ and we only need to recover $q$. Repeated integration by parts in $V_{\xi-\eta}$ (where now $\abs{\bar\sigma}\sim 2^{k_2+k}$) gives the claim if
\begin{equation}
 -p_2+k_2-k+\ell_2\leq (1-\delta)m.
\end{equation}
In the opposite case we use that, since $\beta\leq\frac{1}{3}$, with $2^{\beta(k_2-k)}\sim 2^{-\beta p_2}$ and $\Sz\lesssim 2^{\frac{k+q}{2}}2^{k_1+p_1}$ we can bound
\begin{equation}
\begin{aligned}
 2^{3k^+-\frac{1}{2}k^-} 2^{-\frac{q}{2}}\norm{P_{k,p,q}\Q_\m(f_1,f_2)}_{L^2}&\lesssim 2^{-\frac{q}{2}}\Sz\cdot 2^{k-\frac{1}{2}k^-}\cdot 2^{p_1+\frac{k_1}{2}}\norm{f_1}_{B}\cdot 2^{-(1+\beta)\ell_2-\beta p_2-3k_2^+}\norm{f_2}_{X}\\
 &\lesssim 2^{k+\frac{1}{2}k^+}2^{\frac{3}{2}k_1+2p_1}\norm{f_1}_{B}2^{-(1+\beta)(1-\delta)m}2^{-(1+2\beta)p_2}2^{(1+\beta)(k_2-k)}2^{-3k_2^+}\norm{f_2}_{X}\\
 &\lesssim 2^{-(1+\beta)(1-\delta)m}2^{p_1-3\beta p_2}\norm{f_1}_{B}\norm{f_2}_{X}\\
 &\lesssim 2^{-(1+\frac{\beta}{2})m}\eps_1^2.
\end{aligned} 
\end{equation}

\subsubsection*{Subcase 2.3: $2^{k_2}\ll 2^{k}\sim 2^{k_1}$} Then $2^{p-p_2}\sim 2^{k_2-k}\ll 1$, and thus $p_1\ll p\ll p_2\sim 0$. This is as in Subcase 1.2: if $\ell_2\leq (1-\delta)m$, then repeated integration by parts in $V_{\xi-\eta}$ gives the claim, \blue{whereas for $\ell_2>(1-\delta)m$ we have that
\begin{equation}
\begin{aligned}
2^{k+4k^+}\norm{\mathcal{F}\left[P_{k,p}\Q_\m(f_1,f_2)\right]}_{L^\infty}&\lesssim 2^{2k_1+4k_1^+}\cdot \norm{f_1}_{L^2} 2^{-(1+\beta)\ell_2}\norm{f_2}_{X}
 \ll 2^{-(1+\frac{\beta}{2})m}\eps_1^2.
\end{aligned} 
\end{equation}
}

\subsection{Case $p_{\min}\ll p_{\max}\ll 0$.}\label{ssec:B-pmax}
Here we have that $\abs{\Phi}\gtrsim 1$. We can use a normal form as in \eqref{eq:mdecompNF}--\eqref{mdecompNFNRNorms} with $\lambda=\frac{1}{10}$ so that $\mathfrak{m}^{res}=0$ and we see that
\begin{equation}
 \norm{P_{k,p}\B_{\m}(f_1,f_2)}_{L^2}\lesssim
 \norm{P_{k,p}\Q_{\m\cdot\Phi^{-1}}(f_1,f_2)}_{L^2}+\norm{P_{k,p}\B_{\m\cdot\Phi^{-1}}(\partial_t f_1,f_2)}_{L^2}+\norm{P_{k,p}\B_{\m\cdot\Phi^{-1}}( f_1,\partial_tf_2)}_{L^2}.
\end{equation}
Using Lemma \ref{lem:dtfinL2} the second term can be bounded as
\begin{equation}
2^{k+4k^+}\norm{\mathcal{F}\left[P_{k,p}\B_{\m\cdot\Phi^{-1}}(\partial_t f_1,f_2)\right]}_{L^\infty}\lesssim 2^m\cdot 2^{2k+4k^+} 2^{p_{\max}}\norm{\partial_t f_1}_{L^2}\norm{f_2}_{L^2}\lesssim 2^{-\frac{1}{4}m}\eps_1^2,
\end{equation}
and similarly for the third one.

It thus remains to control the boundary term. If $\min\{p_1,p_2\}=p_1\lesssim p$, this follows from Proposition \ref{prop:decay}, Lemma \ref{lem:phasesymb_bd}, Corollary \ref{cor:extrapol_decay} and the multiplier bounds \eqref{BoundsOnMAddedBenoit} as follows: 
\begin{equation*}
\begin{split}
 2^{3k^+-\frac{1}{2}k^-}2^{-p}\norm{P_{k,p}\Q_{\m\cdot\Phi^{-1}}(f_1,f_2)}_{L^2}&\lesssim 2^{-p}2^{\frac{k+k^+}{2}+3k^+}2^{p_{\max}}2^{p_1}\norm{f_1}_{B}\cdot \Vert e^{it\Lambda}f_2\Vert_{L^\infty}\ll 2^{-\delta m}\eps_1^2.
 \end{split}
\end{equation*}
If $\min\{p_1,p_2\}=p_2\lesssim p$, the situation is similar. We note that if $p_{\max}\leq -\delta m$, then we have that
\begin{equation*}
\begin{split}
 2^{k+4k^+}\norm{\mathcal{F}\left[P_{k,p}\Q_{\m\cdot\Phi^{-1}}(f_1,f_2)\right]}_{L^\infty}&\lesssim 2^{2k+4k^+}2^{p_{\max}}\norm{f_1}_{L^2}\norm{f_2}_{L^2}\lesssim 2^{3k^+}2^{p_1+p_2+p_{max}}\Vert f_1\Vert_B\Vert f_2\Vert_B\\
 &\lesssim 2^{-2\delta m}\eps_1^2,
 \end{split}
\end{equation*}
so we may assume that $p_{\max}>-\delta m$. Then with $\abs{\bar\sigma}\gtrsim 2^{-\delta m}2^{k_{\max}+k_{\min}}$ and the fact that
\begin{equation}
 s^{-1}\abs{\frac{1}{V_\eta\Phi}V_\eta(\Phi^{-1})}\lesssim s^{-1}\abs{\Phi}^{-2}\lesssim s^{-1},
\end{equation}
we are done by integration by parts as in the case of a gap in $p$, Section \ref{ssec:B-p-gap}.

After the estimates in Section \ref{ssec:B-p-gap} and Section \ref{ssec:B-pmax}, we may assume that all $p$'s are comparable:
\begin{equation*}
 p_{\max}\le p_{\min}+100.
\end{equation*}

\subsection{Gap in $q$.}\label{ssec:B-q-gap}
We additionally localize in $q_i$, writing $g_i=P_{k_i,p_i,q_i}R_{\ell_i}f_i$, $i=1,2$, and can assume by $B$ norm bounds that $q_i\geq -3m$.

For this case we now assume that $q_{\min}\ll q_{\max}$ and thus (by the previous case) $2^{p_{\min}}\sim 2^{p_{\max}}\sim 1$ (and thus also $\ell_i\geq 0$). Noting that $\abs{\bar\sigma}\sim 2^{q_{\max}}2^{k_{\min}+k_{\max}}$ and using Lemma \ref{lem:new_ibp}, repeated integration by parts is feasible if
\begin{equation}\label{CondIBPAddedGapqBNorm}
\begin{aligned}
 V_\eta:\qquad &2^{-q_{\max}}2^{2k_1-k_{\min}-k_{\max}}(1+2^{k_2-k_1}(2^{q_2-q_1}+2^{\ell_1}))\leq 2^{(1-\delta)m},\\
 V_{\xi-\eta}:\qquad &2^{-q_{\max}}2^{2k_2-k_{\min}-k_{\max}}(1+2^{k_1-k_2}(2^{q_1-q_2}+2^{\ell_2}))\leq 2^{(1-\delta)m}.
\end{aligned}
\end{equation}

We have two main cases to consider:

\subsubsection{Case 3: $q\ll q_1,q_2$}
By Lemma \ref{lem:gapp-cases} and Remark \ref{lem:gapp-casesRem}, we have three scenarios to consider:

\subsubsection*{Subcase 3.1: $2^{k_1}\sim 2^{k_2}$} Then also $2^{q_1}\sim 2^{q_2}$. Using \eqref{CondIBPAddedGapqBNorm}, we see that repeated integration by parts gives the claim if 
\begin{equation}
 -q_{\max}+k_1-k+\min\{\ell_1,\ell_2\}\leq (1-\delta)m.
\end{equation}
Otherwise, using a crude bound and \eqref{BoundsOnMAddedBenoit}, we have that
\begin{equation}
\begin{aligned}
2^{k+4k^+}\norm{\mathcal{F}\left[P_{k,p,q}\Q_\m(g_1,g_2)\right]}_{L^\infty}&\lesssim 2^{2k+k^+-3k_2^+}2^{q_{\max}}\cdot 2^{\frac{q_1}{2}}\norm{g_1}_{B} 2^{-(1+\beta)\ell_2}\norm{g_2}_{X}\\
 &\lesssim 2^{\delta m}\cdot 2^{(\frac{1}{2}-\beta)q_{\max}}2^{-(1+\beta)(1-\delta)m}\norm{g_1}_{B}\norm{g_2}_{X},
\end{aligned} 
\end{equation}
which is an acceptable contribution.

\subsubsection*{Subcase 3.2: $2^{k_2}\ll 2^{k_1}\sim 2^k$} Then we have $2^{q_1-q_2}\sim 2^{k_2-k_1}\ll 1$, and thus $q\ll q_1\ll q_2$. From \eqref{CondIBPAddedGapqBNorm}, we see that repeated integration by parts in $V_{\xi-\eta}$ gives the claim provided that
\begin{equation}
 -q_{2}+\ell_2\leq (1-\delta)m.
\end{equation}
Otherwise we can conclude just as in Subcase 3.1.

\subsubsection*{Subcase 3.3: $2^{k_1}\ll 2^{k_2}$} This is symmetric to Subcase \blue{3.2}.

\subsubsection{Case 4: $\min\{q_1,q_2\}\ll q$} Without loss of generality, we may assume that $q_1\le q_2$. By Lemma \ref{lem:gapp-cases} we have three scenarios to consider:

\subsubsection*{Subcase 4.1: $2^{k}\sim 2^{k_2}$} Then also $2^{q}\sim 2^{q_2}$.
Inspecting \eqref{CondIBPAddedGapqBNorm},  repeated integration by parts gives the claim if
\begin{equation}
\begin{aligned}
 V_\eta:\quad &\max\{-q_1,\ell_1-q_{\max}\}\leq (1-\delta)m,\qquad
 V_{\xi-\eta}:\quad \max\{k_2-k_1,\ell_2\}\leq (1-\delta)m+q_{\max}.
\end{aligned}
\end{equation}
In the opposite case, if $2^{q_1}\lesssim 2^{-(1-\delta)m}$ we can use Lemma \ref{lem:set_gain} and \eqref{BoundsOnMAddedBenoit} to bound
\begin{equation}
\begin{aligned}
 2^{3k^+-\frac{1}{2}k^-}2^{-\frac{q}{2}}\norm{P_{k,p,q}\Q_\m(g_1,g_2)}_{L^2}&\lesssim 2^{-\frac{q_{\max}}{2}}\Sz\cdot 2^{k-\frac{1}{2}k^-+q_{\max}}\cdot 2^{\frac{q_1}{2}}\norm{g_1}_{B} 2^{-(1+\beta)\ell_2}\norm{g_2}_{X}\\
 &\lesssim 2^{\frac{q_{\max}}{2}}2^{q_1}2^{\frac{k+k^+}{2}+\frac{3}{2}k_1}2^{-(1+\beta)\ell_2}\varepsilon_1^2\\
 &\lesssim 2^{\frac{q_{\max}}{2}}2^{2k_1+\frac{1}{2}k^+}2^{-(1-\delta)m}2^{\frac{k-k_1}{2}}2^{-(1+\beta)\ell_2}\varepsilon_1^2,
\end{aligned} 
\end{equation}
which suffices since $\max\{\ell_2,k-k_1\}>(1-\delta)m+q_{\max}$. If on the other hand $\ell_1-q_{\max}>(1-\delta)m$, then a crude estimate gives
\begin{equation}
\begin{aligned}
2^{k+4k^+}\norm{\mathcal{F}\left[P_{k,p,q}\Q_\m(g_1,g_2)\right]}_{L^\infty}&\lesssim 2^{2k+2k^++q_{\max}}\cdot 2^{-(1+\beta)\ell_1}\norm{g_1}_{X} 2^{\frac{q_2}{2}}\norm{g_2}_{B}\\
&\lesssim 2^{3k^+}2^{-(1+\beta)(\ell_1-q_{\max})}\varepsilon_1^2\lesssim 2^{-(1+\delta)m}\varepsilon_1^2
\end{aligned} 
\end{equation}
which is an acceptable contribution.

\subsubsection*{Subcase 4.2: $2^{k}\ll 2^{k_2}\sim 2^{k_1}$} Then also $2^{q_2-q}\sim 2^{k-k_2}\ll 1$, so that $q_1\ll q_2\ll q$. Using \eqref{CondIBPAddedGapqBNorm}, repeated integration by parts then gives the claim if
\begin{equation}
\begin{aligned}
 V_{\xi-\eta}:\qquad &-q+k_2-k+\ell_2\leq (1-\delta)m.
\end{aligned}
\end{equation}
Otherwise we get the acceptable contribution
\begin{equation}
\begin{aligned}
 2^{k+4k^+}\norm{\mathcal{F}\left[P_{k,p,q}\Q_\m(g_1,g_2)\right]}_{L^\infty}&\lesssim 2^{2(k-k_1^+)+q}\cdot 2^{\frac{k_1}{2}+\frac{q_1}{2}}\norm{g_1}_{B} 2^{-(1+\beta)\ell_2}\norm{g_2}_{X}\\
 &\lesssim 2^{2k_2^+}2^{-(1+\beta)(\ell_2+k_2-k-q)}\varepsilon_1^2.
\end{aligned} 
\end{equation}

\subsubsection*{Subcase 4.3: $2^{k_2}\ll 2^{k}\sim 2^{k_1}$} Then also $2^{q-q_2}\sim 2^{k-k_2}\ll 1$, so that $q_1\ll q\ll q_2$. From \eqref{CondIBPAddedGapqBNorm}, repeated integration by parts then gives the claim if
\begin{equation}
\begin{aligned}
 V_{\xi-\eta}:\quad \ell_2-q_2\leq (1-\delta)m.
\end{aligned}
\end{equation}
Otherwise, we get an acceptable contribution as in Subcase 4.2:
\begin{equation}
\begin{aligned}
2^{k+4k^+}\norm{\mathcal{F}\left[P_{k,p,q}\Q_\m(g_1,g_2)\right]}_{L^\infty}&\lesssim 2^{2k+k^++q_{2}}\cdot 2^{\frac{k_1}{2}+\frac{q_1}{2}}\norm{g_1}_{B} 2^{-(1+\beta)\ell_2}\norm{g_2}_{X}\\
&\lesssim 2^{4k_1^+}2^{-(1+\beta)(\ell_2-q_2)}\varepsilon_1^2\le 2^{-(1+\beta/2)m}\varepsilon_1^2.
\end{aligned} 
\end{equation}

\subsection{No gaps.}\label{ssec:B-nogaps}
Assume now that $2^{p_{\min}}\sim 2^{p_{\max}}$ and $2^{q_{\min}}\sim 2^{q_{\max}}$. Assuming further w.l.o.g.\ that $g_1$ has at most $\frac{N}{2}$ copies of $S$, by the decay estimate in Proposition \ref{prop:decay} we then have that
\begin{equation}
 e^{it\Lambda}g_1=I+I\!I
\end{equation}
with
\begin{equation}
 \norm{I}_{L^\infty}\lesssim 2^{-p-\frac{q}{2}}t^{-\frac{3}{2}}2^{\frac{3}{2}k_1-3k_1^+}\eps_1,\qquad \norm{I\!I}_{L^2}\lesssim 2^{-3k^+_1}t^{-\frac{1}{2}}\eps_1.
\end{equation}
By Corollary \ref{cor:extrapol_decay} we further have that
\begin{equation}
 \norm{e^{it\Lambda} g_2}_{L^\infty}\lesssim 2^{-\frac{3}{2}k_2^+}2^{-\frac{2}{3}m}\eps_1,
\end{equation}
and using \eqref{BoundsOnMAddedBenoit} and an $L^\infty\times L^2$ estimate, we find that
\begin{equation}
\begin{aligned}
 2^{3k^+-\frac{1}{2}k^-}2^{-p-\frac{q}{2}}\norm{P_{k,p,q}\Q_\m(g_1,g_2)}_{L^2}&\lesssim 2^{3k^+}2^{-p-\frac{q}{2}}\cdot 2^{p+q+\frac{k+k^+}{2}} \left(\norm{I}_{L^\infty}2^{p+\frac{q}{2}}\norm{g_2}_B+\norm{I\!I}_{L^2}\norm{e^{it\Lambda} g_2}_{L^\infty}\right)\\
 &\lesssim 2^{\frac{q}{2}+3k^++\frac{k+k^+}{2}}\left(2^{-\frac{3}{2}m}+2^{-\frac{m}{2}}2^{-\frac{2}{3}m}\right)\eps_1^2\\
 &\lesssim 2^{-\frac{9}{8}m}\eps_1^2.
\end{aligned} 
\end{equation}

\end{proof}

\section{$X$ norm bounds}\label{sec:Xnorm}
In this section we finally prove the $X$ norm bounds for the quadratic expressions \eqref{eq:btstrap-concl1.1-X}. This is done first for the case of ``large'' $\ell$ in Section \ref{ssec:Xnorm1}, then for ``small'' $\ell$ in Section \ref{ssec:Xnorm2}.

\subsection{$X$ norm bounds for $\ell>(1+\delta)m$}\label{ssec:Xnorm1}
The goal here is to show that if $\ell$ is sufficiently large, then we have the $X$ norm bounds claimed in the bootstrap conclusion \eqref{eq:btstrap-concl1}. More precisely, we will show:
\begin{proposition}\label{prop:Xnorm1} 
Let $0<\delta=2M^{-\frac{1}{2}}\ll \beta$, and assume the bootstrap assumptions \eqref{eq:btstrap-assump} of Proposition \ref{prop:btstrap} and let $F_j=S^{b_j}\U_{\mu_j}$, $0\leq b_1+b_2\leq N$, $\mu_j\in\{+,-\}$, $j=1,2$. Then there holds that
 \begin{equation}\label{eq:X-claim1}
  \sup_{k,\,\ell+p\geq 0,\,\ell>(1+\delta)m}2^{3k^+}2^{(1+\beta)\ell}2^{\beta p}\norm{P_{k,p}R_\ell\B_\m(F_1,F_2)}_{L^2}\lesssim 2^{-\delta^2 m}\eps_1^2.
 \end{equation}
\end{proposition}

We give next the proof of Proposition \ref{prop:Xnorm1}. As one sees below, here the choice of $\delta^2=O(M^{-1})$ will be convenient for repeated integrations by parts, where $M\in\N$ is the number of vector fields we propagate. 

\begin{proof}[Proof of Proposition \ref{prop:Xnorm1}]
Assuming that $\ell>(1+\delta)m$, we split our arguments into two cases: If $\ell+p\leq \delta m$ (Section \ref{sec:Xnorm1-1}) ``only'' a gain of $2^{(1+\delta)m}$ is needed, and relatively simple arguments suffice. If on the other hand $\ell+p>\delta m$ (Section \ref{sec:Xnorm1-2}) we can make use of a ``finite speed of propagation'' feature\footnote{\blue{This refers to the fact that localization on scales greater than $t$ and linear flow almost commute, similar to the terminology in e.g.\ \cite[Section 7.2]{DIPau} or \cite[Section 3B]{globalbeta}.}} of the equations via the Bernstein property \eqref{eq:R-Bernstein} of the angular Littlewood-Paley decomposition. We remark that in the arguments that follow, no localizations in $q,q_j$ are used.

\subsubsection{Case $\ell+p\leq \delta m$}\label{sec:Xnorm1-1}
Here we have that $p\leq -\ell+\delta m<-m$.

We begin with some direct observations to treat a few simple cases, akin to Section \ref{ssec:Bnorm-reduction}. Noting that
\begin{equation}
\begin{aligned}
 2^{3k^+}2^{(1+\beta)\ell}2^{\beta p}\norm{P_{k,p}R_\ell\B_\m(P_{k_1}F_1,P_{k_2}F_2)}_{L^2}&\lesssim 2^{3k^++\frac{3}{2}k}2^{(1+\beta)(\ell+p)}2^m\min\{2^{-N_0k_1}\norm{F_1}_{H^{N_0}},2^{k_1}\norm{F_1}_{\dot{H}^{-1}}\}\\
 &\qquad\cdot \min\{2^{-N_0k_2}\norm{F_2}_{H^{N_0}},2^{k_2}\norm{F_2}_{\dot{H}^{-1}}\}\\
 &\lesssim 2^{3k^++\frac{3}{2}k} 2^{m+(1+\beta)\delta m}\min\{2^{-N_0k_1},2^{k_1}\}\min\{2^{-N_0k_2},2^{k_2}\}\cdot\eps_1^2,
\end{aligned} 
\end{equation}
we see that it suffices to prove that if $\ell>(1+\delta)m$ then with $\delta_0=2N_0^{-1}\ll\delta^2$ we have
\begin{equation}
  2^{(1+\beta)\ell}2^{\beta p}\norm{P_{k,p}R_\ell\B_\m(P_{k_1}F_1,P_{k_2}F_2)}_{L^2}\lesssim 2^{-2\delta^2 m}\eps_1^2,\qquad -2m\leq k,k_j\leq \delta_0 m,\quad j=1,2.
\end{equation}
As in Section \ref{ssec:Bnorm-reduction} we can now further reduce cases by considering localizations in $p_j, \ell_j$, $j=1,2$, and see that to show the claim it suffices to establish that for $f_j=P_{k_j,p_j}R_{\ell_j}F_j$, $j=1,2$, 
when
\begin{equation}
 -2m\leq k,k_j\leq \delta_0 m,\qquad -2m\leq p_j\leq 0,\qquad -p_j\leq \ell_j\leq 2m,\quad j=1,2, 
\end{equation}
there holds that
\begin{equation}
 2^{(1+\beta)\ell}2^{\beta p}\norm{P_{k,p}R_\ell\B_\m(f_1,f_2)}_{L^2}\lesssim 2^{-\delta m}\eps_1^2.
\end{equation}
This is done in the following Case a and Case b.

\subsubsection*{Case a: $2^{p_1}+2^{p_2}\ll 1$.} 
In this case there holds that $\abs{\Phi}\gtrsim 1$ and a normal form (as in \eqref{eq:mdecompNF}--\eqref{mdecompNFNRNorms} with $\lambda=\frac{1}{10}$ so that $\mathfrak{m}^{res}=0$) gives
\begin{equation}\label{EZNF821a}
 \norm{P_{k,p}\B_{\m}(f_1,f_2)}_{L^2}\lesssim
 \norm{P_{k,p}\Q_{\m\cdot\Phi^{-1}}(f_1,f_2)}_{L^2}+\norm{P_{k,p}\B_{\m\cdot\Phi^{-1}}(\partial_t f_1,f_2)}_{L^2}+\norm{P_{k,p}\B_{\m\cdot\Phi^{-1}}( f_1,\partial_tf_2)}_{L^2}.
\end{equation}
A crude estimate using Lemma \ref{lem:dtfinL2} gives
\begin{equation*}
\begin{split}
 2^{(1+\beta)\ell}2^{\beta p}\norm{P_{k,p}R_\ell\B_{\m\cdot\Phi^{-1}}(\partial_t f_1,f_2)}_{L^2}&\lesssim  2^{\frac{3}{2}k}2^{(1+\beta)(\ell+p)}\norm{\mathcal{F}\left[P_{k,p}R_\ell\B_{\m\cdot\Phi^{-1}}(\partial_t f_1,f_2)\right]}_{L^\infty}\\
 &\lesssim 2^m\cdot 2^{(1+\beta)(\ell+p)}2^{\frac{5}{2}k}\norm{\partial_t f_1}_{L^2}\norm{f_2}_{L^2}\\
 &\lesssim 2^{m}\cdot 2^{-\frac{3}{2}m+\gamma m+(1+\beta)\delta m+3\delta_0 m}\eps_1^2
 \end{split}
\end{equation*}
and symmetrically for the term with $\partial_t f_2$. Assume now w.l.o.g.\ that $p_2\leq p_1$. For the boundary term a direct $L^2\times L^\infty$ estimate using Corollary \ref{cor:extrapol_decay} then gives the claim if $2^{p}\gtrsim 2^{p_2}$, since
\begin{equation*}
 2^{(1+\beta)\ell}2^{\beta p}\norm{P_{k,p}R_\ell\Q_{\m\cdot\Phi^{-1}}(f_1,f_2)}_{L^2}\lesssim 2^{k}2^{(1+\beta)\delta m}2^{-p}\cdot \norm{e^{it\Lambda} f_1}_{L^\infty} 2^{p_2}\norm{f_2}_{B}\lesssim 2^{-\frac{m}{2}}\eps_1^2.
\end{equation*}
We thus assume that $p\ll p_2\leq p_1$. If $p_1\leq -2\delta m$, using \eqref{BoundsOnMAddedBenoit}, we are done since
 \begin{equation*}
 \begin{split}
  2^{(1+\beta)\ell+\beta p}\norm{\Q_{\m\cdot\Phi^{-1}}(f_1,f_2)}_{L^2}&\lesssim 2^{(1+\beta)(\ell+ p)}2^{\frac{3}{2}k}\norm{\mathcal{F}\left[\Q_{\m\cdot\Phi^{-1}}(f_1,f_2)\right]}_{L^\infty}\\
  &\lesssim 2^{\frac{5}{2}k}2^{(1+\beta)\delta m}2^{3p_{\max}}\norm{f_1}_B\norm{f_2}_B.
  \end{split}
 \end{equation*}
Else we have $p_1\geq -2\delta m$, and thus $\abs{\bar\sigma}\gtrsim 2^{p_1+k_1+k}$ and using Lemma \ref{lem:new_ibp}, we can repeatedly integrate by parts in $V_\eta$ if 
\begin{equation*}
 2^{-2p_1+k_1-k}(1+2^{k_2-k_1}2^{\ell_1})<2^{(1-\delta)m}.
\end{equation*}
If this inequality is reversed, we use crude bounds: In case $k=k_{\min}$ we can assume that $-2p_1+\ell_1+k_1-k>(1-\delta)m$ and then, using \eqref{BoundsOnMAddedBenoit},
\begin{equation*}
\begin{split}
  2^{(1+\beta)\ell+\beta p}\norm{P_{k,p}R_\ell\Q_{\m\cdot\Phi^{-1}}(f_1,f_2)}_{L^2}&\lesssim   2^{\frac{3}{2}k}2^{(1+\beta)(\ell+ p)}\norm{\mathcal{F}\left[P_{k,p}R_\ell\Q_{\m\cdot\Phi^{-1}}(f_1,f_2)\right]}_{L^\infty}\\
  &\lesssim 2^{(1+\beta)\delta m}\cdot 2^{\frac{5k}{2}+p_1}\cdot 2^{-\ell_1}\norm{f_1}_X\cdot 2^{p_2}\norm{f_2}_{B}\\
  &\lesssim 2^{-m/3}\eps_1^2,
  \end{split}
\end{equation*}
whereas if $k_{\min}\in\{k_1,k_2\}$ we can assume that $-2p_1+\ell_1>(1-\delta)m$ and the conclusion follows just as above.

\subsubsection*{Case b: $2^{p_1}\sim 1$}
 Then $\abs{\bar\sigma}\sim 2^{k_1+k}$, and thus by Lemma \ref{lem:new_ibp}\eqref{it:ibp_p} repeated integration by parts in $V_\eta$ gives the claim if
 \begin{equation*}
  2^{k_1-k}(1+2^{k_2-k_1}2^{\ell_1})<2^{(1-\delta) m},
 \end{equation*}
 whereas, in the opposite case, a crude estimate using \eqref{BoundsOnMAddedBenoit} suffices
 \begin{equation*}
 \begin{split}
  2^{(1+\beta)\ell}2^{\beta p}\norm{P_{k,p}R_\ell \B_\m(f_1,f_2)}_{L^2}&\lesssim   2^{(1+\beta)(\ell+p)}2^{\frac{3}{2}k}\norm{\mathcal{F}\left[P_{k,p}R_\ell \B_\m(f_1,f_2)\right]}_{L^\infty}\\
  &\lesssim 2^{(1+2\delta)m}\cdot 2^{\frac{5}{2}k}\cdot 2^{-3k_1^+}2^{-(1+\beta)\ell_1} \norm{f_1}_{X}\norm{f_2}_{L^2}\\
  &\lesssim 2^{(1+3\delta)m}\cdot \min\{1,2^{(\frac{3}{2}-\beta)(k-k_1)}\}\cdot \min\{1,2^{-(1+\beta)(\ell_1-k+k_2)}\} \varepsilon_1^2.
 \end{split}
 \end{equation*}

\subsubsection{Case $\ell+p>\delta m$}\label{sec:Xnorm1-2}
By analogous reductions as at the beginning of Section \ref{sec:Xnorm1-1} it suffices to prove that with $f_j=P_{k_j,p_j}R_{\ell_j}F_j$, $j=1,2$, and when 
\begin{equation}\label{eq:Xnorm1-restrict1}
 -2\ell\leq k,k_j\leq \delta_0 \ell,\qquad -2\ell\leq p_j\leq 0,\qquad -p_j\leq \ell_j\leq 2\ell,\quad j=1,2, 
\end{equation}
(with $\delta_0:=2N_0^{-1}\ll \delta^2$ as above) there holds that
\begin{equation}
 2^{(1+\beta)\ell}2^{\beta p}\norm{P_{k,p}R_\ell\B_\m(f_1,f_2)}_{L^2}\lesssim 2^{-\delta^2 \ell}2^{-\delta^2m}\eps_1^2.
\end{equation}
(Note here that the reductions are given naturally in terms of the large parameter $\ell>m$).

A crude bound using \eqref{BoundsOnMAddedBenoit} gives that
\begin{equation}\label{SuffCondAddedBenoitAug}
 2^{(1+\beta)\ell}2^{\beta p}\norm{P_{k,p}R_\ell\B_\m(f_1,f_2)}_{L^2}\lesssim 2^m2^{(1+\beta)(\ell-\ell_1-\ell_2)}2^{\beta(p-p_1-p_2)}2^{k+k_{\max}+\frac{k_{\min}}{2}+p_{\min}}\Vert f_1\Vert_{X}\Vert f_2\Vert_X.
\end{equation}
When
\begin{equation*}
2^{m+p}+2^{k-k_j}2^{m+p_j}+2^{k-k_j+\ell_j}\le 2^{(1-\delta^2)\ell},\quad j=1,2,
\end{equation*}
we want to use the Bernstein property: As in \eqref{BernsteinConsequence}, we can rewrite
\begin{equation*}
\begin{split}
R^{(n)}_\ell\Q_\m(f_1,f_2)&=\sum_{a=1}^2 2^{-2\ell}R^{(n+1)}_\ell\Omega_{a3}^2\Q_\m(f_1,f_2),
\end{split}
\end{equation*}
for $R^{(0)}_\ell=R_\ell$ and where $R^{(n+1)}_\ell$ is bounded for all $n\ge1$. 
We find that
\begin{equation*}
\begin{split}
\Omega_{a3}^\xi\Q_{\m^{(j)}}(f_1,f_2)
=&\Q_{\m^{(j+1)}_1}(f_1,f_2)-\Q_{\m^{(j+1)}_2}(Sf_1,f_2)+\Q_{\m^{(j+1)}_3}(\Omega_{a3}f_1,f_2)
\end{split}
\end{equation*}
where for some coefficient matrices $A^r_{\alpha,\gamma}$, $r=2,3$,
\begin{equation*}
\begin{split}
\mathfrak{m}^{(j+1)}_1(\xi,\eta)&=\left[is(\Omega_{a3}^\xi\Phi(\xi,\eta))-3\right]\mathfrak{m}^{(j)}(\xi,\eta)+\Omega_{a3}^\xi\mathfrak{m}^{(j)}(\xi,\eta),\\
\mathfrak{m}^{(j+1)}_2(\xi,\eta)&=A_{2}^{\alpha,\gamma}\xi_\alpha(\xi-\eta)_\gamma\abs{\xi-\eta}^{-2}\mathfrak{m}^{(j)}(\xi,\eta),\\
\mathfrak{m}^{(j+1)}_3(\xi,\eta)&=A_{3}^{\alpha,\gamma}\xi_\alpha(\xi-\eta)_\gamma\abs{\xi-\eta}^{-2}\mathfrak{m}^{(j)}(\xi,\eta)\\
\end{split}
\end{equation*}
 and we see by induction that
\begin{equation*}
\begin{split}
\Vert \mathfrak{m}^{(j+1)}_1\Vert_{\widetilde{W}}&\lesssim 2^{k+p_{\max}}\cdot \left[2^{-p}+2^{m}(2^p+2^{k-k_1+p_1})+2^{k-k_1-p_1}\right]^j,\\
\Vert \mathfrak{m}^{(j+1)}_2\Vert_{\widetilde{W}}+\Vert \mathfrak{m}^{(j+1)}_3\Vert_{\widetilde{W}}&\lesssim   \Vert \mathfrak{m}^{(j)}\Vert_{\widetilde{W}}\cdot 2^{k-k_1}.
\end{split}
\end{equation*}
Iterating this at most $K$ times, stopping before once a term involves three $S$ derivatives and we use a crude estimate to conclude with \eqref{eq:R-Bernstein} for $f_1$ that
\begin{equation*}
\begin{split}
\Vert R_\ell\Q_\m(f_1,f_2)\Vert_{L^2}&\lesssim 2^{\frac{3}{2}k_{\min}}2^{k+p_{\max}}\cdot \sum_{a=0}^22^{-K\ell}\left[2^{-p}+2^{m}(2^p+2^{k-k_1+p_1})+2^{k-k_1-p_1}+2^{k-k_1+\ell_1}\right]^K\Vert S^af_1\Vert_{L^2}\Vert f_2\Vert_{L^2}\\
&\quad+2^{\frac{3}{2}k_{\min}}2^{k+p_{\max}}\cdot2^{-3\ell}\Vert S^3f_1\Vert_{L^2}\Vert f_2\Vert_{L^2}
\end{split}
\end{equation*}
which gives an acceptable contribution. Using this for $\xi-\eta$ and for $\eta$, and supposing wlog that $k_2\le k_1$, we obtain an acceptable contribution whenever
\begin{equation*}
\begin{split}
\ell_1\le (1-\delta^2)\ell\qquad\hbox{ or }\qquad k-k_2+\max\{m+p_2,\ell_2\}\le (1-\delta^2)\ell.
\end{split}
\end{equation*}
If $\ell_2\ge  m+p_2$, this and \eqref{SuffCondAddedBenoitAug} cover all the cases. In the opposite case, it remains to consider the case when
\begin{equation}\label{SuffCondLargeL}
\begin{split}
k_2\le k_1\sim k,\qquad \ell_1\ge(1-\delta^2)\ell,\qquad  (1+\delta)m\le \ell \le (1+2\delta^2)(m+p_2+k-k_2)
\end{split}
\end{equation}
which in particular implies that $p_2+k-k_2\ge\delta m/2$.

\medskip

Assume now that \eqref{SuffCondLargeL} holds and in addition,
\begin{equation*}
\begin{split}
2^{-p_2-p_{\max}}2^{k_2-k}+2^{-p_2-p_{\max}+\ell_2}\le 2^{(1-\delta^2)m}.
\end{split}
\end{equation*}
In this case, if $p_{\min}\ll p_{\max}$ we can integrate by parts along the vector field $V_{\xi-\eta}$ using Lemma \ref{lem:new_ibp}, or else an $L^2\times L^\infty$ estimate ($\ell_1\geq (1-\delta^2)\ell$) using Cor.\ \eqref{cor:decay} and $p_2+k-k_2\geq\frac{\delta m}{2}$ gives an acceptable contribution.

\medskip

If \eqref{SuffCondLargeL} holds and
\begin{equation*}
 -p_2-p_{\max}+k_2-k\geq (1-\delta^2)m-100,
\end{equation*}
a crude estimate using \eqref{BoundsOnMAddedBenoit} gives that
\begin{equation*}
\begin{split}
2^{3k^+}2^{(1+\beta)\ell}2^{\beta p}\Vert P_{k,p}R_\ell\B_\m(f_1,f_2)\Vert_{L^2}&\lesssim 2^m2^{4k^+}2^{p_{\max}}2^{(1+\beta)\ell}2^{\beta p}\cdot 2^{k+\frac{k_2}{2}}2^{p_{\min}}\cdot\Vert f_1\Vert_{L^2}\Vert f_2\Vert_{L^2}\\
&\lesssim 2^m2^{2k^+}2^{(1+\beta)(\ell-\ell_1)}2^{\beta (p-p_1)}2^{p_{\min}+p_{\max}+p_2}2^{\frac{k_2^-}{2}}\Vert f_1\Vert_X\Vert f_2\Vert_{B}\\
&\lesssim 2^m2^{2k^+}2^{2\delta^2\ell}2^{(1-\beta)p_{\min}+p_{\max}+p_2}2^{k_2^-}\varepsilon_1^2
\end{split}
\end{equation*}
which gives an acceptable contribution. Finally, if \eqref{SuffCondLargeL} holds and
\begin{equation*}
\begin{split}
 -p_2-p_{\max}+\ell_2\geq (1-\delta^2)m-100,
\end{split}
\end{equation*}
we estimate $f_2$ in the $X$ norm instead to get
\begin{equation*}
\begin{split}
2^{3k^+}2^{(1+\beta)\ell}2^{\beta p}\Vert P_{k,p}R_\ell\B_\m(f_1,f_2)\Vert_{L^2}&\lesssim 2^m2^{4k^+}2^{p_{\max}}2^{(1+\beta)\ell}2^{\beta p}\cdot 2^{k_2+\frac{k_1}{2}+\frac{p_1+p_2}{2}}\cdot\Vert f_1\Vert_{L^2}\Vert f_2\Vert_{L^2}\\
&\lesssim 2^m2^{2k^+}2^{(1+\beta)(\ell-\ell_1-\ell_2)}2^{\beta (p-p_1-p_2)}2^{\frac{p_1+p_2}{2}}2^{k_2^-}\Vert f_1\Vert_X\Vert f_2\Vert_{X}\\
&\lesssim 2^m2^{2k^+}2^{2\delta^2\ell}2^{-(1+\beta)(1-\delta^2)m}2^{(\frac{1}{2}-\beta)p_1}2^{-(\frac{1}{2}+\beta)p_2}2^{k_2^-}\varepsilon_1^2
\end{split}
\end{equation*}
and since $p_2+k-k_2\ge\delta m/2$, this also leads to an acceptable contribution. This covers all cases.

\end{proof}

\subsection{$X$ norm bounds for $\ell\leq (1+\delta)m$}\label{ssec:Xnorm2}
Next we prove the main bounds for the propagation of the $X$ norm. By Proposition \ref{prop:Xnorm1} it suffices to consider the case where $\ell<(1+\delta)m$. We will show the following:
\begin{proposition}\label{prop:Xnorm2}
Assume the bootstrap assumptions \eqref{eq:btstrap-assump} of Proposition \ref{prop:btstrap}, and let $\delta=2M^{-\frac{1}{2}}>0$.
Then for $F_j=S^{b_j}\U_{\mu_j}$, $0\leq b_1+b_2\leq N$, $\mu_j\in\{+,-\}$, $j=1,2$ there holds that 
\begin{equation}\label{eq:X-claim2}
 \sup_{k,\,\ell+p\geq 0,\,\ell\leq (1+\delta)m}2^{3k^+}2^{(1+\beta)\ell}2^{\beta p}\norm{P_{k,p}R_\ell\B_\m(F_1,F_2)}_{L^2}\lesssim 2^{-\delta^2 m}\eps_1^2.
\end{equation}
\end{proposition}

The remainder of this section is devoted to the proof of Proposition \ref{prop:Xnorm2}. After a standard reduction to ``atomic'' estimates with localized versions of the inputs, we will make ample use of the integration by parts along vector fields and normal forms. To this end, we note that by choice of $\delta$ we can repeatedly integrate by parts at least $M=O(\delta^{-2})\gg O(\delta^{-1})$ times.

We proceed in a similar fashion as in the proof of the $B$ norm bounds in Section \ref{sec:Bnorm}, but the estimates are more delicate since we always require a gain of $(2+\beta+)$ powers of the time variable. We use the possibility to integrate by parts along vector fields to push $\ell_i\sim (1-\delta)m$ up to losses in adjacent parameters $k_j,p_j$, then we use a normal form to gain a copy of $m$ at the cost of adjacent parameters.

\begin{proof}[Proof of Proposition \ref{prop:Xnorm2}]
We begin with a reduction: Note that if $\ell\leq (1+\delta)m$ then $(1+\beta)\ell\leq (1+\beta+2\delta)m$, and thus by energy estimates, further localizations in $p_j, \ell_j$, $j=1,2$, and $B$ resp.\ $X$ norm bounds it suffices to prove that (again with $\delta_0=2N_0^{-1}$) for
\begin{equation}
 f_j=P_{k_j,p_j}R_{\ell_j}F_j,\quad -2m\leq k,k_j\leq \delta_0 m,\qquad -2m\leq p_j\leq 0,\qquad -p_j\leq \ell_j\leq 2m,\quad j=1,2,\quad 
\end{equation}
we have that
\begin{equation}\label{eq:X-claim2'}
 \sup_{k,\,\ell+p\geq 0,\,\ell\leq (1+\delta)m}2^{(1+\beta+2\delta)m}2^{\beta p}\norm{P_{k,p}R_\ell\B_\m(f_1,f_2)}_{L^2}\lesssim 2^{-\delta m}\eps_1^2.
\end{equation}
This is the bound we shall prove in the rest of this section. Similar to the $B$ norm bounds we do this first in the setting of a gap in $p$ with $p_{\max}\sim 0$ (Section \ref{ssec:X-p-gap}), secondly when $p_{\max}\ll 0$ (Section \ref{ssec:X-pmax}), then for the case of a gap in $q$ (Section \ref{ssec:X-q-gap}) and finally for the case of no gaps (Section \ref{ssec:X-nogaps}).

\subsubsection{Gap in $p$, with $p_{\max}\sim 0$.}\label{ssec:X-p-gap}
We consider here the case where $p_{\min}\ll p_{\max}\sim 0$. We further subdivide according to whether the output $p$ or one of the inputs $p_i$ is small, and use Lemma \ref{lem:gapp-cases} to organize these cases. Wlog we assume that $p_1\leq p_2$, so that we have two main cases to consider.

Noting that $\abs{\bar\sigma}\sim 2^{p_{\max}}2^{k_{\min}+k_{\max}}$ and using that $\ell_i+p_i\geq 0$, by Lemma \ref{lem:new_ibp}\eqref{it:ibp_p} repeated integration by parts is feasible if
\begin{equation}\label{IBPVF82}
\begin{aligned}
 V_\eta:\qquad &2^{-p_1}2^{2k_1-k_{\min}-k_{\max}}(1+2^{k_2-k_1}2^{\ell_1})\leq 2^{(1-\delta)m},\\
 V_{\xi-\eta}:\qquad &2^{-p_2}2^{2k_2-k_{\min}-k_{\max}}(1+2^{k_1-k_2}2^{\ell_2})\leq 2^{(1-\delta)m}.
\end{aligned}
\end{equation}

\subsubsection*{Case 1: $p\ll p_1,p_2$}\label{ssec:X-p-gap-p}
By Lemma \ref{lem:gapp-cases} we have three scenarios to consider:
\subsubsection*{Subcase 1.1: $2^{k_1}\sim 2^{k_2}$.} Here we have $2^{p_1}\sim 2^{p_2}\sim 1$. 

As for \eqref{IBPVF82}, after repeated integration by parts ($O(\delta^{-1})\ll M$ times) we can assume that $\ell_i>(1-\delta)m+k-k_1$, $i=1,2$. Then a direct $X$ norm bound gives the claim: we have that
\begin{equation*}
\begin{aligned}
 \norm{P_{k,p}R_\ell \Q_\m(f_1,f_2)}_{L^2}&\lesssim 2^k\Sz 2^{-(1+\beta)(\ell_1+\ell_2)}\norm{f_1}_X\norm{f_2}_X\\
 &\lesssim 2^{(\frac{5}{2}-2-2\beta)k}2^{2(1+\beta)k_1}\cdot 2^{-2(1+\beta)(1-\delta)m}\norm{f_1}_X\norm{f_2}_X,
\end{aligned} 
\end{equation*}
and hence
\begin{equation*}
\begin{split}
 2^{3k^+}2^{(1+\beta)\ell}2^{\beta p}\norm{P_{k,p}R_\ell\B_\m(f_1,f_2)}_{L^2}
&\lesssim 2^{(1+\beta+2\delta)m}\norm{P_{k,p}R_\ell\B_\m(f_1,f_2)}_{L^2}\\
&\lesssim 2^{(2+\beta+2\delta)m}\norm{P_{k,p}R_\ell \Q_\m(f_1,f_2)}_{L^2}\lesssim 2^{-(\beta-2\delta(2+\beta)-\frac{5}{2}\delta_0)m}\eps_1^2,
 \end{split}
\end{equation*}
which suffices since $\beta\gg \delta$.

\subsubsection*{Subcase 1.2: $2^{k_2}\ll 2^{k_1}\sim 2^k$} Then we have that $2^{p_1-p_2}\sim 2^{k_2-k_1}\ll 1$, so that $p\ll p_1\ll p_2\sim 0$.

As in \eqref{IBPVF82}, by iterated integration by parts we can assume that $\ell_2>(1-\delta)m$ and $\ell_1>p_1+(1-\delta)m$, which suffices for a direct $X$ norm bound provided that $\beta\leq \frac{1}{4}$,
\begin{equation}
\begin{aligned}
 \norm{P_{k,p}R_\ell \Q_\m(f_1,f_2)}_{L^2}&\lesssim 2^k\Sz 2^{-(1+\beta)(\ell_1+\ell_2)}2^{-\beta p_1}\norm{f_1}_X\norm{f_2}_X\\
 &\lesssim 2^{k}\cdot 2^{\frac{3}{2}k_2}\cdot 2^{-(1+2\beta)p_1}\cdot 2^{-2(1+\beta)(1-\delta)m}\varepsilon_1^2\\
 &\lesssim 2^{\frac{5}{2}k_1} 2^{(\frac{1}{2}-2\beta)p_1}\cdot 2^{-2(1+\beta)(1-\delta)m}\varepsilon_1^2,
\end{aligned} 
\end{equation}
using that $2^{k_2}\sim 2^{k_1+p_1}$. This leads to an acceptable contribution.

\subsubsection*{Subcase 1.3: $2^{k_1}\ll 2^{k_2}\sim 2^k$} This would imply $p_2\ll p_1$, which is excluded by assumption.

\subsubsection*{Case 2: $p_1\ll p,p_2$}\label{ssec:X-p-gap-p_1}
By Lemma \ref{lem:gapp-cases} we have three scenarios to consider:
\subsubsection*{Subcase 2.1: $2^{k}\sim 2^{k_2}$} Then $2^{p}\sim 2^{p_2}\sim 1$. Using \eqref{IBPVF82}, we can assume that $\ell_1-p_1\geq (1-\delta)m$ and $\max\{k_2-k_1,\ell_2\}\geq (1-\delta)m$.

\begin{enumerate}[label=(\alph*),wide]
\item $k_2-k_1>(1-\delta)m$. Here we have that since $\Sz\lesssim 2^{p_1}2^{\frac{3}{2}k_1}$ there holds
\begin{equation*}
\begin{split}
 \norm{P_{k,p}R_\ell \Q_\m(f_1,f_2)}_{L^2}&\lesssim 2^k\Sz 2^{-\ell_1}\norm{f_1}_X\norm{f_2}_{L^2}\lesssim 2^{k+\frac{3}{2}k_1-N_0k_2^+}2^{-(1-\delta)m}\varepsilon_1^2,
 \end{split}
\end{equation*}
which is more than enough thanks to the smallness of $k_1-k_2$.
 
\item $\ell_2>(1-\delta)m$.  Here we will further split cases towards a normal form. Assume first that
\begin{equation}\label{p1NotTooSmall82}
\begin{split}
2p_1-k_1\ge -k-100.
\end{split}
\end{equation}
In this case, a crude estimate gives
\begin{equation*}
\begin{split}
 \norm{P_{k,p}R_\ell \Q_\m(f_1,f_2)}_{L^2}&\lesssim 2^k\Sz\cdot \norm{f_1}_{L^2}\norm{f_2}_{L^2}\\
 &\lesssim 2^{k+\frac{3}{2}k_1+p_1} \cdot 2^{-(1+2\beta)p_1}2^{-2(1+\beta)(1-\delta)m}\norm{f_1}_X\norm{f_2}_X\\
 &\lesssim 2^{(1+\beta)k}2^{(\frac{3}{2}-\beta)k_1}2^{\beta (k_1-k-2p_1)}2^{-2(1+\beta)(1-\delta)m}\varepsilon_1^2,
 \end{split}
\end{equation*}
which gives an acceptable contribution.

We now assume that \eqref{p1NotTooSmall82} does not hold and do a normal form away from the resonant set (see also Section \ref{ssec:nfs}). For $\lambda=2^{-200\delta m}$, we decompose as in \eqref{eq:mdecompNF}
\begin{equation}\label{822DecompositionResNRes}
\begin{split}
\mathfrak{m}(\xi,\eta)=\psi(\lambda^{-1}\Phi)\mathfrak{m}(\xi,\eta)+(1-\psi(\lambda^{-1}\Phi))\mathfrak{m}(\xi,\eta)=\mathfrak{m}^{res}(\xi,\eta)+\mathfrak{m}^{nr}(\xi,\eta).
\end{split}
\end{equation}
On the support of $\mathfrak{m}^{res}$, using (the contraposite of) \eqref{p1NotTooSmall82}, we observe that
\begin{equation*}
\begin{split}
 \vert\partial_{\eta_3}\Phi(\xi,\eta)\vert\gtrsim 2^{-k}
\end{split}
\end{equation*}
and using Lemma \ref{lem:nfs}\eqref{it:NF-bd34} we find that
\begin{equation}\label{eq:gap-p-2.1}
\begin{aligned}
 \norm{P_{k,p}R_\ell \B_{\m^{res}}(f_1,f_2)}_{L^2}&\lesssim 2^m\cdot 2^k2^{p_1+k_1}\lambda^{1/2}2^{\frac{k_2}{2}}\cdot 2^{-\ell_1}2^{-(1+\beta)\ell_2}\norm{f_1}_{X}\norm{f_2}_{X}\\
 &\lesssim 2^{k+k_1+\frac{k_2}{2}}\cdot \lambda^{1/2}\cdot 2^{m-(2+\beta)(1-\delta)m}\eps_1^2
\end{aligned} 
\end{equation}
and again, we obtain an acceptable contribution. On the support of $\mathfrak{m}^{nr}$, the phase is large and we can perform a normal and see as in \eqref{mdecompNFNRNorms} that
\begin{equation}\label{eq:gap-p-nfsplit}
\begin{aligned}
 \norm{P_{k,p}R_\ell \B_{\m^{nr}}(f_1,f_2)}_{L^2}&\lesssim \norm{P_{k,p}R_\ell \Q_{\Phi^{-1}\m^{nr}}(f_1,f_2)}_{L^2} + \norm{P_{k,p}R_\ell \B_{\Phi^{-1}\m^{nr}}(\partial_t f_1,f_2)}_{L^2}\\
 &\qquad +\norm{P_{k,p}R_\ell \B_{\Phi^{-1}\m^{nr}}(f_1,\partial_t f_2)}_{L^2},
\end{aligned} 
\end{equation}
and using crude estimates and Lemma \ref{lem:nfs}\eqref{it:NF-bd1}, we see that
\begin{equation*}
\begin{aligned}
 \norm{P_{k,p}R_\ell \Q_{\Phi^{-1}\m^{nr}}(f_1,f_2)}_{L^2}&\lesssim 2^k2^{p_1+\frac{3}{2}k_1}\lambda^{-1}\cdot 2^{-\ell_1}2^{-(1+\beta)\ell_2}\norm{f_1}_{X}\norm{f_2}_{X}\\
 &\lesssim 2^{k+\frac{3}{2}k_1}\cdot 2^{200\delta m-(2+\beta)(1-\delta)m}\eps_1^2,
\end{aligned} 
\end{equation*}
which suffices since $\beta\gg \delta$. Similarly, using Lemma \ref{lem:dtfinL2}, we obtain that
\begin{equation}\label{eq:gap-p-2.1end}
\begin{aligned}
 \norm{P_{k,p}R_\ell \B_{\Phi^{-1}\m^{nr}}(f_1,\partial_t f_2)}_{L^2}&\lesssim 2^m\cdot 2^k2^{p_1+\frac{3}{2}k_1}\lambda^{-1} 2^{-\ell_1}\norm{f_1}_{X}\norm{\partial_t f_2}_{L^2}\\
 &\lesssim 2^{k+k_1+\frac{k_2}{2}}\cdot 2^{\gamma+300\delta m-\frac{3}{2}m}\eps_1^2,
\end{aligned}
\end{equation}
and similarly for the term with $\partial_t f_1$.

\end{enumerate}

\subsubsection*{Subcase 2.2: $2^{k}\ll 2^{k_2}\sim 2^{k_1}$} Then $2^{p_2-p}\sim 2^{k-k_2}\ll 1$, and thus $p_1\ll p_2\ll p\sim 0$.

After repeated integration by parts we may assume that
\begin{equation}\label{Suff822}
\ell_i\ge\max\{p_i-k_1+k+(1-\delta)m,-p_i\},\qquad i\in\{1,2\}.
\end{equation}
This is sufficient if
\begin{equation*}
\begin{split}
p_{1}\ge -10\delta m.
\end{split}
\end{equation*}
We first localize the analysis to the resonant set by decomposing $\m(\xi,\eta)=\m^{res}(\xi,\eta)+\m^{nr}(\xi,\eta)$ as in \eqref{eq:mdecompNF} with $\lambda=2^{-100}(2^q+2^{2p_2})$.
For the nonresonant terms, we can do a normal form as in \eqref{mdecompNFNRNorms}, and with Lemma \ref{lem:nfs}\eqref{it:NF-bd1} a crude estimate gives
\begin{equation*}
\begin{aligned}
 \norm{P_{k,p,q}R_\ell \Q_{\Phi^{-1}\m^{nr}}(f_1,f_2)}_{L^2}&\lesssim \Sz\cdot 2^k\cdot (2^{q}+2^{2p_2})^{-1}\cdot \Vert f_1\Vert_{L^2}\Vert f_2\Vert_{L^2}\\
 &\lesssim 2^{k_1+\frac{3}{2}k}2^{p_1+\frac{q}{2}}\cdot (2^q+2^{2p_2})^{-1}\cdot 2^{-\ell_1-(1+\beta)\ell_2-\beta p_2}\Vert f_1\Vert_{X}\Vert f_2\Vert_X\\
 &\lesssim 2^{-(1+\beta+2\delta)m}\varepsilon_1^2\cdot 2^{p_2+\frac{q}{2}}(2^q+2^{2p_2})^{-1}\cdot 2^{(\beta+3\delta)m}2^{2k_1}2^{\frac{1}{2}k}2^{-(1+\beta)(\ell_2+p_2)}.
\end{aligned} 
\end{equation*}
If
\begin{equation*}
\begin{split}
\max\{k_1-k,\ell_2+p_2\}\ge 5\beta m
\end{split}
\end{equation*}
this gives an acceptable contribution; else, using that $2^{p_2}\sim 2^{k-k_1}$, we obtain a contradiction with \eqref{Suff822}.
In addition, another use of Lemma \ref{lem:dtfinL2} and Lemma \ref{lem:nfs}\eqref{it:NF-bd1} gives
\begin{equation*}
\begin{aligned}
 \norm{P_{k,p,q}R_\ell \B_{\Phi^{-1}\m^{nr}}(f_1,\partial_t f_2)}_{L^2}&\lesssim 2^m\cdot 2^k(2^{q}+2^{2p_2})^{-1}\cdot \Sz\cdot \norm{f_1}_{L^2}\norm{\partial_t f_2}_{L^2}\\
 &\lesssim 2^m\cdot (2^{q}+2^{2p_2})^{-1}2^{\frac{q}{2}}2^{\frac{3}{2}k+k_1}\cdot 2^{p_1-\ell_1}2^{-\frac{3}{2}m+\gamma m}\cdot \eps_1^3.
\end{aligned} 
\end{equation*}
Now, using that $(2^{q}+2^{2p_2})^{-1}2^{\frac{q}{2}}2^{k-k_1}\le 2^{-p_2}2^{k-k_1}\lesssim 1$, we see that if $p_1\le -m/2$, then
\begin{equation*}
\begin{aligned}
 \sum_q\norm{P_{k,p,q}R_\ell \B_{\Phi^{-1}\m^{nr}}(f_1,\partial_t f_2)}_{L^2} &\lesssim 2^m\cdot 2^{\frac{1}{2}k+2k_1}\cdot 2^{2p_1}2^{-\frac{3}{2}m+\gamma m}\cdot \eps_1^3\\
\end{aligned} 
\end{equation*}
which gives an acceptable contribution, while if $-m/2\le p_1\le p_2$, we see that
\begin{equation*}
\begin{aligned}
 \sum_q\norm{P_{k,p,q}R_\ell \B_{\Phi^{-1}\m^{nr}}(f_1,\partial_t f_2)}_{L^2} &\lesssim 2^m\cdot 2^{-p_2}\cdot 2^{\frac{1}{2}k+2k_1}\cdot 2^{k-k_1+p_1-\ell_1}2^{-\frac{3}{2}m+\gamma m}\cdot \eps_1^3\\
 &\lesssim 2^{-(\frac{3}{2}-\gamma-\delta)m}\cdot 2^{-\frac{1}{2}p_2}\cdot 2^{\frac{5}{2}k_1}\cdot \eps_1^3,
\end{aligned} 
\end{equation*}
which is acceptable. The term involving $\partial_tf_1$ is treated similarly.

\medskip

We now turn to the resonant term. First, we observe that, on the support of $\mathfrak{m}^{res}$,
\begin{equation*}
\begin{split}
\vert \Lambda(\xi-\eta)\vert+\vert\Lambda(\eta)\vert\ge 3/2,\qquad
\vert \Lambda(\xi-\eta)\vert-\vert\Lambda(\eta)\vert=\frac{\sqrt{1-\Lambda^2(\eta)}^2-\sqrt{1-\Lambda^2(\xi-\eta)}^2}{\vert\Lambda(\xi-\eta)\vert+\vert\Lambda(\eta)\vert}\ge 2^{-8}2^{2p_2},
\end{split}
\end{equation*}
so that smallness of $\vert\Phi\vert$ implies that $2^q\sim 2^{2p_2}\sim 2^{2(k-k_1)}$, but we will need to restrict the support further. 

\medskip

We first observe that since $\vert k_1-k_2\vert\le 10$ and $p_1\ll p_2$, we have that, on the support of $\mathfrak{m}^{res}$,
\begin{equation*}
\begin{split}
\vert \partial_{\eta_3}\Phi(\xi,\eta)\vert\gtrsim 2^{2p_2-k_2}.
\end{split}
\end{equation*}
and we can use the analysis in Section \ref{ssec:D3ibp} to obtain an acceptable contribution unless we have
\begin{equation*}
\begin{split}
\max\{\ell_1+p_1-2p_2,\ell_2-p_2\}\ge (1-\delta)m,
\end{split}
\end{equation*}
which improves upon \eqref{Suff822} in that it does not incur $k$ losses. If the first term is largest, a crude estimate gives that
\begin{equation*}
\begin{aligned}
 \norm{P_{k,p,q}R_\ell \B_{\m^{res}}(f_1,f_2)}_{L^2}&\lesssim 2^m\cdot 2^k\Sz \cdot\norm{f_1}_{L^2}\norm{f_2}_{L^2}\\
 &\lesssim 2^m\cdot 2^{\frac{5}{2}k}2^{\frac{q}{2}}2^{-(1+\beta)\ell_1-\beta p_1}2^{-(1+\beta)\ell_2-\beta p_2}\Vert f_1\Vert_{X}\Vert f_2\Vert_X\\
 &\lesssim 2^{-(1+2\beta-3\delta)m}\cdot 2^{p_1-p_2}\cdot 2^{-(1+4\beta)p_2}2^{(\frac{3}{2}-\beta)k} 2^{(2+\beta)k_1}\varepsilon_1^2,
\end{aligned} 
\end{equation*}
and since $2^{k-k_1}\lesssim 2^{p_2}$, we obtain an acceptable contribution. Thus from now on, we may assume that
\begin{equation*}
\begin{split}
\ell_1-p_1-k+k_1\ge (1-\delta)m,\qquad \ell_2-p_2\ge(1-\delta)m,\qquad 2^{\frac{q}{2}}\sim 2^{p_2}\sim 2^{k-k_1}.
\end{split}
\end{equation*}

\medskip

In this case, a crude estimate gives that
\begin{equation*}
\begin{aligned}
 \norm{P_{k,p,q}R_\ell \B_{\m^{res}}(f_1,f_2)}_{L^2}&\lesssim 2^m\cdot 2^k\Sz \cdot\norm{f_1}_{L^2}\norm{f_2}_{L^2}\\
 &\lesssim 2^m\cdot 2^{\frac{3}{2}k+k_1}2^{p_1+\frac{q}{2}}2^{-\ell_1}2^{-\beta (\ell_1+p_1)}2^{-(1+\beta)\ell_2}2^{-\beta p_2}\Vert f_1\Vert_{X}\Vert f_2\Vert_X\\
 &\lesssim 2^{-(1+\beta-3\delta)m}\cdot 2^{-\beta(\ell_1+p_1)}\cdot 2^{(\frac{1}{2}-\beta)k}2^{-2\beta p_2}\cdot 2^{(2+\beta)k_1}\varepsilon_1^2
\end{aligned} 
\end{equation*}
and this leads to an acceptable contribution whenever
\begin{equation}\label{822ParametersNotSoSmall}
 p_2\le -40\delta m.
\end{equation}
In the opposite case, we do another normal form, choosing a smaller phase restriction $\lambda=2^{-300\delta m}$. Thus we set
\begin{equation*}
\begin{split}
\mathfrak{m}^{res}(\xi,\eta)&:=\psi(\lambda^{-1}\Phi(\xi,\eta))\mathfrak{m}^{res}(\xi,\eta)+(1-\psi(\lambda^{-1}\Phi(\xi,\eta)))\mathfrak{m}^{res}(\xi,\eta)=\mathfrak{m}^{rr}(\xi,\eta)+\mathfrak{m}^{nrr}(\xi,\eta).
\end{split}
\end{equation*}
On the support of $\mathfrak{m}^{rr}$, we have that
\begin{equation*}
\begin{split}
\vert\partial_{\eta_3}\Phi(\xi,\eta)\vert\gtrsim 2^{2p_2-k_2}\gtrsim 2^{-100\delta m}
\end{split}
\end{equation*}
and using Lemma \ref{lem:set_gain2}, as in Lemma \ref{lem:nfs}\eqref{it:NF-bd34} we see that
\begin{equation*}
\begin{split}
 \norm{P_{k,p,q}R_\ell \B_{\m^{rr}}(f_1,f_2)}_{L^2}&\lesssim 2^{m}\cdot 2^{k+k_1}(2^{100\delta m}\lambda)^{\frac{1}{2}}2^{p_1}\cdot \Vert f_1\Vert_{L^2}\Vert f_2\Vert_{L^2}\\
 &\lesssim 2^{m}\cdot 2^{k_1}2^{-100\delta m}\cdot 2^{p_1+k-\ell_1}\cdot 2^{-(1+\beta)(\ell_2-p_2)}2^{-(1+2\beta) p_2}\Vert f_1\Vert_X\Vert f_2\Vert_X\\
 &\lesssim 2^{-(1+\beta-3\delta+100\delta)m}\cdot 2^{2k_1}\cdot 2^{-(1+2\beta)p_2}\varepsilon_1^2,
\end{split}
\end{equation*} 
which gives an acceptable contribution using \eqref{822ParametersNotSoSmall}. Independently, we treat the nonresonant term via a normal form as in \eqref{mdecompNFNRNorms}. First a crude estimate using Lemma \ref{lem:nfs}\eqref{it:NF-bd1} gives that
\begin{equation*}
\begin{aligned}
 \norm{P_{k,p,q}R_\ell \Q_{\Phi^{-1}\m^{nrr}}(f_1,f_2)}_{L^2} &\lesssim 2^k\lambda^{-1}\cdot 2^{\frac{1}{2}k+k_1}2^{p_1+\frac{q}{2}}\cdot \Vert f_1\Vert_{L^2}\Vert f_2\Vert_{L^2}\\
 &\lesssim 2^{\frac{1}{2}k}2^{500\delta m}\cdot 2^{p_1+k-\ell_1}2^{p_2-\ell_2}\Vert f_1\Vert_X\Vert f_2\Vert_X\\
 &\lesssim 2^{-(2-600\delta)m}2^{\frac{1}{2}k+2k_1}\varepsilon_1^2,
\end{aligned} 
\end{equation*}
which is again acceptable. In addition, Lemma \ref{lem:dtfinL2} and Lemma \ref{lem:nfs}\eqref{it:NF-bd1} give
\begin{equation*}
\begin{aligned}
 \norm{P_{k,p,q}R_\ell \B_{\Phi^{-1}\m^{nrr}}(f_1,\partial_t f_2)}_{L^2}&\lesssim 2^m\cdot 2^k\lambda^{-1}\cdot 2^{\frac{1}{2}k+k_1}2^{p_1+\frac{q}{2}}\cdot \Vert f_1\Vert_{L^2}\Vert \partial_tf_2\Vert_{L^2}\\
 &\lesssim 2^{\frac{1}{2}k}\cdot 2^{p_1+k-\ell_1}\cdot 2^{-(\frac{1}{2}-\gamma-500\delta)m}\Vert f_1\Vert_X\varepsilon_1^2\\
 &\lesssim 2^{-(\frac{3}{2}-\gamma-500\delta)m}2^{\frac{1}{2}k+k_1}\varepsilon_1^3,
\end{aligned} 
\end{equation*}
and once again the term involving $\partial_tf_1$ is easier.

\subsubsection*{Subcase 2.3: $2^{k_2}\ll 2^{k}\sim 2^{k_1}$} Then $2^{p-p_2}\sim 2^{k_2-k}\ll 1$, and thus $p_1\ll p\ll p_2\sim 0$.

Using Lemma \ref{lem:new_ibp}, repeated integration by parts give the result unless
\begin{equation}\label{SuffCond823}
\ell_2\geq (1-\delta)m\qquad\hbox{ and }\qquad -p_1+\ell_1\geq (1-\delta)m.
\end{equation}
A crude estimate gives that
\begin{equation*}
\begin{split}
 \norm{P_{k,p}R_\ell \B_{\m}(f_1,f_2)}_{L^2}&\lesssim 2^m\cdot 2^k\cdot 2^{k_1+p_1}2^{\frac{1}{2}k_2}\cdot 2^{-\ell_1}2^{-(1+\beta)\ell_2}\norm{f_1}_{X}\norm{f_2}_{X}\\
 &\lesssim 2^{-(1+\beta-3\delta)m}\cdot 2^{2k+\frac{1}{2}k_2}\eps_1^2,
\end{split} 
\end{equation*}
and this gives an acceptable contribution unless $k_2\geq -10\delta m$. In this case, we split again into resonant and nonresonant regions with $\lambda=2^{-200\delta m}$ as in \eqref{eq:mdecompNF}. On the support of the resonant term, we see that (since $p_1\ll p_2$, $k_2\ll k_1$)
\begin{equation*}
\begin{split}
 \vert\partial_{\eta_3}\Phi(\xi,\eta)\vert\gtrsim 2^{-25\delta m},
\end{split}
\end{equation*}
and using crude estimates and Lemma \ref{lem:nfs}\eqref{it:NF-bd34}, we see that
\begin{equation*}
\begin{aligned}
 \norm{P_{k,p}R_\ell \B_{\m^{res}}(f_1,f_2)}_{L^2}&\lesssim 2^m\cdot 2^k\cdot 2^{k_1+p_1}(2^{25\delta m}\lambda)^\frac{1}{2}\cdot 2^{-\ell_1}2^{-(1+\beta)\ell_2}\norm{f_1}_{X}\norm{f_2}_{X}\\
 &\lesssim 2^{-(2+\beta+50\delta)m}2^{2k}\eps_1^2.
\end{aligned} 
\end{equation*}
For the nonresonant term, we use a normal form as in \eqref{mdecompNFNRNorms}. Using crude estimates and Lemma \ref{lem:nfs}\eqref{it:NF-bd1}, we see that
\begin{equation*}
\begin{aligned}
 \norm{P_{k,p}R_\ell \Q_{\Phi^{-1}\m^{nr}}(f_1,f_2)}_{L^2}&\lesssim 2^k2^{p_1+\frac{3}{2}k_1}\lambda^{-1}\cdot 2^{-\ell_1}2^{-(1+\beta)\ell_2}\norm{f_1}_{X}\norm{f_2}_{X}\\
 &\lesssim 2^{\frac{5}{2}k}\cdot 2^{200\delta m-(2+\beta)(1-\delta)m}\eps_1^2,
\end{aligned} 
\end{equation*}
which suffices since $\beta\gg \delta$. Similarly, using Lemma \ref{lem:dtfinL2}, we obtain that
\begin{equation}
\begin{aligned}
 \norm{P_{k,p}R_\ell \B_{\Phi^{-1}\m^{nr}}(f_1,\partial_t f_2)}_{L^2}&\lesssim 2^m\cdot 2^k2^{p_1+\frac{3}{2}k_1}\lambda^{-1} 2^{-\ell_1}\norm{f_1}_{X}\norm{\partial_t f_2}_{L^2}\\
 &\lesssim 2^{\frac{5}{2}k}\cdot 2^{\gamma+300\delta m-\frac{3}{2}m}\eps_1^2,
\end{aligned}
\end{equation}
and similarly for the symmetric case, and once again, we obtain an acceptable contribution.

\medskip
\subsubsection{Case $p_{\max}\ll 0$}\label{ssec:X-pmax}
In case $2^{p_{\max}}\ll 1$ we have that $\abs{\Phi}\gtrsim 1$, and we can do a normal form as in \eqref{eq:mdecompNF}--\eqref{mdecompNFNRNorms} with $\lambda=\frac{1}{10}$ so that $\mathfrak{m}^{res}=0$. Using Lemma \ref{lem:interpol}, we have that $\norm{e^{it\Lambda}f_i}_{L^\infty}\lesssim 2^{-\frac{2}{3}m}\eps_1$, $i=1,2$, and thus by Lemma \ref{lem:nfs}\eqref{it:NF-bd2}
\begin{equation}
 \norm{P_{k,p}R_\ell \B_{\m\cdot\Phi^{-1}}(\partial_t f_1,f_2)}_{L^2}\lesssim 2^m\cdot 2^k\norm{\partial_t f_1}_{L^2}\norm{e^{it\Lambda} f_2}_{L^\infty}\lesssim 2^{k+\frac{3}{2}k_2}2^{m-\frac{3}{2}m+\gamma m-\frac{2}{3}m}\eps_1^2
\end{equation}
and symmetrically for $B_{\m\cdot\Phi^{-1}}( f_1,\partial_tf_2)$. The boundary term requires a bit more care: Assuming w.l.o.g.\ that $p_2\leq p_1$, we distinguish two cases:
\begin{itemize}[wide]
 \item If $f_1$ has fewer vector fields than $f_2$, by Proposition \ref{prop:decay} (and again Lemma \ref{lem:nfs}\eqref{it:NF-bd2}) there holds that
 \begin{equation}
  \norm{P_{k,p}\Q_{\m\cdot\Phi^{-1}}(f_1,f_2)}_{L^2}\lesssim 2^{k} \norm{e^{it\Lambda}f_1^{(1)}}_{L^\infty}2^{p_2}\norm{f_2}_{B}+2^k\norm{e^{it\Lambda}f_1^{(2)}}_{L^2}\norm{e^{it\Lambda} f_2}_{L^\infty}\lesssim 2^{-\frac{3}{2}m}\eps_1^2,
 \end{equation}
 and analogously if $2^{p_2}\gtrsim 2^{p_1}$.
 \item If $f_1$ has more vector fields than $f_2$ and $p_2\ll p_1$, we note that since $\abs{\bar\sigma}\sim 2^{p_{\max}}2^{k_{\max}+k_{\min}}$ we have that repeated integration by parts in $V_\eta$ gives the claim if
 \begin{equation}
  2^{-p_1-p_{\max}}2^{2k_1-k_{\max}-k_{\min}}(1+2^{k_2-k_1}2^{\ell_1})<2^{(1-\delta) m}.
 \end{equation}
 Otherwise we are done by a standard $L^2\times L^\infty$ estimate, using the localization information. The most difficult term is when $k=k_{\min}$, where we can assume that $-p_1-p_{\max}+k_1-k+\ell_1>(1-\delta)m$ and obtain
 \begin{equation}
 \begin{aligned}
  \norm{P_{k,p}R_\ell \Q_{\m\cdot\Phi^{-1}}(f_1,f_2)}_{L^2}&\lesssim 2^{k+p_{\max}}\norm{f_1}_{L^2}\norm{e^{it\Lambda}f_2}_{L^\infty}\lesssim 2^{k+p_{\max}}2^{-\frac{\ell_1+p_1}{2}}\norm{f_1}_X^{\frac{1}{2}}\norm{f_1}_B^{\frac{1}{2}}\norm{e^{it\Lambda}f_2}_{L^\infty}\\
  &\lesssim 2^{\frac{k+p_{\max}}{2}+\frac{k_1}{2}}\cdot 2^{-\frac{1-\delta}{2}m-m}\cdot \eps_1^2,
 \end{aligned} 
 \end{equation}
 an acceptable contribution.
\end{itemize}

\subsubsection{Gap in $q$.}\label{ssec:X-q-gap}
We additionally localize in $q_i$, write $g_i=P_{k_i,p_i,q_i}R_{\ell_i}f_i$, $i=1,2$. A crude estimate using \eqref{BoundsOnMAddedBenoit} gives that
\begin{equation*}
\begin{split}
\norm{P_{k,p,q}R_\ell \B_{\m}(g_1,g_2)}_{L^2}&\lesssim 2^m\cdot 2^{k+q_{\max}}\cdot 2^{\frac{3}{2}k_{\max}+\frac{q_{\min}}{2}}\cdot\Vert g_1\Vert_{L^2}\Vert g_2\Vert_{L^2}\\
&\lesssim 2^m\cdot 2^{\frac{5}{2}k_{\max}}2^{\frac{q_{\min}+q_1+q_2}{2}+q_{\max}}\cdot\Vert g_1\Vert_B\Vert g_2\Vert_B,
\end{split}
\end{equation*}
and we obtain acceptable contributions unless
\begin{equation}\label{CrudeBoundq1}
q_{\min}\ge-10 m,\qquad q_{\max}\ge -6m/7,
\end{equation}
and in particular, we have at most $O(m^3)$ choices for $\{q,q_1,q_2\}$.

In this section, we assume that $q_{\min}\ll q_{\max}$ and (by the previous case) $2^{p_{\min}}\sim 2^{p_{\max}}\sim 1$. Using Lemma \ref{lem:new_ibp} and noting that $\abs{\bar\sigma}\sim 2^{q_{\max}}2^{k_{\min}+k_{\max}}$, repeated integration by parts is allow us to deal with the case when
\begin{equation}\label{ConsLem57Gapq}
\begin{aligned}
 V_\eta:\qquad &2^{-q_{\max}}2^{2k_1-k_{\min}-k_{\max}}(1+2^{k_2-k_1}(2^{q_2-q_1}+2^{\ell_1}))\leq 2^{(1-\delta)m},\\
 V_{\xi-\eta}:\qquad &2^{-q_{\max}}2^{2k_2-k_{\min}-k_{\max}}(1+2^{k_1-k_2}(2^{q_1-q_2}+2^{\ell_2}))\leq 2^{(1-\delta)m}.
\end{aligned}
\end{equation}

Wlog we assume that $q_1\leq q_2$, so that we have two main cases to consider:

\subsubsection*{Case 3: $q\ll q_1,q_2$} 
By Lemma \ref{lem:gapp-cases} we have two scenarios to consider:

\subsubsection*{Subcase 3.1: $2^{k_1}\sim 2^{k_2}$} Then also $2^{q_1}\sim 2^{q_2}$. 

Using \eqref{ConsLem57Gapq}, we see that we can assume $-q_{1}+k_1-k+\min\{\ell_1,\ell_2\}>(1-\delta)m$.
We now want to use the precised decay estimate. Assuming wlog that $g_2$ has fewer vector fields than $g_1$, we recall that by Proposition \ref{prop:decay} we have
\begin{equation}
 e^{it\Lambda}g_2=e^{it\Lambda}g_2^{(1)}+e^{it\Lambda}g_2^{(2)},
\end{equation}
with
\begin{equation}
 \norm{e^{it\Lambda}g_2^{(1)}}_{L^\infty}\lesssim 2^{\frac{3}{2}k_2-\frac{q_2}{2}}t^{-\frac{3}{2}}\eps_1,\qquad \norm{e^{it\Lambda}g_2^{(2)}}_{L^2}\lesssim t^{-1-\beta'}\mathfrak{1}_{\{q_2\gtrsim -m\}}\eps_1.
\end{equation}
Using a simple $L^\infty\times L^2$ bound with \eqref{BoundsOnMAddedBenoit}, we get
\begin{equation}
\begin{aligned}
 \norm{P_{k,p,q}R_\ell \Q_{\m}(g_1,g_2^{(1)})}_{L^2}&\lesssim 2^{k+q_2}\norm{g_1}_{L^2}\norm{e^{it\Lambda}g_2^{(1)}}_{L^\infty}\\
 &\lesssim 2^{-\frac{3}{2}m}2^{k+\frac{q_1}{2}}\min\{2^{q_1},2^{-(1+\beta)\ell_1}\}\varepsilon_1^2\lesssim 2^{-\frac{9}{4}m}\varepsilon_1^2
\end{aligned} 
\end{equation}
and using a crude estimate with Lemma \ref{lem:set_gain},
\begin{equation}
\begin{aligned}
 \norm{P_{k,p,q}R_\ell \Q_{\m}(g_1,g_2^{(2)})}_{L^2}&\lesssim 2^{k+q_2}\Sz\cdot\norm{g_1}_{L^2}\norm{e^{it\Lambda}g_2^{(2)}}_{L^2} \\
 &\lesssim 2^{2k+\frac{k_2}{2}+\frac{3}{2}q_2}2^{-(1+\beta)\ell_1}\norm{g_1}_X\cdot 2^{-(1+\beta')m} \cdot\eps_1\\
 &\lesssim 2^{(1-\beta)k+(\frac{3}{2}+\beta)k_1}2^{(\frac{1}{2}-\beta)q_2}\cdot 2^{-(2+\beta'+\beta-3\delta)m}\eps_1^2,
\end{aligned} 
\end{equation}
which are acceptable contributions.

\subsubsection*{Subcase 3.2: $2^{k_2}\ll 2^{k_1}$} Then we have $2^{q_1-q_2}\sim 2^{k_2-k_1}\ll 1$, and thus $q\ll q_1\ll q_2$. Using \eqref{ConsLem57Gapq}, we can assume that
\begin{equation}\label{IBPSubcase832}
-q_2+\ell_2\geq (1-\delta)m-100\qquad\hbox{ and }\qquad \max\{-q_1,\ell_1-q_2\}\ge (1-\delta) m-100.
\end{equation}
We also observe that
\begin{equation}\label{eq:sc3.2-philowbd}
\begin{split}
\vert \Phi(\xi,\eta)\vert\ge 2^{q_2-10}.
\end{split}
\end{equation}

\medskip

Assume first that $q_1\ge (1+2\beta)q_2$, so that from \eqref{IBPSubcase832}, we see that $\ell_i\gtrsim (1-\delta)m+q_2$ for $i=1,2$. We can use the precised dispersive decay from Proposition \ref{prop:decay}. The worst case is when $g_2$ has more than $N-3$ vector fields. In this case, we split
\begin{equation*}
\begin{split}
e^{it\Lambda}g_1=e^{it\Lambda}g_1^{I}+e^{it\Lambda}g_1^{II},\qquad\Vert e^{it\Lambda}g_1^{I}\Vert_{L^\infty}\lesssim \varepsilon_1 2^{-\frac{3}{2}m-\frac{q_1}{2}},\qquad \Vert e^{it\Lambda}g_1^{II}\Vert_{L^2}\lesssim \varepsilon_1 2^{-(1+\beta^\prime)m} 
\end{split}
\end{equation*}
and we use a crude estimate to estimate
\begin{equation*}
\begin{split}
\norm{P_{k,p,q}R_\ell \B_{\m}(g_1^{I},g_2)}_{L^2}&\lesssim 2^m\cdot 2^{k+q_2}\cdot\Vert e^{it\Lambda}g_1^I\Vert_{L^\infty}\Vert g_2\Vert_{L^2}\\
&\lesssim 2^{k-\frac{1}{2}m}2^{-\frac{q_1}{2}}2^{q_2-(1+\beta)\ell_2} \cdot\Vert g_2\Vert_{X}\varepsilon_1\\
&\lesssim 2^{-(1+\beta+2\delta)m}\varepsilon_1^2\cdot 2^{-(\frac{1}{2}-4\delta) m}2^{k}2^{-\frac{1}{2}q_1-\beta q_2}
\end{split}
\end{equation*}
and this is enough using \eqref{CrudeBoundq1} since $-q_1\le-(1+2\beta)q_2$ and $q_2\ge-6m/7$. Similarly a crude estimate gives
\begin{equation*}
\begin{split}
\norm{P_{k,p,q}R_\ell \B_{\m}(g_1^{II},g_2)}_{L^2}&\lesssim 2^m\cdot 2^{k+q_2}\cdot\Sz\cdot \Vert e^{it\Lambda}g_1^{II}\Vert_{L^2}\Vert g_2\Vert_{L^2}\\
&\lesssim 2^m\cdot 2^{k+\frac{3}{2}k_2+\frac{3}{2}q_2}\cdot 2^{-(1+\beta^\prime)m}2^{-(1+\beta)\ell_2}\cdot\varepsilon_1\Vert g_2\Vert_X\\
&\lesssim 2^{-(1+\beta^\prime+\beta-3\delta)m}\cdot 2^{(\frac{1}{2}-\beta)q_2}\cdot 2^{\frac{5}{2}k}\cdot\varepsilon_1^2
\end{split}
\end{equation*}
and this is acceptable.

\medskip

We can now assume that $q_1\le (1+2\beta)q_2$, so that $2^{k_2}\lesssim 2^{k_1}2^{2\beta q_2}$. We can do a normal form as in \eqref{eq:mdecompNF}--\eqref{mdecompNFNRNorms} with $\lambda=2^{q_2-20}$, so that $\m=\m^{nr}$ by \eqref{eq:sc3.2-philowbd}. On the one hand, a crude estimate using \eqref{BoundsOnMAddedBenoit} gives that
\begin{equation}
\begin{aligned}
 \norm{P_{k,p,q}R_\ell \Q_{\m \Phi^{-1}}(g_1,g_2)}_{L^2}&\lesssim 2^{k}\cdot 2^{\frac{3}{2}k_2+\frac{q_2}{2}}\norm{g_1}_{L^2}\norm{g_2}_{L^2}\\
 &\lesssim 2^k2^{\frac{3}{2}k_2+\frac{q_2}{2}}\cdot \min\{2^{\frac{q_1}{2}},2^{-(1+\beta)\ell_1}\}2^{-(1+\beta)\ell_2}\cdot \left[\norm{g_1}_B+\Vert g_1\Vert_X\right]\norm{g_2}_X\\
 &\lesssim 2^{-(1+\beta-2\delta)m}\cdot 2^{\frac{3}{2}k_2-(\frac{1}{2}+\beta)q_2}\cdot2^{(1-\beta)\frac{q_1}{2}}2^{-\beta(1+\beta)\ell_1}\cdot 2^{k}\varepsilon_1^2,
\end{aligned} 
\end{equation}
which gives an acceptable contribution. Similarly,
\begin{equation}
\begin{aligned}
 \norm{P_{k,p,q}R_\ell \B_{\m \Phi^{-1}}(\partial_tg_1,g_2)}_{L^2}&\lesssim 2^m\cdot 2^{k}\cdot 2^{\frac{3}{2}k_2+\frac{q_2}{2}}\cdot\norm{\partial_t g_1}_{L^2}\norm{g_2}_{L^2}\\
 &\lesssim 2^{-(\frac{1}{2}-\gamma)m}\cdot 2^{\frac{3}{2}(q_1-q_2)+\frac{q_2}{2}}2^{-(1+\beta)\ell_2}2^{\frac{5}{2}k}\cdot\varepsilon_1^2\Vert g_2\Vert_{X}\\
 &\lesssim 2^{-(\frac{3}{2}-\gamma-2\delta)m}2^{-\frac{1}{2}q_2}2^{\frac{5}{2}k}\cdot\varepsilon_1^3
\end{aligned} 
\end{equation}
which is enough since $q_2\ge-6m/7$ and $\beta\le 1/10$. The other case is simpler:
\begin{equation}
\begin{aligned}
 \norm{P_{k,p,q}R_\ell \B_{\m \Phi^{-1}}(g_1,\partial_tg_2)}_{L^2}&\lesssim 2^m\cdot 2^{k}\cdot 2^{\frac{3}{2}k_2+\frac{q_2}{2}}\cdot\norm{g_1}_{L^2}\norm{\partial_tg_2}_{L^2}\\
 &\lesssim 2^{-(\frac{1}{2}-\gamma)m}\cdot 2^{(q_1-q_2)+\frac{q_1}{2}}\min\{2^{\frac{q_1}{2}},2^{-(1+\beta)\ell_1}\}2^{\frac{5}{2}k}\cdot\varepsilon_1^2\Vert g_1\Vert_{X}\\
\end{aligned} 
\end{equation}
and if $q_1\le-(1-\delta)m$, we obtain an acceptable contribution, while if $\ell_1\ge (1-\delta)m+q_2-300$, we have the same numerology as in the term above. In all cases, we have an acceptable contribution.

\subsubsection*{Case 4: $q_1\ll q,q_2$}
By Lemma \ref{lem:gapp-cases} we have three scenarios to consider:

\subsubsection*{Subcase 4.1: $2^{k}\sim 2^{k_2}$} Then also $2^{q}\sim 2^{q_2}$.

Here repeated integration by parts gives the claim if
\begin{equation}\label{82341}
\begin{aligned}
 V_\eta:\qquad &2^{-q_1}+2^{\ell_1-q_2}\leq 2^{(1-\delta)m},\\
 V_{\xi-\eta}:\qquad &2^{-q_{2}}(2^{k_2-k_1}+2^{\ell_2})\leq 2^{(1-\delta)m}.
\end{aligned}
\end{equation}
This leads to the following cases to be distinguished:

\begin{enumerate}[label=(\alph*),wide]
\item Assume first that
\begin{equation}\label{823Case41}
k_2-k_1-q_2\ge(1-\delta)m-200.
\end{equation}
In this case, we can use \eqref{BoundsOnMAddedBenoit} with crude estimates as in Lemma \ref{lem:set_gain} to get
\begin{equation}
\begin{aligned}
 \norm{P_{k,p,q}R_\ell \B_{\m}(g_1,g_2)}_{L^2}&\lesssim 2^m\cdot 2^{k+q_2}\cdot 2^{\frac{3}{2}k_1+\frac{q_1}{2}}\cdot \Vert g_1\Vert_{L^2}\Vert g_2\Vert_{L^2}\\
 &\lesssim 2^m\cdot 2^{\frac{3}{2}(q_2+k_1-k)}2^{\frac{5}{2}k}2^{\frac{q_1}{2}}\min\{2^{\frac{q_1}{2}},2^{-(1+\beta)\ell_1}\}\cdot\left[\Vert g_1\Vert_B+\Vert g_1\Vert_X\right]\Vert g_2\Vert_B\\
 &\lesssim 2^{-(\frac{1}{2}-3\delta)m}2^{\frac{5}{2}k}2^{\frac{q_1}{2}}\min\{2^{\frac{q_1}{2}},2^{-(1+\beta)\ell_1}\}\cdot\varepsilon_1^2.
\end{aligned} 
\end{equation}  
If $q_1\leq -(1-\delta)m-100$, we can use the first estimate, while if $\ell_1\ge q_2+(1-\delta)m-100$, we can use the second term in the minimum since $q_2\ge -6m/7$ from \eqref{CrudeBoundq1}. In view of \eqref{82341}, this covers all cases when \eqref{823Case41} holds.

\item From \eqref{823Case41}, we can now assume that
\begin{equation}\label{823Case42}
\begin{split}
&\qquad \ell_2-q_2\ge(1-\delta)m-100\qquad\hbox{ and }\\
& \hbox{either }\quad q_2-k_2\ge q_1-k_1+10\quad\hbox{ or }\quad\ell_1-q_2\ge(1-\delta)m-10.
\end{split}
\end{equation}

Assume first that
\begin{equation*}
\begin{split}
q_2-k_2\ge q_1-k_1+10,
\end{split}
\end{equation*}
in this case, we have that, on the support of integration,
\begin{equation}\label{LargeHorGrad841}
\begin{split}
 \vert\nabla_{\eta_h}\Phi(\xi,\eta)\vert\gtrsim 2^{q_2-k_2}.
\end{split}
\end{equation}
We will proceed as in Lemma \ref{lem:nfs}\eqref{it:NF-bd34}, and decompose for $\lambda>0$ to be determined
\begin{equation*}
\begin{split}
\mathfrak{m}(\xi,\eta)&=\mathfrak{m}^{res}(\xi,\eta)+\sum_{r\ge1}\mathfrak{m}_r(\xi,\eta),\\
\mathfrak{m}^{res}(\xi,\eta)&=\psi(\lambda^{-1}\Phi(\xi,\eta))\mathfrak{m}(\xi,\eta),\qquad \mathfrak{m}_r(\xi,\eta)=\varphi(2^{-r}\lambda^{-1}\Phi(\xi,\eta))\mathfrak{m}(\xi,\eta).
\end{split}
\end{equation*}
We can treat the resonant term using \eqref{LargeHorGrad841}, Lemma \ref{lem:set_gain2} and \eqref{BoundsOnMAddedBenoit}:
\begin{equation}\label{83241Res}
\begin{aligned}
\norm{P_{k,p,q}R_\ell \B_{\m^{res}}(g_1,g_2)}_{L^2}
&\lesssim 2^m\cdot 2^{k+q_2}\cdot 2^{\frac{k_1+q_1}{2}}\cdot (2^{k-q_2}\lambda)^{\frac{1}{2}}\cdot \Vert g_1\Vert_{L^2}\Vert g_2\Vert_{L^2}\\
 &\lesssim 2^m\cdot 2^{\frac{3k+k_1}{2}}\cdot \lambda^{\frac{1}{2}}\cdot\min\{2^{q_1},2^{-(1+\beta)\ell_1+\frac{q_1}{2}}\}2^{-(1+\beta)\ell_2+\frac{q_2}{2}}\cdot\left[\Vert g_1\Vert_B+\Vert g_1\Vert_X\right]\cdot\Vert g_2\Vert_X\\
  &\lesssim 2^{-(\beta-3\delta)m}\cdot 2^{\frac{3k+k_1}{2}}2^{-(\frac{1}{2}+\beta)q_2}\cdot \lambda^{\frac{1}{2}}\cdot\min\{2^{q_1},2^{-(1+\beta)\ell_1+\frac{q_1}{2}}\}\cdot\varepsilon_1^2.
\end{aligned} 
\end{equation}
On the other hand, for the nonresonant terms, $r\ge 1$, we use a normal form transformation as in \eqref{mdecompNFNRNorms} and we estimate with a crude estimate, using \eqref{LargeHorGrad841} and Lemma \ref{lem:set_gain2} (see also Lemma \ref{lem:nfs}\eqref{it:NF-bd34}):
\begin{equation}\label{83241NRes1}
\begin{aligned}
 \norm{P_{k,p,q}R_\ell \Q_{\m_r \Phi^{-1}}(g_1,g_2)}_{L^2}&\lesssim 2^{k+q_2}\cdot 2^{\frac{k_1+q_1}{2}}\cdot 2^{-r}\lambda^{-1}(2^{k-q_2}2^{r}\lambda)^{\frac{1}{2}}\cdot \Vert g_1\Vert_{L^2}\Vert g_2\Vert_{L^2}\\
 &\lesssim 2^{\frac{3k+k_1}{2}}\cdot 2^{-\frac{r}{2}}\lambda^{-\frac{1}{2}}\cdot\min\{2^{q_1},2^{-(1+\beta)\ell_1+\frac{q_1}{2}}\}2^{-(1+\beta)\ell_2+\frac{q_2}{2}}\cdot\left[\Vert g_1\Vert_X+\Vert g_1\Vert_B\right]\cdot\Vert g_2\Vert_X\\
  &\lesssim 2^{-(1+\beta-2\delta)m}2^{\frac{3k+k_1}{2}}\cdot 2^{-(\frac{1}{2}+\beta)q_2}2^{-\frac{r}{2}}\lambda^{-\frac{1}{2}}\cdot\min\{2^{q_1},2^{-(1+\beta)\ell_1+\frac{q_1}{2}}\}\cdot\varepsilon_1^2,
\end{aligned} 
\end{equation}
and using Lemma \ref{lem:dtfinL2} as well
\begin{equation}\label{83241NRes2}
\begin{aligned}
 \norm{P_{k,p,q}R_\ell \B_{\m_r \Phi^{-1}}(\partial_tg_1,g_2)}_{L^2}&\lesssim 2^m\cdot 2^{k+q_2}2^{\frac{k_2}{2}} 2^{\frac{k_1+q_1}{2}}\cdot 2^{-r}\lambda^{-1}(2^{k-q_2}2^{r}\lambda)^{\frac{1}{2}}\cdot \Vert \partial_tg_1\Vert_{L^2}\Vert g_2\Vert_{L^2}\\
 &\lesssim 2^m\cdot2^{\frac{4k+k_1}{2}+\frac{q_1+q_2}{2}}\cdot 2^{-\frac{r}{2}}\lambda^{-\frac{1}{2}}\cdot 2^{-(1+\beta)\ell_2}\cdot 2^{-(\frac{3}{2}-\gamma)m}\varepsilon_1^3.
\end{aligned} 
\end{equation}
and
\begin{equation}\label{83241NRes3}
\begin{aligned}
\norm{P_{k,p,q}R_\ell \B_{\m_r \Phi^{-1}}(g_1,\partial_tg_2)}_{L^2}
 &\lesssim 2^m\cdot 2^{k+q_2}\cdot 2^{\frac{k_1+q_1}{2}}\cdot 2^{-r}\lambda^{-1}(2^{k-q_2}2^r\lambda)^{\frac{1}{2}}\cdot \Vert g_1\Vert_{L^2}\Vert \partial_tg_2\Vert_{L^2}\\
 &\lesssim 2^{-(\frac{1}{2}-\gamma)m}\cdot2^{\frac{3k+k_1+q_2}{2}}\cdot 2^{-\frac{r}{2}}\lambda^{-\frac{1}{2}}\cdot\min\{2^{q_1},2^{-(1+\beta)\ell_1+\frac{q_1}{2}}\}\cdot\left[\Vert g_1\Vert_X+\Vert g_1\Vert_B\right]\cdot\varepsilon_1^2.
\end{aligned} 
\end{equation}

Inspecting \eqref{83241Res}, \eqref{83241NRes1}, \eqref{83241NRes2} and \eqref{83241NRes3}, we obtain an acceptable contribution when
\begin{equation*}
q_1\le -(1-3\beta)m-2\beta q_2,\qquad\hbox{ and }\qquad \lambda=2^{(1+6\beta)q_2-6\beta m},
\end{equation*}
since $q_2\ge -6m/7$ from \eqref{CrudeBoundq1} and $\beta\leq 1/100$.

Finally, when
\begin{equation*}
\begin{split}
\ell_i\ge q_2+(1-\delta)m\qquad \hbox{ and }\qquad  q_1\ge -(1-3\beta)m-2\beta q_2
\end{split}
\end{equation*}
we use the precised dispersive decay from Proposition \ref{prop:decay}. The worst case is when $g_2$ has too many vector fields, in which case we decompose
\begin{equation*}
\begin{split}
e^{it\Lambda}g_1&=e^{it\Lambda}g_1^I+e^{it\Lambda}g_1^{II}
\end{split}
\end{equation*}
and we compute that
\begin{equation*}
\begin{aligned}
\norm{P_{k,p,q}R_\ell \B_{\m}(g_1^{I},g_2)}_{L^2} &\lesssim 2^m\cdot 2^{k+q_2}\cdot \Vert e^{it\Lambda}g_1^I\Vert_{L^\infty} \Vert g_2\Vert_{L^2}\\
&\lesssim 2^{-(\frac{3}{2}-2\delta)m}2^{\frac{3}{2}k_1-\frac{1}{2}q_1-\beta q_2}\Vert g_1\Vert_D\Vert g_2\Vert_X \lesssim 2^{-(1+5\beta/4)m}\varepsilon_1^2,
\end{aligned} 
\end{equation*}
and
\begin{equation*}
\begin{aligned}
\norm{P_{k,p,q}R_\ell \B_{\m}(g_1^{II},g_2)}_{L^2} &\lesssim 2^m\cdot 2^{k+q_2}\cdot \Sz\cdot \Vert e^{it\Lambda}g_1^{II}\Vert_{L^2} \Vert g_2\Vert_{L^2}\\
&\lesssim 2^{-(1+\beta^\prime+\beta -2\delta)m}2^{(\frac{1}{2}-\beta) q_2}2^{\frac{5}{2}k}\cdot \Vert g_1\Vert_D\Vert g_2\Vert_X \lesssim 2^{-(1+\beta+10\delta)m}\varepsilon_1^2,
\end{aligned} 
\end{equation*}
which is acceptable.

\end{enumerate}

\subsubsection*{Subcase 4.2: $2^{k}\ll 2^{k_2}\sim 2^{k_1}$} Then also $2^{q_2-q}\sim 2^{k-k_2}\ll 1$, so that $q_1\ll q_2\ll q$. Using Lemma \ref{lem:new_ibp}, repeated integration by parts then gives the claim if
\begin{equation}\label{8342Suff}
\begin{aligned}
(V_\eta)\quad \max\{-q_1,\ell_1-q_2\}&\le(1-\delta)m+k-k_2,\quad\hbox{ or }\quad (V_{\xi-\eta})\quad \ell_2-q_2\le (1-\delta)m+k-k_2.
\end{aligned}
\end{equation}

In addition, we have that
\begin{equation}
 \abs{\Phi}\gtrsim 2^{q_{\max}},
\end{equation}
so that, using Lemma \ref{lem:phasesymb_bd}, we see that
\begin{equation}\label{8342BoundedNF}
\Vert \m \Phi^{-1}\Vert_{\widetilde{\mathcal{W}}}\lesssim 2^k.
\end{equation}
And we can do a normal form as in \eqref{mdecompNFNRNorms}. The most difficult term is the boundary term. First a crude estimate using Lemma \ref{lem:set_gain} gives
\begin{equation*}
\begin{aligned}
 \norm{P_{k,p,q}R_\ell \Q_{\m\Phi^{-1}}(g_1,g_2)}_{L^2}&\lesssim  2^k\Sz\cdot \norm{g_1}_{L^2}\norm{g_2}_{L^2}\\
 &\lesssim 2^{2k+\frac{k_1}{2}+\frac{q_1}{2}}\cdot 2^{\frac{q_1+q_2}{2}}\cdot\Vert g_1\Vert_B\Vert g_2\Vert_{B}\\
 &\lesssim 2^{\frac{5}{2}k_1}2^{q_1+\frac{q_2}{2}}\cdot\varepsilon_1^2,
 \end{aligned} 
\end{equation*}
which is acceptable if $q_2+q_1/2\le -5/4m$. Independently, if
\begin{equation*}
q_1\le -(1-\delta)m-k+k_2,\qquad q_2\ge -5m/6,
\end{equation*}
a crude estimate using Lemma \ref{lem:set_gain} gives
\begin{equation}
\begin{aligned}
 \norm{P_{k,p,q}R_\ell \Q_{\m\Phi^{-1}}(g_1,g_2)}_{L^2}&\lesssim  2^k\Sz\cdot \norm{g_1}_{L^2}\norm{g_2}_{L^2}\\
 &\lesssim 2^{2k+\frac{k_1}{2}+\frac{q_1}{2}}\cdot 2^{\frac{q_1}{2}}2^{-(1+\beta)\ell_2}\cdot\Vert g_1\Vert_B\Vert g_2\Vert_{X}\\
 &\lesssim 2^{-(1+\beta-2\delta)m}2^{q_1-q_2-\beta q_2}\cdot 2^{\frac{5}{2}k_1}2^{(1-\beta)(k-k_1)}\varepsilon_1^2\\
 \end{aligned} 
\end{equation}
and this gives an acceptable contribution. Finally, if
\begin{equation*}
\begin{split}
\ell_i\ge (1-\delta)m+k-k_2+q_2,\,\, i\in\{1,2\}\qquad q_2+\frac{1}{2}q_1\ge-5/4m,
\end{split}
\end{equation*}
we can use the precised dispersion inequality from Proposition \ref{prop:decay}. The most difficult case is when $g_2$ has too many vector fields, in which case we decompose
\begin{equation*}
\begin{split}
g_1=g_1^I+g_1^{II},\qquad\Vert e^{it\Lambda}g_1^I\Vert_{L^\infty}&\lesssim 2^{-\frac{3}{2}m-\frac{q_1}{2}+\frac{3}{2}k_1}\Vert g_1\Vert_D,\qquad \Vert g_1^{II}\Vert_{L^2}\lesssim 2^{-(1+\beta^\prime)m}\Vert g_1\Vert_D 
\end{split}
\end{equation*}
and we compute using \eqref{8342BoundedNF}
\begin{equation*}
\begin{aligned}
 \norm{P_{k,p,q}R_\ell \Q_{\m\Phi^{-1}}(g_1^I,g_2)}_{L^2}&\lesssim  2^k\cdot \norm{e^{it\Lambda}g_1^I}_{L^\infty}\norm{g_2}_{L^2}\\
 &\lesssim 2^{k+\frac{3}{2}k_1}2^{-\frac{3}{2}m-\frac{q_1}{2}-\ell_2}\cdot \Vert g_1\Vert_D\Vert g_2\Vert_X\\
 &\lesssim 2^{\frac{5}{2}k_1}2^{-(\frac{5}{2}-\delta)m-\frac{q_1}{2}-q_2}\cdot\varepsilon_1^2
 \end{aligned} 
\end{equation*}
and, using a crude estimate from Lemma \ref{lem:set_gain},
\begin{equation*}
\begin{aligned}
 \norm{P_{k,p,q}R_\ell \Q_{\m\Phi^{-1}}(g_1^{II},g_2)}_{L^2}&\lesssim  2^k\cdot \Sz\cdot\norm{g_1^{II}}_{L^2}\norm{g_2}_{L^2}\\
 &\lesssim 2^{\frac{5}{2}k+\frac{q}{2}}2^{-(1+\beta^\prime)m-(1+\beta)\ell_2}\cdot \Vert g_1\Vert_D\Vert g_2\Vert_X\\
 &\lesssim 2^{\frac{5}{2}k_1}2^{-(2+\beta+\beta^\prime-2\delta)m-(\frac{1}{2}+\beta)q_2}\cdot\varepsilon_1^2,
 \end{aligned} 
\end{equation*}
which is acceptable since $q_2\ge(2/3)(k_2+k_1/2)\ge-5/6m$. The terms with derivatives are easier to control using Lemma \ref{lem:dtfinL2}.:
\begin{equation*}
\begin{aligned}
 \norm{P_{k,p,q}R_\ell \B_{\Phi^{-1}\m}(\partial_tg_1,g_2)}_{L^2}&\lesssim 2^m\cdot  2^k\Sz\cdot \norm{\partial_tg_1}_{L^2}\norm{g_2}_{L^2}\\
 &\lesssim 2^{-(\frac{1}{2}-\gamma)m}2^{2k+\frac{k_1}{2}}\min\{2^{q_2},2^{-(1+\beta)\ell_2+\frac{q_2}{2}}\}\varepsilon_1^2\left[\Vert g_2\Vert_B+\Vert g_2\Vert_X\right].
\end{aligned} 
\end{equation*}
If $q_2\le -3/4m$, the first term in the $\min$ gives an acceptable contribution, else the second term gives an acceptable contribution using \eqref{8342Suff}. Similarly,
\begin{equation*}
\begin{aligned}
 \norm{P_{k,p,q}R_\ell \B_{\Phi^{-1}\m}(g_1,\partial_tg_2)}_{L^2}&\lesssim 2^m\cdot  2^k\Sz\cdot \norm{g_1}_{L^2}\norm{\partial_tg_2}_{L^2}\\
 &\lesssim 2^{-(\frac{1}{2}-\gamma)m}2^{2k+\frac{k_1}{2}}\min\{2^{q_1},2^{-(1+\beta)\ell_1+\frac{q_2}{2}}\}\varepsilon_1^2\left[\Vert g_1\Vert_B+\Vert g_1\Vert_X\right]
\end{aligned} 
\end{equation*}
and we can conclude similarly.

\subsubsection*{Subcase 4.3: $2^{k_2}\ll 2^{k}\sim 2^{k_1}$} Then also $2^{q-q_2}\sim 2^{k_2-k}\ll 1$, so that $q_1\ll q\ll q_2$. Using Lemma \ref{lem:new_ibp}, repeated integration by parts gives the claim if
\begin{equation}
\begin{aligned}
 (V_\eta)\quad\max\{-q_1,\ell_1-q_2\}\le (1-\delta)m,\qquad\hbox{ or }\qquad
( V_{\xi-\eta})\quad\ell_2-q_2\leq (1-\delta)m,
\end{aligned}
\end{equation}
and we can proceed as for Subcase 4.2, since once again
\begin{equation*}
\vert\Phi(\xi,\eta)\vert\gtrsim 2^{q_{\max}}
\end{equation*}
so that \eqref{8342BoundedNF} holds and since we do not need to keep track of the $k$ contributions.

\medskip
\subsubsection{No gaps}\label{ssec:X-nogaps}
It remains (see \eqref{CrudeBoundq1}) to consider the case  $p_{\min}\ge -10$ and $-6m/7\le q_{\max}\le q_{\min}+10$. We use the dichotomy of Proposition \ref{prop:phasevssigma}. We decompose $\m=\m^{res}+\m^{nr}$ as in \eqref{eq:mdecompNF} with $\lambda=2^{-100}2^{q}$.

\subsubsection*{The nonresonant case $\mathfrak{m}^{nr}$}
On the support of the nonresonant set, we use a normal form transformation as in \eqref{mdecompNFNRNorms}. Lemma \ref{lem:phasesymb_bd} gives
\begin{equation}\label{GoodBoundMultiplierNoGapQ}
 \abs{\m^{nr}\Phi^{-1}}\lesssim \norm{\m^{nr}\Phi^{-1}}_{\W}\lesssim 2^k.
\end{equation}
For the boundary term, we may assume that $g_1$ has fewer vector fields than $g_2$ and we use the precised dispersion estimate from Proposition \ref{prop:decay} to decompose
\begin{equation}\label{PrecisedDec}
\begin{split}
g_1=g_1^I+g_1^{II},\qquad\Vert e^{it\Lambda}g_1^I\Vert_{L^\infty}&\lesssim 2^{-\frac{3}{2}m}2^{-\frac{q}{2}}\Vert g_1\Vert_D,\qquad \Vert g_1^{II}\Vert_{L^2}\lesssim 2^{-(1+\beta^\prime)m}\Vert g_1\Vert_D,
\end{split}
\end{equation}
and using \eqref{GoodBoundMultiplierNoGapQ} we compute that
\begin{equation*}
\begin{aligned}
 \norm{P_{k,p,q}\Q_{\m^{nr}\Phi^{-1}}(g_1^I,g_2)}_{L^2}&\lesssim 2^k\cdot \Vert e^{it\Lambda}g_1^I\Vert_{L^\infty}\Vert g_2\Vert_{L^2}\\
 &\lesssim 2^k\cdot 2^{-\frac{3}{2}m}\Vert g_1\Vert_{D}\Vert g_2\Vert_{B}\lesssim 2^k\cdot 2^{-\frac{3}{2}m}\varepsilon_1^2,
\end{aligned} 
\end{equation*}
while for the other term, we use Corollary \ref{cor:extrapol_decay} as well to get
\begin{equation*}
\begin{aligned}
 \norm{P_{k,p,q}\Q_{\m^{nr}\Phi^{-1}}(g_1^{II},g_2)}_{L^2}&\lesssim 2^k\cdot \Vert g_1^{II}\Vert_{L^2}\Vert e^{it\Lambda}g_2\Vert_{L^\infty}\\
 &\lesssim 2^k\cdot 2^{-(1+\frac{2}{3}+\beta^\prime)m}\Vert g_1\Vert_{D}\,\varepsilon_1.
\end{aligned} 
\end{equation*}
For the terms with the time derivatives, we proceed similarly, using \eqref{GoodBoundMultiplierNoGapQ}, Lemma \ref{lem:dtfinL2} and Corollary \ref{cor:extrapol_decay}:
\begin{equation*}
\begin{aligned}
 \norm{P_{k,p,q}\B_{\m^{nr}\Phi^{-1}}(g_1,\partial_tg_2)}_{L^2}&\lesssim 2^m2^k\cdot \Vert e^{it\Lambda}g_1\Vert_{L^\infty}\Vert \partial_tg_2\Vert_{L^2}\\
 &\lesssim 2^k\cdot 2^{-(\frac{1}{2}+\frac{2}{3}-\gamma)m}\varepsilon_1^3
\end{aligned} 
\end{equation*}
and similarly for the symmetric term.

\subsubsection*{The resonant term $\mathfrak{m}^{res}$}

A crude estimate using Lemma \ref{lem:set_gain} gives
\begin{equation*}
\begin{aligned}
 \norm{P_{k,p,q}\B_{\m^{res}}(g_1,g_2)}_{L^2}&\lesssim 2^m2^{q+k}\cdot\Sz\cdot \Vert g_1\Vert_{L^2}\Vert g_2\Vert_{L^2}\\
 &\lesssim 2^m2^{k+\frac{3}{2}k_{\min}+\frac{3}{2}q}\cdot\min\{2^{k_1},2^{\frac{q}{2}}\}\cdot\min\{2^{k_2},2^{\frac{q}{2}}\}\cdot\left[\Vert g_1\Vert_{H^{-1}}+\Vert g_1\Vert_B\right]\cdot\left[\Vert g_2\Vert_{H^{-1}}+\Vert g_2\Vert_B\right]\\
 &\lesssim 2^m2^{k+\frac{3}{2}k_{\min}+2q}\cdot\min\{2^\frac{q}{2},2^{k_1},2^{k_2}\}\varepsilon_1^2.
\end{aligned} 
\end{equation*}
This gives an acceptable contribution when
\begin{equation}\label{824ResCase1}
k_{\min}+q\le -(1-\beta)m.
\end{equation}

We see from Proposition \ref{prop:phasevssigma} that $\abs{\bar\sigma}\gtrsim 2^{q+k_{\max}+k_{\min}}$ and we can proceed as above in the case of gaps (without need to worry about the losses in $p$'s and $q$'s). Observing that
\begin{equation*}
 (sV_\eta\Phi)^{-1}\cdot V_\eta(\psi(2^{-q}\Phi))=(2^{100}s2^{q})^{-1}\cdot\psi^\prime(\lambda^{-1}\Phi)
\end{equation*}
we can use Lemma \ref{lem:new_ibp} to control the terms when
\begin{equation*}
\begin{split}
\max\{2k_i,k_1+k_2+\ell_i\}\le (1-\delta)m+q+k_{\max}+k_{\min},\qquad i\in\{1,2\}.
\end{split}
\end{equation*}
Using the conclusion from \eqref{824ResCase1}, it suffices to consider the case
\begin{equation*}
\begin{split}
k_1+k_2+\ell_i\ge (1-\delta)m+q+k_{\max}+k_{\min},\qquad i\in\{1,2\},
\end{split}
\end{equation*}
(else $k_{\min}+q\lesssim -(1-\beta)m$).

To conclude, we want to use the precised dispersion from Proposition \ref{prop:decay}. Assuming that $g_1$ has fewer vector fields, we decompose as in \eqref{PrecisedDec}, we compute, using Lemma \ref{lem:ECmult_bds},
\begin{equation*}
\begin{aligned}
 \norm{P_{k,p,q}\B_{\m^{res}}(g_1^I,g_2)}_{L^2}&\lesssim 2^m2^{q+k}\cdot \Vert e^{it\Lambda}g_1\Vert_{L^\infty}\Vert g_2\Vert_{L^2}\\
 &\lesssim 2^{-\frac{1}{2}m}2^{\frac{q}{2}+k+\frac{3}{2}k_1-\ell_2}\Vert g_1\Vert_D\Vert g_2\Vert_X\\
 &\lesssim 2^{-(\frac{3}{2}-\delta)m}2^{-\frac{q}{2}+k+\frac{5}{2}k_1-k_{\max}-k_{\min}+k_2}\varepsilon_1^2
\end{aligned} 
\end{equation*}
and this gives an acceptable contribution since $q\ge -6m/7$. Similarly, using Lemma \ref{lem:set_gain}
\begin{equation*}
\begin{aligned}
 \norm{P_{k,p,q}\B_{\m^{res}}(g_1^{II},g_2)}_{L^2}&\lesssim 2^m2^{q+k}\cdot\Sz\cdot \Vert g_1\Vert_{L^2}\Vert g_2\Vert_{L^2}\\
 &\lesssim 2^{-\beta^\prime m}2^{\frac{3}{2}q+k+\frac{3}{2}k_{\min}-(1+\beta)\ell_2}\Vert g_1\Vert_D\Vert g_2\Vert_X\\
 &\lesssim 2^{-(1+\beta+\beta^\prime-2\delta)m}2^{\frac{5}{2}k_{\max}^+}\varepsilon_1^2
\end{aligned} 
\end{equation*}
which gives an acceptable contribution.

\end{proof}

\subsection*{Acknowledgments}
Y.\ Guo's research is supported in part by NSF grant DMS-2106650. B.\ Pausader is supported in part by NSF grant DMS-2154162, a Simons fellowship and benefited from the hospitality of CY-Advanced Studies fellow program. K.\ Widmayer gratefully acknowledges support of the SNSF through grant PCEFP2\_203059. \blue{We thank the referees for their careful reading and the many valuable suggestions.}

\newpage

\appendix
\section{Auxiliary results}\label{apdx}
\subsection{Proof of Proposition \ref{prop:LPOmega}}\label{apdx:angLP}
Here we give the proof of the properties of the angular Littlewood-Paley decomposition introduced in Section \ref{ssec:angLP}. Denoting by $P_n^{(a,b)}$ the Jacobi polynomials, we begin by recalling that with 
\begin{equation}
 P^{(0,0)}_n(z)=L_n(z)=\frac{1}{2^n n!}\frac{d^n}{dz^n}[(z^2-1)^n]
\end{equation}
there holds that (see \cite[Section 4.5]{Sze1975})
\begin{equation}\label{DerJacobi}
\begin{split}
 \frac{d}{dz}P^{(a,a)}_n=\frac{n+2a+1}{2}P^{(a+1,a+1)}_{n-1},\qquad a\in\N.
\end{split}
\end{equation}
Moreover, we have the following asymptotics (see \cite[Section 8.21]{Sze1975}):
\begin{lemma}\label{LemZonal0}
Fix $0<c<\pi$, then 
\begin{equation}\label{SogForm}
P^{(a,a)}_n(\cos\theta)=\begin{cases}2^a\sqrt{\frac{2}{\pi}}\frac{1}{\sqrt{n}(\sin\theta)^{a+\frac{1}{2}}}\left(\cos((n+a+\frac{1}{2})\theta-\frac{(2a+1)\pi}{4})+\frac{1}{n\sin\theta}O(1)\right)&\quad\hbox{if}\quad c/n\le\theta\le\pi-c/n,\\
O(n^a)&\quad\hbox{else.}
\end{cases}
\end{equation}
\end{lemma}

These estimates are relevant in view of the following fact about angular regularity (see \cite[Section 2.8.4]{Atkinson2012}):
\begin{lemma}\label{LemZonal1}
Let $\Pi_n$ denote the $L^2$-projector onto the $n$-th eigenspace of the spherical Laplacian $\Delta_{\mathbb{S}^2}$ associated to the eigenvalue $n(n+1)$. Then for any $P\in\mathbb{S}^2$ and $f\in L^2(\mathbb{S}^2)$ there holds that
\begin{equation}\label{ZonConv}
(\Pi_n f)(P)=\int_{\mathbb{S}^2}f(Q)\mathfrak{Z}_n(\langle P,Q\rangle)d\nu_{\mathbb{S}^2}(Q),\qquad \mathfrak{Z}_n:=\frac{2n+1}{4\pi}P_n^{(0,0)}.
\end{equation}

\end{lemma}

We are now in the position to give the proof of Proposition \ref{prop:LPOmega}.

\begin{proof}[Proof of Proposition \ref{prop:LPOmega}]

We start with a proof of \ref{it:LPang-comm}. It is straightforward to see that for any $\ell\in\Z$ there holds that $[\bar{R}_\ell,S]=0$. The commutation with $\Omega_{ab}$ follows from the identity
\begin{equation*}
\begin{split}
 (\Omega^x_{ab}+\Omega^\vartheta_{ab})\langle x,\vartheta\rangle=0.
\end{split}
\end{equation*}
It thus suffices to prove that $\bar{R}_\ell$ commutes with the Fourier transform. After writing down the explicit formula, we see that it suffices to check that
\begin{equation*}
\begin{split}
\int_{\mathbb{S}^2}e^{i\vert x\vert\vert\xi\vert\langle \alpha,\frac{\xi}{\vert\xi\vert}\rangle}\mathfrak{Z}_n(\langle \alpha,\beta\rangle)d\nu_{\mathbb{S}^2}(\alpha)=\int_{\mathbb{S}^2}e^{i\vert x\vert\vert\xi\vert\langle \alpha,\beta\rangle}\mathfrak{Z}_n(\langle \alpha,\frac{\xi}{\vert\xi\vert}\rangle)d\nu_{\mathbb{S}^2}(\alpha).
\end{split}
\end{equation*}
Let $\rho_\gamma$ denote a rotation that sends $N$ to $\gamma\in\mathbb{S}^2$. After a change of variable, we observe that
\begin{equation}\label{FRk}
\begin{split}
\int_{\mathbb{S}^2}e^{i\lambda \langle \alpha,\gamma\rangle}\mathfrak{Z}_n(\langle \alpha,\beta\rangle)d\nu_{\mathbb{S}^2}(\alpha)&=\int_{\mathbb{S}^2}e^{i\lambda\langle N,\alpha\rangle}\mathfrak{Z}_n(\langle \beta,\rho_\gamma \alpha\rangle)d\nu_{\mathbb{S}^2}(\alpha)=\Pi_n[e^{i\lambda\langle N,\cdot\rangle}](\rho_{-\gamma}\beta).
\end{split}
\end{equation}
It remains to observe that $e^{i\lambda\langle N,\cdot\rangle}$ is a zonal function, and that $\Pi_n$ respects zonal functions: this follows by direct inspection, or by the fact that $\Delta_{\mathbb{S}^2}$ and the $z$-angular momentum $\Omega_{12}$ commute. Hence $\Pi_n[e^{i\lambda\langle N,\cdot\rangle}]$ only depends on the distance to the north pole. But
\begin{equation*}
\langle \rho_{\gamma}^{-1}\beta,N\rangle=\langle\beta,\gamma\rangle=\langle \rho_{\beta}^{-1}\gamma,N\rangle
\end{equation*}
and therefore the last term in \eqref{FRk} is symmetric in $\gamma,\beta$.

\medskip

The first and third affirmation in \ref{it:LPang-orth} follow from the same properties on $L^2(\mathbb{S}^2)$. For the second statement, using the reproducing property \eqref{ZonConv}, we compute that
\begin{equation*}
\begin{split}
\left(\bar{R}_\ell \bar{R}_{\ell^\prime}f\right)(x)&=\sum_{n,n'\ge0}\varphi(2^{-\ell}n)\varphi(2^{-\ell^\prime}n')\int_{\mathbb{S}^2}f(\vert x\vert\vartheta)\left(\int_{\mathbb{S}^2}\mathfrak{Z}_{n'}(\langle\frac{x}{\vert x\vert},\alpha\rangle)\mathfrak{Z}_n(\langle\alpha,\vartheta\rangle)d\nu_{\mathbb{S}^2}(\alpha)\right) d\nu_{\mathbb{S}^2}(\vartheta)\\
&=\sum_{n'\ge0}\varphi(2^{-\ell}n')\varphi(2^{-\ell^\prime}n')\int_{\mathbb{S}^2}f(\vert x\vert\vartheta)\mathfrak{Z}_{n'}(\langle\frac{x}{\vert x\vert},\vartheta\rangle)d\nu_{\mathbb{S}^2}(\vartheta)
\end{split}
\end{equation*}
This shows that $\bar{R}_\ell \bar{R}_{\ell^\prime}=0$ whenever $\abs{\ell-\ell^\prime}\ge4$. The last statement in \ref{it:LPang-orth} follows by duality from the fact that
\begin{equation*}
\int_{\mathbb{S}^2}\Pi_{n_1}[f]\cdot \Pi_{n_2}[g]\cdot\Pi_{n_3}[h]d\nu_{\mathbb{S}^2}=0
\end{equation*}
whenever $\max\{n_1,n_2,n_3\}\ge\frac{1}{2}\hbox{med}\{n_1,n_2,n_3\}+4$. (This can also be seen from the fact that spherical harmonics of degree $n$ are restrictions to $\mathbb{S}^2$ of homogeneous harmonic polynomials of degree $n$.)

\medskip

For \ref{it:LPang-bd} we let
\begin{equation*}
K_\ell(\omega,\vartheta)=\sum_{n\ge 0}\varphi(2^{-\ell}n)\mathfrak{Z}_n(\langle \omega,\vartheta\rangle),
\end{equation*}
and we claim that
\begin{equation*}
\sup_\omega\Vert K_\ell(\omega,\vartheta)\Vert_{L^1(\mathbb{S}^2_\vartheta)}+\sup_\vartheta\Vert K_\ell(\omega,\vartheta)\Vert_{L^1(\mathbb{S}^2_\omega)}\lesssim 1.
\end{equation*}

This essentially follows from \eqref{SogForm}: With \eqref{DerJacobi} we have that
\begin{equation}
 (2n+1)P_n^{(0,0)}(x)=\frac{d}{dx}\left(P_{n+1}^{(0,0)}(x)-P_{n-1}^{(0,0)}(x)\right)=\frac{1}{2}\left((n+2)P_{n}^{(1,1)}(x)-nP_{n-2}^{(1,1)}(x)\right),
\end{equation}
and thus
\begin{equation*}
\begin{split}
 \sum_{n\geq 0}\varphi(2^{-\ell}n)\frac{2n+1}{4\pi}P_n^{(0,0)}(x)=\frac{1}{8\pi}\sum_{n\geq 0}P_n^{(1,1)}(x)\cdot D_n,\qquad D_n:=(n+2)\left[\varphi(2^{-\ell}n)-\varphi(2^{-\ell}(n+2))\right].
\end{split}
\end{equation*}
In view of \eqref{SogForm}, let
\begin{equation*}
\begin{split}
C_n(\theta)&=\sum_{0\le j\le n-1}\cos\left((j+\frac{3}{2})\theta-\frac{3\pi}{4}\right)=\Re\left(e^{-i\frac{3\pi}{4}}e^{i\frac{\theta}{2}}\frac{1-e^{in\theta}}{1-e^{i\theta}}\right)=\frac{\sin\frac{n\theta}{2}}{\sin\frac{\theta}{2}}\cos\left(\left(\frac{n}{2}+1\right)\theta-\frac{3\pi}{4}\right),\\
\end{split}
\end{equation*}
then
\begin{equation*}
\begin{split}
I_\ell(\theta):=\sin(\theta)^{-3/2}\sum_{n\ge0}\frac{1}{\sqrt{n}}\cos((n+\frac{3}{2})\theta-\frac{3\pi}{4})\cdot D_n&=\sin(\theta)^{-3/2}\cdot \sum_{n\ge0}\frac{D_n}{\sqrt{n}}\cdot\left[C_{n+1}(\theta)-C_n(\theta)\right]\\
&=\sin(\theta)^{-3/2}\cdot\sum_{n\ge0}C_n(\theta)\left[\frac{D_{n-1}}{\sqrt{n-1}}-\frac{D_n}{\sqrt{n}}\right]
\end{split}
\end{equation*}
so that since $\abs{\frac{D_{n+1}}{\sqrt{n+1}}-\frac{D_{n}}{\sqrt{n}}}\lesssim 2^{-3\ell/2}$ and $\abs{C_n(\theta)}\lesssim \abs{\sin(\frac{\theta}{2})}^{-1}$ we have
\begin{equation*}
\begin{split}
\int_{c2^{-\ell}\le \theta\le\pi-c2^{-\ell}}\vert I_\ell(\theta)\vert\cdot\sin(\theta) d\theta\lesssim 1.
\end{split}
\end{equation*}
Similarly,
\begin{equation*}
\begin{split}
I\!I_\ell(\theta):=\sin(\theta)^{-5/2}\sum_{n\ge0}\frac{1}{n^{3/2}},\qquad \int_{c2^{-\ell}\le \theta\le\pi-c2^{-\ell}}\vert I\!I_\ell(\theta)\vert\cdot\sin(\theta) d\theta\lesssim 1,
\end{split}
\end{equation*}
and again by \eqref{SogForm}
\begin{equation*}
 \sum_{n\geq 0}\varphi(2^{-\ell}n)\int_{0\le \theta\le c2^{-\ell}}\abs{P_n^{(1,1)}(\cos(\theta))}\cdot\sin(\theta) d\theta\lesssim 1.
\end{equation*}
This shows that the kernel of $\bar{R}_\ell$ is integrable. A similar proof works for $\bar{R}_{\le \ell}$.

\medskip

We now turn to \ref{it:LPang-Bern}. Starting from
\begin{equation*}
\Delta_{\mathbb{S}^2}=\sum_{j<l}\Omega_{jl}\Omega_{jl},
\end{equation*}
we obtain the self-reproducing formula
\begin{equation*}
\begin{split}
\mathcal{Z}_n=\frac{1}{n(n+1)}\sum_{a<b}\Omega_{ab}\Omega_{ab}\mathcal{Z}_n,\qquad \mathcal{Z}_n(P)=\mathfrak{Z}_n(\langle P,N\rangle)
\end{split}
\end{equation*}
and therefore
\begin{equation}\label{BernsteinConsequence}
\begin{split}
 \bar{R}_\ell f&=2^{-2\ell}\sum_{a<b}\Omega_{ab}\Omega_{ab}\widetilde{R}_\ell f,\\
 \widetilde{R}_\ell f&:= \sum_{n\ge 0}\varphi(2^{-\ell}n)\frac{2^{2\ell}}{n(n+1)}\int_{\mathbb{S}^2}f(\vert x\vert\vartheta)\mathfrak{Z}_n(\langle \vartheta,\frac{x}{\vert x\vert}\rangle)d\nu_{\mathbb{S}^2}(\vartheta)
\end{split}
\end{equation}
where $\widetilde{R}_\ell$ obeys similar properties as $\bar{R}_\ell$. It suffices now to show that
\begin{equation}
\begin{split}
\Vert \Omega_{ab}\bar{R}_\ell f\Vert_{L^r}\lesssim 2^\ell\Vert f\Vert_{L^r},\qquad 1\le r\le\infty,
\end{split}
\end{equation}
and similarly for $\widetilde{R}_\ell$. We provide the details for $\bar{R}_\ell$, $\widetilde{R}_\ell$ is similar. Once again we consider the kernel of $\Omega_{ab} \bar{R}_\ell$:
\begin{equation*}
\begin{split}
(\Omega_{ab} \bar{R}_\ell f)(x)&=2^\ell\int_{\mathbb{S}^2}f(\vert x\vert\vartheta)\cdot \mathcal{K}_\ell(\frac{x}{\vert x\vert},\vartheta)d\nu_{\mathbb{S}^2}(\vartheta),\qquad
\mathcal{K}_\ell(\frac{x}{\vert x\vert},\vartheta):=2^{-\ell}\sum_{n\ge0}\varphi(2^{-\ell}n)\Omega_{ab}\left[\mathfrak{Z}_n(\langle \frac{x}{\vert x\vert},\vartheta\rangle)\right]
\end{split}
\end{equation*}
and claim that
\begin{equation*}
\sup_\omega\Vert \mathcal{K}_\ell(\omega,\vartheta)\Vert_{L^1(\mathbb{S}^2_\vartheta)}+\sup_\vartheta\Vert \mathcal{K}_\ell(\omega,\vartheta)\Vert_{L^1(\mathbb{S}^2_\omega)}\lesssim 1.
\end{equation*}
Indeed, we compute that
\begin{equation*}
\begin{split}
\mathcal{K}_\ell(\omega,\vartheta)&=2^{-\ell}\sum_{n\ge0}\varphi(2^{-\ell}n)\mathfrak{Z}_n^\prime(\langle \omega,\vartheta\rangle)\cdot\Omega^\omega_{ab}(\langle\omega,\vartheta\rangle),\qquad
\vert \Omega^\omega_{ab}(\langle\omega,\vartheta\rangle)\vert \lesssim \sqrt{1-\langle \omega,\vartheta\rangle^2},
\end{split}
\end{equation*}
and the rest follows in a similar way from the boundedness of $\bar{R}_\ell$ by using \eqref{DerJacobi} and \eqref{SogForm}.
\end{proof}


\subsection{Set size gain}\label{apdx:set_gain} 
The idea here is that in the bilinear estimates we can always gain the smallest of \emph{both $p,p_j$ and $q,q_j$}, since they correspond to different directions.
\begin{lemma}\label{lem:set_gain}
Consider a typical bilinear expression $\Q_\mathfrak{m}$ with localizations and a multiplier $\mathfrak{m}$, i.e.\
\begin{equation}
\begin{aligned}
 &\widehat{\Q_\mathfrak{m}(f,g)}(\xi):=\int_\eta e^{\pm is\Phi}\chi(\xi,\eta) \mathfrak{m}(\xi,\eta)\hat{f}(\xi-\eta)\hat{g}(\eta) d\eta,\\
 &\chi(\xi,\eta)=\varphi_{k,p,q}(\xi)\varphi_{k_1,p_1,q_1}(\xi-\eta)\varphi_{k_2,p_2,q_2}(\eta).
\end{aligned}
\end{equation}
Then with
\begin{equation}
 \Sz:=\min\{2^{p+k},2^{p_1+k_1},2^{p_2+k_2}\}\cdot\min\{2^{\frac{q+k}{2}},2^{\frac{q_1+k_1}{2}},2^{\frac{q_2+k_2}{2}}\}
\end{equation} 
we have that
\begin{equation}
 \norm{\Q_\mathfrak{m}(f,g)}_{L^2}\lesssim \Sz\cdot\norm{\mathfrak{m}}_{L^\infty_{\xi,\eta}} \norm{P_{k_1,p_1,q_1}f}_{L^2}\norm{P_{k_2,p_2,q_2}g}_{L^2}.
\end{equation}

\end{lemma}
\begin{proof}
 To begin, let us assume that $p+k<p_1+k_1$ and $q+k>q_1+k_1$ (the ``symmetric cases'' of $p+k<p_1+k_1$ with $q+k<q_1+k_1$ and reverse are direct).
 Then, for any $h\in L^2$ we find that
 \begin{equation*}
 \begin{split}
  \abs{\ip{\Q_\m(f,g),h}}&\lesssim \iint_{\mathbb{R}^3}\abs{\mathfrak{m}(\xi,\eta)}\varphi(2^{-k-p}\xi_\h)\vert\widehat{f}(\xi-\eta)\vert\varphi(2^{-k_1-q_1}(\xi_3-\eta_3))\vert\widehat{g}(\eta)h(\xi)\vert d\xi d\eta\\
  &\lesssim \Vert \mathfrak{m}\Vert_{L^\infty}\Vert \widehat{h}(\xi)\widehat{f}(\xi-\eta)\Vert_{L^2_{\xi,\eta}}\Vert \varphi(2^{-k-p}\xi_\h)\varphi(2^{-k_1-q_1}(\xi_3-\eta_3))\widehat{g}(\eta)\Vert_{L^2_{\xi,\eta}}\\
  &\lesssim 2^{k+\frac{k_1}{2}}2^{p+\frac{q_1}{2}}\Vert \mathfrak{m}\Vert_{L^\infty}\Vert f\Vert_{L^2}\Vert g\Vert_{L^2}\Vert h\Vert_{L^2}.
 \end{split}
 \end{equation*}
 The claim then follows upon changing variables $\eta\leftrightarrow\xi-\eta$.
\end{proof}

\begin{lemma}\label{lem:set_gain2}
With notation as in Lemma \ref{lem:set_gain}, consider
\begin{equation}
\begin{aligned}
 &\widehat{\bar{\Q}_\mathfrak{m}(f,g)}(\xi):=\int_\eta e^{\pm is\Phi}\chi(\xi,\eta) \varphi(\lambda^{-1}\Phi)\mathfrak{m}(\xi,\eta)\hat{f}(\xi-\eta)\hat{g}(\eta) d\eta,\qquad \lambda>0.
\end{aligned}
\end{equation}
\begin{enumerate}
 \item\label{it:vertgain} Assume that on the support of $\chi$ we have $\abs{\partial_{\eta_3}\Phi}\gtrsim L>0$. Then we have that
\begin{equation}
 \norm{\bar{\Q}_\mathfrak{m}(f,g)}_{L^2}\lesssim \min\{2^{k_1+p_1},2^{k_2+p_2}\}\cdot (\lambda L^{-1})^{\frac{1}{2}}\cdot\norm{\mathfrak{m}}_{L^\infty_{\xi,\eta}} \norm{P_{k_1,p_1}f}_{L^2}\norm{P_{k_2,p_2}g}_{L^2}.
\end{equation}

\item\label{it:horgain} Assume that on the support of $\chi$ we have $\abs{\nabla_{\eta_\h}\Phi}\gtrsim L>0$. Then we have that
\begin{equation}
 \norm{\bar{\Q}_\mathfrak{m}(f,g)}_{L^2}\lesssim \min\{2^{k_1+q_1},2^{k_2+q_2}\}^{\frac{1}{2}}\cdot 2^{\frac{k_{2}+p_{2}}{2}}\cdot (\lambda L^{-1})^{\frac{1}{2}}\cdot\norm{\mathfrak{m}}_{L^\infty_{\xi,\eta}} \norm{P_{k_1,p_1,q_1}f}_{L^2}\norm{P_{k_2,p_2,q_2}g}_{L^2}.
\end{equation}
(Analogous statements hold if $\abs{\partial_{\xi_3}\Phi}\gtrsim L>0$ resp.\ $\abs{\nabla_{\xi_\h}\Phi}\gtrsim L>0$.)
\end{enumerate}

\end{lemma}

\begin{proof}
It suffices to prove \eqref{it:vertgain}, part \eqref{it:horgain} is similar. Assume without loss of generality that $2^{k_2+p_2}\lesssim 2^{k_1+p_1}$ (else exchange the roles of $\hat{f}$ and $h$ below).  We have for any $h\in L^2$ that
 \begin{equation}
 \begin{aligned}
  \abs{\ip{\bar{\Q}_\m(f,g),h}}&\lesssim \iint_{\mathbb{R}^3}\abs{\mathfrak{m}(\xi,\eta)}\chi(\xi,\eta) \varphi(\lambda^{-1}\Phi)\vert\widehat{f}(\xi-\eta)\widehat{g}(\eta)\vert \vert h(\xi)\vert d\xi d\eta\\
  &\lesssim \norm{\m}_{L^\infty_{\xi,\eta}}\cdot \norm{\widehat{f}(\xi-\eta)\widehat{g}(\eta)}_{L^2_{\xi,\eta}}\cdot\norm{\chi(\xi,\eta) \varphi(\lambda^{-1}\Phi)h(\xi)}_{L^2_{\xi,\eta}},
 \end{aligned}
 \end{equation}
and the claim follows since
\begin{equation}
 \norm{\chi(\xi,\eta) \varphi(\lambda^{-1}\Phi)h(\xi)}_{L^2_{\xi,\eta}}\lesssim \lambda^{\frac{1}{2}}\cdot L^{-\frac{1}{2}} 2^{p_2+k_2}\norm{h}_{L^2},
\end{equation}
where we have used that 
\begin{equation}
 \sup_\xi \int \chi(\xi,\eta) \varphi(\lambda^{-1}\Phi)d\eta\lesssim \lambda L^{-1}\cdot 2^{2p_2+2k_2}
\end{equation}
by changing variables (for fixed $\xi$) $\eta\mapsto\zeta:=(\eta_1,\eta_2,\Phi(\xi,\eta))$ with Jacobian $\abs{\textnormal{det}\frac{\partial\eta}{\partial\zeta}}=\abs{\partial_{\eta_3}\Phi}^{-1}\lesssim L^{-1}$.
\end{proof}

\subsection{Control of Fourier transform in $L^\infty$}
We record here that our decay norm $D$ in \eqref{eq:defDnorm} also controls the Fourier transform in $L^\infty$:
\begin{lemma}\label{lem:ControlLinfty}
Assume that $f$ is axisymmetric. Then there holds that
\begin{equation*}
\begin{split}
\Vert \widehat{P_{k,p,q}f}\Vert_{L^\infty}
&\lesssim 2^{-\frac{3k}{2}}\left[\norm{P_k f}_B+\norm{SP_kf}_B\right]+2^{-\frac{3k}{2}}\left[\norm{P_k f}_X+\norm{SP_k f}_X\right]\lesssim 2^{-\frac{3k}{2}}\norm{f}_D.
\end{split}
\end{equation*}
\end{lemma}

\begin{proof}
We recall the notation
\begin{equation}
 \varphi_{k,p,q}(\xi)=\varphi(2^{-k}\abs{\xi})\varphi(2^{-p}\sqrt{1-\Lambda^2(\xi)})\varphi(2^{-q}\Lambda(\xi)),
\end{equation}
and assume that $f=P_{k,p,q}f$. 

Switching to spherical coordinates $(\rho,\theta,\phi)\in\R_+\times[0,2\pi]\times [0,\pi]$  and using that $f$ is axisymmetric (and thus independent of $\theta$), we have that for any $(\rho_0,\phi_0)$ on the support of $\varphi_{k,p,q}\widehat{f}$ there holds
\begin{equation*}
\begin{split}
 \varphi_{k,p,q}\widehat{f}(\rho,\phi)&=\varphi_{k,p,q}\widehat{f}(\rho_0,\phi_0)+\int_{\rho_0}^\rho\partial_\rho(\varphi_{k,p,q}\widehat{f})(s,\phi_0)ds+\int_{\phi_0}^\phi\partial_\phi(\varphi_{k,p,q}\widehat{f})(\rho_0,\alpha)d\alpha\\
 &\qquad+\int_{\phi_0}^\phi\int_{\rho_0}^\rho\partial_\rho\partial_\phi(\varphi_{k,p,q}\widehat{f})(s,\alpha)dsd\alpha.
\end{split}
\end{equation*}
On the one hand, for any choice of $(\rho_0,\phi_0)$ we have that for $\widetilde{\varphi}_{k,p,q}$ with similar support properties as $\varphi_{k,p,q}$ there holds
\begin{equation}
 \vert\partial_\rho\partial_\phi(\varphi_{k,p,q}\widehat{f})(s,\alpha)\vert \lesssim\widetilde{\varphi}_{k,p,q}\left[2^{-k}2^{-p-q}\vert\widehat{f}(s,\alpha)\vert+2^{-p-q}\vert\partial_\rho\widehat{f}(s,\alpha)\vert+2^{-k}\vert\partial_\phi\widehat{f}(s,\alpha)\vert+\vert\partial_\rho\partial_\phi\widehat{f}(s,\alpha)\vert\right].
\end{equation}
Now we note that for any $g$ there holds that
\begin{equation}
\begin{aligned}
 \abs{\int_{\phi_0}^\phi\int_{\rho_0}^\rho\widetilde{\varphi}_{k,p,q}\widehat{g}(s,\alpha)dsd\alpha}^2
 &\lesssim \iint_{\substack{\rho\in [2^k,2^{k+1}],\\\sin\phi\in[2^{p-2},2^{p+2}],\\\cos\phi\in[2^{q-2},2^{q+2}]}} d\rho d\phi\cdot 2^{-2k-p}\iint_{\substack{\rho\in [2^k,2^{k+1}],\\\sin\phi\in[2^{p-2},2^{p+2}],\\\cos\phi\in[2^{q-2},2^{q+2}]}}\vert\widehat{g}\vert^2\rho^2\sin\phi d\rho d\phi\\
 &\lesssim 2^{-k}2^{q} \norm{g}_{L^2}^2.
\end{aligned}
\end{equation}
Recalling that $S=\rho\partial_\rho$ and $\partial_\phi=\Ups$, it thus follows that
\begin{equation}
\begin{aligned}
 \abs{\int_{\phi_0}^\phi\int_{\rho_0}^\rho\partial_\rho\partial_\phi(\varphi_{k,p,q}\widehat{f})(s,\alpha)dsd\alpha}&\lesssim 2^{-\frac{3}{2}k}\Big[2^{-p-\frac{q}{2}}\norm{f}_{L^2}+2^{-p-\frac{q}{2}}\norm{Sf}_{L^2}+2^{\frac{q}{2}}\norm{\Ups f}_{L^2}+2^{\frac{q}{2}}\norm{\Ups Sf}_{L^2}\Big].
\end{aligned}
\end{equation}
We can now average over $\rho_0$ and $\phi_0$ to obtain similarly that
\begin{equation*}
\begin{split}
\left\vert \int_{\rho_0}^\rho\partial_\rho(\varphi_{k,p,q}\widehat{f})(s,\phi_0)ds\right\vert&\lesssim 2^{-k-p}\left\vert \iint_{s,\phi}\widetilde{\varphi}_{k,p,q}\left[\vert\widehat{f}\vert+\vert\rho\partial_\rho\widehat{f}\vert\right](s,\phi)dsd\phi\right\vert\\
& \lesssim 2^{-\frac{3}{2}k-p+\frac{q}{2}}\left[\Vert f\Vert_{L^2}+\norm{S\hat{f}}_{L^2}\right],\\
\end{split}
\end{equation*}
and that
\begin{equation*}
\begin{split}
\abs{\int_{\phi_0}^\phi\partial_\phi(\varphi_{k,p,q}\widehat{f})(\rho_0,\alpha)d\alpha}&\lesssim 2^{-k} \iint_{s,\phi}\widetilde{\varphi}_{k,p,q}\left[2^{-p-q}\vert \widehat{f}\vert+\vert\partial_\phi\widehat{f}\vert\right](s,\phi)dsd\phi\\
&\lesssim 2^{-\frac{3}{2}k}\left[2^{-p-\frac{q}{2}}\sum_{\vert p-p^\prime\vert+\vert q-q^\prime\vert\le 4}\Vert \widehat{P_{k,p^\prime,q^\prime}f}\Vert_{L^2}+\norm{P_{k,p^\prime,q^\prime}\Ups f}_{L^2}\right],\\
\vert \widehat{f}(\rho_0,\phi_0)\vert &\lesssim 2^{-(k+p)}\iint_{s,\phi}\vert \widehat{f}(s,\phi)\vert dsd\phi\lesssim 2^{-3k/2-p}\sum_{\vert p-p^\prime\vert+\vert q-q^\prime\vert\le 4}\Vert \widehat{P_{k,p^\prime,q^\prime}f}\Vert_{L^2}.
\end{split}
\end{equation*}

To conclude the proof it suffices to note that
\begin{equation}
 \norm{P_{k,p,q}\Ups f}_{L^2}\lesssim \sum_{\ell+p\geq 0}2^{\ell}\norm{P_{k,p}R^{(p)}_{\ell}f}_{L^2}\lesssim \sum_{\ell+p\geq 0}2^{-\beta(\ell+p)}\norm{f}_X\lesssim \norm{f}_X.
\end{equation}

\end{proof}

\subsection{An interpolation}
Here we record two interpolation inequalities that should allow us to gain some $X$ norm or decay even for ``many'' vector fields.
We define the operator
\begin{equation*}
\Vert S^{\le a}f\Vert_{L^r}:=\sup_{0\le \alpha\le a}\Vert S^\alpha f\Vert_{L^r}.
\end{equation*}
Using this, we can prove the following interpolation result:
\begin{lemma}\label{lem:interpol}
Let $r\ge 1$, $a,b\ge0$, $K\ge1$ be integers. There holds that
\begin{equation}\label{eq:vf-interpol}
\begin{split}
\Vert S^{\le a+b}f\Vert_{L^{2r}}&\lesssim_b \Vert S^{\le a}f\Vert_{L^{2r}}^\frac{1}{2}\Vert S^{\le a+2b}f\Vert_{L^{2r}}^\frac{1}{2},\\
\Vert S^{\le b}f\Vert_{L^{2r}}&\lesssim_{K,r,b} \Vert f\Vert_{L^{2r}}^{1-\frac{1}{K}}\Vert S^{\le Kb}f\Vert_{L^{2r}}^\frac{1}{K},
\end{split}
\end{equation}
uniformly in $a\ge0$ and $f$.
\end{lemma}
\begin{proof}
It suffices to assume that $a=0$. The first estimate follows from the second one for $r=1$ and $K=2$. Given $f\ne0$, and $C>\frac{1}{2}\log(2r+1)$, we claim that $g(n):=\log(\Vert S^{\le n}f\Vert_{L^{2r}})+Cn^2$ is a discrete convex function. This follows by integration by parts since
\begin{equation*}
\begin{split}
\Vert S^{n+1}f\Vert_{L^{2r}}^{2r}&=\int \rho \partial_\rho(S^nf)\cdot (S^{n+1}f)^{2r-1}\cdot\rho^2d\rho\\
&=-3\int S^nf\cdot (S^{n+1}f)^{2r-1}\cdot\rho^2d\rho-(2r-1)\int S^nf\cdot (S^{n+1}f)^{2r-2}\cdot S^{n+2}f\cdot\rho^2d\rho\\
&\le \Vert S^nf\Vert_{L^{2r}}\cdot\Vert S^{n+1}f\Vert_{L^{2r}}^{2r-2}\cdot \left[3\Vert S^{n+1}f\Vert_{L^{2r}}+(2r-1)\Vert S^{n+2}f\Vert_{L^{2r}}\right].
\end{split}
\end{equation*}
Since $\Vert S^{n+1}f\Vert_{L^{2r}}> 0$, we can divide and we deduce that
\begin{equation*}
\begin{split}
\Vert S^{\le n+1}f\Vert_{L^{2r}}^{2}&\le 2(r+1)\cdot\Vert S^{\le n}f\Vert_{L^{2r}}\Vert S^{\le n+2}f\Vert_{L^{2r}}.\\
\end{split}
\end{equation*}
The claimed inequality follows by convexity\footnote{\blue{We say that the sequence $a_n=\norm{S^{\leq n}f}_{L^{2r}}$ is convex} if and only if the piecewise linear function such that $f(n)=a_n$ is convex.}.
\end{proof}

 To apply this when $f=P_kf$ is a dispersive unknown it suffices to note that $[S,e^{it\Lambda}]=0$, so that with \eqref{eq:vf-interpol} we have
 \begin{equation}\label{eq:interpol_example}
 \begin{aligned}
  \norm{e^{it\Lambda}S^N f}_{L^{2r}}&\lesssim \norm{e^{it\Lambda}S^{N-3}f}_{L^{2r}}^{1-\frac{1}{K}}\norm{S^{\leq N+3(K-1)}f}_{L^{2r}}^{\frac{1}{K}}\\
  &\lesssim \norm{e^{it\Lambda}S^{N-3}f}_{L^{\infty}}^{(1-\frac{1}{r})(1-\frac{1}{K})}\norm{S^{N-3}f}_{L^2}^{\frac{1}{r}(1-\frac{1}{K})}\cdot \norm{S^{\leq N+3(K-1)}f}_{L^{2}}^{\frac{1}{K}}2^{k\frac{3(r-1)}{2r}\frac{1}{K}}.
 \end{aligned} 
 \end{equation}
Let us record that this gives us some decay also for the maximum number of vector fields (in the $X$ or $B$ norms) on our unknowns.
\begin{corollary}\label{cor:extrapol_decay}
 Under the bootstrap assumptions \eqref{eq:btstrap-assump}, if $f$ is a dispersive unknown of \eqref{eq:EC_disp} and the number of vector fields $M>0$ in \eqref{eq:id} is sufficiently large, then we have that for some $0<\kappa\ll\beta$ there holds
 \begin{equation}
  \norm{P_{k}e^{it\Lambda} S^b f}_{L^\infty}\lesssim 2^{\frac{3k}{2}-3k^+}t^{-1+\kappa}\eps_1,\qquad 0\leq b\leq N.
 \end{equation}
\end{corollary}
\begin{proof}
For $b\leq N-3$ the faster decay rate $t^{-1}$ follows from Proposition \ref{prop:decay}, whereas when $N-2\leq b\leq N$ this follows from \eqref{eq:interpol_example} and choosing $K\gg\kappa^{-1}$ and $r\gg 1$ sufficiently large. 
\end{proof}

\subsection{Symbol bounds}\label{sec:symbols}
In this section we give the relevant symbol estimates for the multipliers we need. We recall the notations \eqref{eq:loc_def3} and for a multiplier $m\in L^1_{loc}(\R^3\times\R^3)$ we let
\begin{equation}
\begin{aligned}
 \norm{m}_{\W_\h}&:=\sup_{k,q,\,k_i,q_i,i=1,2}\norm{\mathcal{F}(\chi_\h m)}_{L^1(\R^3\times\R^3)},\\
 \norm{m}_{\W}&:=\sup_{k,p,q,\,k_i,p_i,q_i,i=1,2}\norm{\mathcal{F}(\chi m)}_{L^1(\R^3\times\R^3)}.
\end{aligned} 
\end{equation}
We then have H\"older's inequality
\begin{equation}\label{ProdRule2}
 \norm{\Q_{m\chi_\h}(f,g)}_{L^r}\lesssim \norm{m}_{\W_\h}\norm{f}_{L^p}\norm{g}_{L^q},\quad \norm{\Q_{m\chi}(f,g)}_{L^r}\lesssim \norm{m}_{\W}\norm{f}_{L^p}\norm{g}_{L^q},\qquad \frac{1}{r}=\frac{1}{p}+\frac{1}{q},
\end{equation}
and the algebra property
\begin{equation}\label{eq:alg_prop}
 \norm{m_1\cdot m_2}_{\W}\lesssim \norm{m_1}_{\W}\norm{m_2}_{\W}.
\end{equation}
We have the following symbol bounds:
\begin{lemma}\label{lem:phasesymb_bd}
Let $\psi_>:=1-\psi$. Then
 \begin{equation}\label{eq:phasesymb_bd}
 \begin{aligned}
  \norm{\Phi^{-1}\psi_>(2^{-q_{\max}}\Phi)}_{\W}\lesssim 2^{-q_{\max}} \qquad\textnormal{and} \qquad\norm{\Phi^{-1}\psi_>(\Phi)}_{\W_\h}\lesssim 1.
 \end{aligned}
 \end{equation}
\end{lemma}
 
\begin{proof}
 The first inequality was established in \cite[Lemma A.15]{rotE}, while the second one follows from a direct adaption of that proof.
\end{proof}

\bibliographystyle{siam}
\bibliography{refs}

\end{document}